\documentclass[11pt]{amsart}
\usepackage[english]{babel}
\usepackage[T1]{fontenc}
\usepackage[utf8]{inputenc}

\usepackage{graphicx}
\usepackage{amssymb,amscd,amsthm,amsxtra}
\usepackage{latexsym}
\usepackage{epsfig}

\usepackage{fancyhdr}
\usepackage{pgf} 

\usepackage{mathtools}
\usepackage{esint}
\usepackage[active]{srcltx}
\usepackage{color}
\usepackage{hyperref}
\usepackage{stackengine}
\stackMath
\hypersetup{colorlinks,breaklinks,
            linkcolor=[rgb]{0,0,0},
            citecolor=[rgb]{0,0,0},
            urlcolor=[rgb]{0,0,0}}

\newtheorem{thm}{Theorem}[section]
\newtheorem{cor}[thm]{Corollary}
\newtheorem{lem}[thm]{Lemma}
\newtheorem{prop}[thm]{Proposition}
\theoremstyle{definition}
\newtheorem{defn}[thm]{Definition}
\theoremstyle{remark}
\newtheorem{rem}[thm]{Remark}
\numberwithin{equation}{section}
\newcommand{\be}{\begin{equation}}
\newcommand{\ee}{\end{equation}}

\newcommand{\R}{\mathbb R}
\newcommand{\N}{\mathbb N}

\newcommand{\eps}{\varepsilon}
\newcommand{\loc}{{{\tiny{\mbox{loc}}}}}

\newcommand{\p}{\partial}
\newcommand{\di}{\displaystyle}

\newcommand{\comment}[1]{}

\newcommand{\abs}[1]{\big\lvert {#1} \big\rvert}
\newcommand{\norm}[2]{\big\Vert {#1} \big\Vert_{#2}}
\newcommand{\normpic}[2]{\Vert {#1} \Vert_{#2}}

\theoremstyle{definition}
\newtheorem{teo}{Theorem}[section]

\newtheorem{theorem}[teo]{Theorem}
\newtheorem{proposition}[teo]{Proposition}

\newtheorem{remark}[teo]{Remark}

\usepackage{cfr-lm}
\usepackage[osf]{libertine}

\begin{document}

\title[Regularity of shape optimizers for some spectral fractional problems]{Regularity of shape optimizers for some spectral fractional problems}

\author[G. Tortone]{Giorgio Tortone}\thanks{}
\address {Giorgio Tortone \newline \indent
	Dipartimento di Matematica, Universit\`a di Pisa \newline \indent
	Largo B. Pontecorvo 5, 56127 Pisa - ITALY}
\email{giorgio.tortone@dm.unipi.it}

\thanks{Work partially supported by the ERC project no.\ 853404 \emph{Variational approach to the regularity of the free
boundaries - VAREG} held by Bozhidar Velichkov.}
\subjclass[2010] {
49Q10,
35R11,
47A75,
49R05
35R35, 
}
\keywords{Shape optimization, Dirichlet eigenvalues, fractional Laplacian, viscosity solution, improvement of flatness.}
\begin{abstract}
This paper is dedicated to the spectral optimization problem
$$
\mathrm{min}\left\{\lambda_1^s(\Omega)+\cdots+\lambda_m^s(\Omega) + \Lambda \mathcal{L}_n(\Omega)\colon \Omega\subset D \mbox{ s-quasi-open}\right\}
$$
where $\Lambda>0, D\subset \R^n$ is a bounded open set and $\lambda_i^s(\Omega)$ is the $i$-th eigenvalue of the fractional Laplacian on
$\Omega$ with Dirichlet boundary condition on $\R^n\setminus \Omega$.\\ We first prove that the first $m$ eigenfunctions on an optimal set are
locally H\"{o}lder continuous in the class $C^{0,s}$ and, as a consequence, that the optimal sets are open sets.
Then, via a blow-up analysis based on a Weiss type monotonicity formula, we prove that the topological boundary in $D$ of a minimizer $\Omega$ is composed of a relatively
open \emph{regular part} and a closed \emph{singular part} of Hausdorff dimension at most $n-n^*$, for some $n^*\geq 3$.
Finally we use a viscosity approach to prove $C^{1,\alpha}$-regularity of the regular part of the boundary.
\end{abstract}
\maketitle

\section{Introduction}

Functionals involving the eigenvalues of an elliptic operator and the Lebesgue measure were the object of an intense study in the last decades. The major interest of this researches was to better comprehend the interaction between the geometric features of the optimizers and their spectrum.\\

One of the most active line of research was the one related to eigenvalues of the Dirichlet Laplacian, that is
$$
\mathrm{min}\left\{F(\lambda_1(\Omega),\dots,\lambda_m(\Omega)) + \Lambda \mathcal{L}_n(\Omega)\colon \Omega\subset D \mbox{ quasi-open}\right\}
$$
with $\Lambda>0, D\subseteq \R^n$ and $\lambda_i(\Omega)$ the $i$-th eigenvalue of the Dirichlet Laplacian on $\Omega$.
Unfortunately, despite of their simple formulation, these problems turn out to be quite challenging and an explicit construction of optimal sets is known only for the cases of the first (the optimizer is the ball) and the second eigenvalue (the optimizer is the union of two equal and disjoint balls) with $D=\R^n$, thanks to the Faber-Krahn and Hong-Krahn-Szeg\"{o} inequalities.\\
Indeed, for more general functionals, the regularity of the optimal sets and of the associated eigenfunctions turns out to be a rather difficult topic, usually due to the min-max formulation of the high order eigenvalues. Starting from the seminal contribution of Buttazzo-Dal Maso \cite{Buttazzodalmaso}, and the additional existence results of \cite{Bucur2012, mazzpratt}, several results were recently obtained both in the context of nondegenerate 
\cite{brianconlamboley,BMPV,MTV1,trey1,trey2} and degenerate functionals \cite{BMPV,KL2,MTreyV} of eigenvalues.\\

Recently, many authors approached the same class of problems 
in the context of nonlocal operators, with greater attention on the fractional Laplacian (see for example \cite{bralinpar14,brascoparini,frank,KWASNICKI,parinisalort,tortonezilio}). As in the local case, the validity of a  Faber-Krahn inequality (see \cite[Theorem 3.5]{bralinpar14}) immediately implies the minimality of the first eigenvalue of the fractional Laplacian on the ball, among domains of the same Lebesgue measure. Nevertheless, already in the case of the second eigenvalue the nonlocal attitude of the fractional Laplacian
affects the minimization problem: indeed in \cite{brascoparini} the authors proved a Hong-Krahn-Szeg\"{o} type inequality which implies that the infimum of the second eigenvalue, among sets of the same Lebesgue measure, is reached by a sequence of two disjoint balls with same volume whose mutual distance tends to infinity. Focusing on this specific phenomena, in \cite{parinisalort} the authors studied the behaviour of minimizing sequences for nonlocal shape functionals wondering how the mutual positions of connected components can impact the existence of minimizers. Exploiting a nonlocal version of a concentration-compactness type principle, they prove that either an optimal shape exists or there exists a minimizing
sequence consisting of two components whose mutual distance tends to infinity.\\

In this paper we are interested in shape optimization problems depending on the first $m$ eigenfunctions of the fractional Laplacian. Precisely, for $s\in (0,1)$, we consider the following nondegenerate problem
\be\label{min.shape}
\mathrm{min}\left\{\lambda_1^s(\Omega)+\cdots+\lambda_m^s(\Omega)+\Lambda\mathcal{L}_n(\Omega)\colon \Omega\subset D \text{ $s$-quasi-open}\right\};
\ee
where $\Lambda>0, D\subset \R^n$ is a bounded open set and $\lambda_i^s(\Omega)$ denotes the $i$-th eigenvalue of the fractional Laplacian on $\Omega$ with Dirichlet boundary condition on $\R^n\setminus \Omega$, counted with the due multiplicity (see \eqref{prob.eigen} for a detailed min-max definition).\\
Comparing the functional in \eqref{min.shape} with the ones studied in the local case, if $m=1$ our result extends to the fractional setting the minimization problem for the first eigenvalue first addressed in \cite{brianconlamboley}, while for $m>1$ it can be seen as the fractional counterpart of the recent works of \cite{KL1,MTV1}.

\subsection*{Main results and organization of the paper.}
Once we summarize few notions concerning the fractional Laplacian, its spectrum and the min-max formulation of high order eingevalues, we start by proving the existence of shape optimizers of \eqref{min.shape}, for every $\Lambda>0, D\subset \R^n$ open and bounded.\\
Afterwards, in order to translate the regularity issues of the minima $\Omega$ of \eqref{min.shape} in terms of its first $m$ eigenfunctions, we prove the validity of an almost-minimality conditions for the vector $V=(v^1,\dots,v^m)$ of the first $m$ normalized eigenfunctions on $\Omega$.\\
More precisely, in Proposition \ref{alm} we prove the existence of constant $K,\eps>0$ such that
 $$
 [V]^2_{H^s(\R^n)} +\Lambda\mathcal{L}_n(\{\abs{V}>0\}) \leq \left( 1+ K \normpic{\tilde{V}-V}{L^1(\R^n)}\right )[\tilde{V}]^2_{H^s(\R^n)} +\Lambda\mathcal{L}_n(\{|\tilde{V}|>0\})
  $$
  for every $\tilde{V} \in H^{s}_0(D;\R^m)$ such that $\normpic{\tilde{V}}{L^\infty(\R^n)}\leq \normpic{V}{L^\infty(\R^n)} $ and $\normpic{V-\tilde{V}}{L^1(\R^n)}\leq \eps$.\\ Nevertheless, the presence of the Gagliardo–Slobodecki\u{\i} seminorm makes unworkable the application of some basic tool such as the harmonic replacement and the blow-up analysis as well. Thus, in the same spirit of \cite{DR,DS2,DS1,DS3,DSS}, we overcome this difficulty by considering the local realization of the fractional Laplacian in terms of the Caffarelli-Silvestre extension \cite{CS2007} (see Section \ref{general} for more details).\\Indeed, given the vector $G=(g^1,\dots,g^m) \in H^{1,a}_\Omega(\R^{n+1};\R^m)$, with $g^i$ the extension in $\R^{n+1}$ of the $i$-th normalized eigenfunctions on $\Omega$, in Theorem \ref{definition} we introduce an \emph{almost-minimality condition} in terms of the functional
 $$
\mathcal{J}(G) = \int_{\R^{n+1}}|y|^a |\nabla G|^2\mathrm{d}X +\tilde{\Lambda}\mathcal{L}_n(\{\abs{G}>0\}\cap \R^n) \quad\mbox{with }\tilde{\Lambda}=\frac{2\Lambda}{d_s},
$$
that is, there exist $K,\eps>0$ such that
$$
  \mathcal{J}(G) \leq \left( 1+ K \normpic{\tilde{G}-G}{L^1(\R^n)}\right )  \mathcal{J}(\tilde{G})
$$
  for every $\normpic{G-\tilde{G}}{L^1(\R^n)}\leq \eps$ (see Section \ref{general} for the precise formulation). Moreover, exploiting the localization of the energies in $\mathcal{J}$, we prove a \emph{localized almost-minimality condition}, namely there exist $\sigma>0,r_0>0$ such that
    $$
    \mathcal{J}(G,B_r(X_0)) \leq \left( 1+ \sigma r^n\right )   \mathcal{J}(\tilde{G},B_r(X_0)) + \sigma \frac{2}{d_s}\sum_{i=1}^{m} \lambda_i^s(\Omega)r^n,
    $$
for every $X_0 \in \R^n,r\in (0,r_0]$ and $G-\tilde{G}\in H^{1,a}_0(B_r(X_0);\R^m)$ (see also \eqref{almost.fin.new} and \eqref{almost.local.new} for two alternative almost-minimality conditions). By using this new formulations, we can prove the following regularity result for the first $m$ normalized eigenfunctions.
\begin{thm}\label{mmm0}
Let $D\subset \R^n$ be an open bounded set and $\Lambda >0$. Then, the shape optimization problem \eqref{min.shape} admits a solution $\Omega$. Moreover, the vector $G \in H^{1,a}(\R^{n+1};\R^m)$ of the first $m$ normalized eigenfunctions is $C^{0,s}$-H\"{o}lder continuous in $\R^{n+1}$. Thus, the shape optimizer $\Omega$ is an open set and
$$
\mathcal{L}_n(\Omega)=\mathcal{L}_n(\{v^i\neq 0 \})=\mathcal{L}_n(\{g^i\neq 0 \}\cap \R^n)
$$
for every $i=1,\dots,m$.
\end{thm}
The first part of the Theorem reflects the regularity result obtained in the local case in \cite{BMPV}, where the authors proved Lipschitz continuity of eigenfunctions on optimal domains of \eqref{min.shape} for $s=1$.\\
The novelty of the second part of the Theorem lies in the validity of the unique continuation principle for eigenfunctions inside the optimal domains $\Omega$ even if disconnected. This is a purely nonlocal feature: indeed it is not possible to prove that the optimizer is connected, as in the local case \cite[Corollary 4.3]{MTV1}, but on the other hand the support of the first $m$ normalized eigenfunction coincides with $\Omega$, up to a $(n-1)$-dimensional set.\\

Afterwards, with the aim to rely on a free boundary approach to the regularity of the shape optimizer, we denote through the paper
 \be\label{nota}
\{|G|>0\}\cap \R^n=\Omega\quad\mbox{and}\quad F(G) = \partial\Omega\cap D,
\ee
in such a way as to express the regularity issues of the thin domain $\Omega \subset \R^n$ in terms of the ``extended'' vector $G$. Thus, suitably
modifying many ideas of the regularity theory with thin free boundary \cite{DR, DSalmost,DSS,DT2}, we finally prove the following result on the topological boundary of $\Omega$.
\begin{thm}\label{mmm}
Let $\Omega$ be a shape optimizer of \eqref{min.shape} for some $\Lambda>0, D\subset \R^n$ open and bounded. Then, $\Omega$ has locally finite perimeter in $D$ and its topological boundary $\partial\Omega\cap D$ in $D$ is a disjoint union of a regular part $\mathrm{Reg}(\partial\Omega\cap D)$ and one-phase singular set $\mathrm{Sing}(\partial\Omega\cap D)$:
\begin{enumerate}
  \item[1.] $\mathrm{Reg}(\partial\Omega\cap D)$ is an open subset of $\partial\Omega$ and is locally the graph of a $C^{1,\alpha}$ function.
\item[2.] $\mathrm{Sing}(\partial\Omega\cap D)$ consists only of points in which the Lebesgue density of $\Omega$ is strictly between $1/2$ and $1$. Moreover, there is $n^* \geq 3$ such that:
    \begin{itemize}
      \item if $n < n^*$, $\mathrm{Sing}(\partial\Omega\cap D)$ is empty;
\item if $n = n^*$, $\mathrm{Sing}(\partial\Omega\cap D)$ contains at most a finite number of isolated points;
\item if $n > n^*$, the $(n - n^*)$-dimensional Hausdorff measure of $\mathrm{Sing}(\partial\Omega\cap D)$ is locally finite in $D$.
    \end{itemize}
\end{enumerate}
\end{thm}
Notice that the presence of a dimensional threshold $n^*\geq 3$ (see the precise definition \eqref{n*}) is deeply related to the existence of $s$-homogeneous global minimizer of the thin one-phase problem (see \cite{DS1, EKPSS}).\\
The analysis of the free boundary $F(G)=\partial \Omega\cap D$ is carried out by combining a classification of blow-up limits via a Weiss type monotonicity formula and a viscosity formulation of the problem near free boundary points.\\

More precisely, we prove that the vector of the first $m$ normalized eigenfunctions $G=(g^1,\dots, g^m)$ is a viscosity solutions (see Section \ref{visco} for the precise definition) of the problem
$$ \begin{cases}
-L_a G =0 & \text{in $\R^{n+1}\setminus \{y=0\}$;}\\
-\partial^a_y G = \lambda G & \text{in $\{|G|>0\}\cap \R^n$;}\\
 \frac{\partial}{\partial t^{s}} |G|= \frac{\sqrt{\Lambda}}{\Gamma(1+s)} & \text{on $F(G)$}
\end{cases}
\quad\text{with $\lambda=(\lambda_1^s(\Omega),\dots,\lambda_m^s(\Omega))$,}
$$
where
$$ \frac{\partial }{\partial t^{s}}|G|(X_0):=\di\lim_{t \rightarrow 0^+} \frac{|G|(x_0+t\nu(x_0),0)} {t^{s}} , \quad \textrm{$X_0=(x_0,0) \in F(G)$},
$$
with $\nu(x_0)$ the unit normal to $F(G)$ at $X_0$ pointing toward $\{|G|>0\}\cap \R^n$.\\
Thus, the analysis of the regular part of the free boundary is addressed by slighlty modifying the viscosity methods developed in \cite{DSalmost, DT2}. More precisely, we prove an Harnack type inequality strictly relying on the behaviour of the first eigenfunction $g^1$ and of the norm $|G|$ near free boundary points. Lastly, the regularity of $\mathrm{Reg}(\partial \Omega)$ is achieved by proving $C^{1,\alpha}$ regularity result for flat free boundaries for a general class of linear problems (see Subsection \ref{general.linear} for more details).
\begin{remark}
The main difference with the results of the local analogue \cite{MTV1}, besides the issue of connection and unique continuation, lies in the density estimates on $\Omega$. As pointed out in Remark \ref{upper}, our proof only relies on the different local regularity of fractional eigenfunctions near their zero set depending on whether or not they change sign. This peculiarity has been already observed in other similar problem, and it resembles the same differences that appear between the classical two-phase problem of Alt-Caffarelli-Friedman \cite{ACF} and its fractional counterpart \cite{AP} (see also \cite[Theorem 1.1]{DT} where the authors described a similar phenomenon in the thin-counterpart of \cite{MTV2}).\\
Therefore, we think that this feature can be exploited in the analysis of branching points in shape optimization problems involving degenerate functional of fractional eigenvalues, that is
$$
\mathrm{min}\left\{\lambda_{k_1}^s(\Omega)+\cdots+\lambda_{k_m}^s(\Omega)+\Lambda\mathcal{L}_n(\Omega)\colon \Omega\subset D \text{ $s$-quasi-open}\right\}\quad\mbox{with }1<k_1\leq \cdots \leq k_m,
$$
where the
differences with the local case could be more evident.
\end{remark}
The paper is organized as follows:
in Section \ref{general} we prove the existence of shape optimizer for \eqref{min.shape} and the almost-minimality of the eigenfunctions for a more general
free boundary problem. At the end of the Section, exploiting the Caffarelli-Silvestre extension, we introduce a more useful formulation of these almost-minimality conditions.
In Section \ref{local} we study the local behavior of eigenfunctions answering the classical questions of optimal regularity, non-degeneracy and density estimates for local minimizers. Moreover, we prove the validity of a unique continuation result for the eigenfunctions of optimal domain. Then, in Section \ref{weiss.sect} we derive a Weiss-type formula which will allow to characterize the blow-up limits in Section \ref{blow}. Finally, the blow-up analysis of Section \ref{blow} will lead to the definition of $\mathrm{Reg}(F(G))$ and $\mathrm{Sing}(F(G))$ and to some classical estimates of the Hausdorff dimension of the singular set. Finally, in Section \ref{visco} we introduce the notion of viscosity solution near $F(G)$ and in Section \ref{final} we use a viscosity approach to obtain $C^{1,\alpha}$ smoothness of the regular part $\mathrm{Reg}(F(G))$. More precisely, we develop an Harnack type inequality for flat solutions, which will be the basic tool for an improvement of flatness result, from which the $C^{1,\alpha}$ regularity of a flat free boundary follows by classic arguments.

\subsection*{Acknowledgements.} We are grateful to Dario Mazzoleni for useful discussions and suggestions.

\subsection*{Notations.} Let $s \in (0,1)$, we say that $\Omega$ is a $s$-quasi-open set if there exists a decreasing sequence $\{\Omega_k\}_{k\in \N}$ of open subsets of $\R^n$ such that $\Omega\cup \Omega_k$ is open and $\mathrm{cap}_s(\Omega_k) \to 0^+$ as $k\to \infty$. We refer to \cite{BRS, parinisalort} for more details on $s$-capacity and $s$-quasi-open set.\\
Let $\Omega\subset \R^n$ be bounded and open, we introduce the space $H^s_0(\Omega)$  as the completion of $C^\infty_c(\Omega)$ with respect to the
norm
$$
\norm{u}{H^s(\R^n)}= \left(\int_{\R^n}{u^2\mathrm{d}x} + [u]^2_{H^s(\R^n)} \right)^{1/2}$$
where
\be\label{gag}
[u]_{H^s(\R^n)}^2=\frac{C(n,s)}{2} \int_{\R^n}\int_{\R^n}{\frac{\abs{u(x)-u(z)}^2}{\abs{x-z}^{n+2s}}\mathrm{d}x\mathrm{d}z} \quad\text{with }
C(n,s) = \frac{2^{2s}s\Gamma(\frac{n}{2}+s)}{\pi^{n/2}\Gamma(1-s)}.
\ee
More generally, for a bounded $s$-quasi-open set $\Omega \subset \R^n$  we denote with $H^s_0(\Omega)$ the Sobolev space defined as the set of functions in $H^s(\R^n)$ which, up to a set of zero $s$-capacity, vanish outside of $\Omega$.\\
From now on, we denote by $\{e_i\}_{i=1,\ldots,n}$ and $\{f^i\}_{i=1,\ldots,m}$ canonical basis in $\R^n$ and $\R^m$ respectively. Unit directions in $\R^n$ and $\R^m$ will be typically denoted by $e$ and $f$. The Euclidean norm in either space is denoted by $|\cdot|$ while the dot product is denoted by $\langle \cdot, \cdot \rangle$.\\
A point $X \in \R^{n+1}$ will be denoted by $X=(x,y)\in \R^n\times \R$ and we will use the notation $x=(x',x_n)$. Moreover, a ball in $\R^{n+1}$ with radius $r>0$ centered at $X$ is denoted by $B_r(X)$ and for simplicity $B_r=B_r(0)$. Also, we use $\mathcal{B}_r = B_r\cap \{y=0\}$ to denote the $n$-dimensional ball in $\R^n\times \{y=0\}$.\\ We will often consider the following sets: let $g$ be a continuous non-negative function in $B_r$, then
\begin{align*}
B^+_r(g) &:= B_r \setminus \{(x,0)\colon g(x,0)=0\} \subset \R^{n+1}\\
\mathcal{B}^+_r(g) &:= B^+_r(g) \cap \{y=0\}\subset \R^{n}.
\end{align*} By abuse of notation, if $G=(g^1,\ldots,g^m)$ is a vector valued continuous function, we use $B_r^+(G), \mathcal B_r^+(G)$ in place of $B^+_r(|G|), \mathcal B^+_r(|G|)$ respectively.
Also, we will denote with $P$ and $L$ respectively the half-hyperplane $P:= \{X \in \R^{n+1}\colon x_n\leq 0, y =0 \}$ and $L:= \{X \in \R^{n+1}\colon x_n= 0, y =0 \}$.

In regard to the thin one-phase problem \eqref{FB}, we remark that if $F(g)$ is $C^2$ then any function $g$ which is harmonic in $B^+_1(g)$ has an asymptotic expansion at a point $ x_0\in F(g),$
$$g(x,y) = \alpha U(\langle x-x_0 ,\nu(x_0) \rangle , y)  + o(|x-x_0|^{s}+|y|^{s}).$$
Here $U(t,z)$ is the prototype of regular one-plane solution of the thin one-phase problem \eqref{FB}, which is given by
\begin{equation}\label{Unew}U(t,z) = \left(\frac{\sqrt{t^2+z^2}+t}{2}\right)^s= r^{s}\cos^{2s}\left(\frac \theta 2\right), \end{equation}
where in the second formulation we used the polar coordinates $$t= r\cos\theta, \quad  z=r\sin\theta, \quad r\geq 0, \quad -\pi \leq  \theta \leq \pi.$$

\section{Properties of the eigenfunctions on the optimal sets}\label{general}
In this Section we start with some general properties about the spectrum of the fractional Laplacian and we then prove existence of shape optimizers for the problem \eqref{min.shape} among the class of $s$-quasi-open sets. Afterwards, we introduce an almost-minimality condition satisfied by the eigenfunctions associated to shape optimizer, in terms of their Gagliardo–Slobodecki\u{\i} seminorm.\\ Since this condition is not convenient for our analysis, we conclude the Section with a reformulation of the almost-minimality by exploiting the local realization of the fractional Laplacian with the extension techinique.\\

Let $\Omega\subset \R^n$ be a bounded open set and $s\in (0,1)$, consider the Dirichlet-type eigenvalue problem
\be\label{equation.eigen}
\begin{cases}
  (-\Delta)^s v=\lambda v & \mbox{in } \Omega \\
  v=0 & \mbox{in }\R^n \setminus \Omega.
\end{cases}
\ee
where the fractional Laplacian is defined by
$$
(-\Delta)^s v(x)=  C(n,s) \mbox{ P.V.}\int_{\R^n}{\frac{v(x)-v(z)}{\abs{x-z}^{n+2s}}\mathrm{d}z} \quad\text{with }
C(n,s) = \frac{2^{2s}s\Gamma(\frac{n}{2}+s)}{\pi^{n/2}\Gamma(1-s)}.
$$
In analogy with the local case, if $\lambda \in \R$ is such that  \eqref{equation.eigen} admits a nontrivial solution $v$, then we say that $\lambda$ is an eigenvalue of $\Omega$ with $v$ an associated eigenfunction (see \cite{bralinpar14, brascoparini, LL} and reference therein for more result and generalisation of the concept of fractional eigenvalues).\\
Moreover, the spectrum of the Dirichlet fractional Laplacian is a closed set of strictly positive real numbers defined as an increasing sequence
$$
0<\lambda_1^s(\Omega)< \lambda_2^s(\Omega)\leq \lambda_3^s(\Omega) \leq \cdots \leq \lambda_m^s(\Omega)\leq \cdots
$$
where the strict inequality $\lambda_1^s(\Omega)< \lambda_2^s(\Omega)$ is a consequence of the nonlocal attitude of the fractional Laplacian (see \cite[Theorem 2.8]{brascoparini} and Subsection \ref{sub.unique}).
Thus, from now on we will denote by $v^i$ the eigenfunction corresponding to the eigenvalue $\lambda_i^s(\Omega)$, normalized with respect to $L^2(\Omega)$-norm. We recall that the family of eigenfunctions $\{v^i\}_{i\in \N}$ form a complete orthonormal system in $L^2(\Omega)$ and the Dirichlet eigenvalues $\lambda_i^s(\Omega)$ of the fractional Laplacian on $\Omega$ can be variationally characterized by the following min-max formulation
\begin{align}\label{prob.eigen}
\begin{aligned}
\lambda_1^s(\Omega)&=\min\left\{\frac{[u]^2_{H^s(\R^N)}}{\|u\|^2_{L^2(\R^N)}} : u\in H^s_0(\Omega)\setminus \{0\}\right\},\\
\lambda_i^s(\Omega)&=\min_{E_i\subset H^s_0(\Omega)} \max_{u\in E_i\setminus \{0\}}\frac{[u]^2_{H^s(\R^N)}}{\|u\|^2_{L^2(\R^N)}},\quad\text{if $i\geq 2$},
\end{aligned}
\end{align}
where in the second case the infimum is taken over all $i$-dimensional linear subspaces $E_i$ of $H^s_0(\Omega)$.\\ As pointed out in \cite[Theorem 3.1]{brascoparini}(see also \cite[Theorem 3.1]{bralinpar14}), the supremum of an eigenfunction $v^i$ on a set $\Omega$ can be estimated in terms of a power of the corresponding eigenvalue
\be\label{Linfinity.u}
\norm{v^i}{L^\infty(\R^n)} \leq \left[\tilde{C}_{n,s} \lambda_i^s(\Omega)\right]^{\frac{n}{4s}}, \quad\mbox{where }\tilde{C}_{n,s} =T_{2,s}\left(\frac{n}{n-2s}\right)^{\frac{n-2s}{2s}}
\ee
with $T_{2,s}$ the sharp Sobolev constant.\\
As a consequence of the min-max formulation, we have the following variational formulation for the sum of the first $m$ Dirichlet eigenfunctions on a $s$-quasi-open set $\Omega$
\be\label{var}
\sum_{i=1}^m \lambda_i^s(\Omega) = \text{min}\left\{\sum_{i=1}^m [u^i]^2_{H^s(\R^n)} \colon U=(u^1,\dots,u^m)\in H^s_0(\Omega;\R^m), \int_\Omega u^i u^j\mathrm{d}x=\delta_{ij}\right\},
\ee
where $V=(v^1,\dots,v^m)$ is the vector of the first $m$-eigenfunctions normalized with respect to the $L^2$-norm. Finally, from this last formulation we can deduce the existence of a minimum for \eqref{min.shape}.
\begin{prop}\label{existence}
Let $\Lambda>0$ and $D\subset \R^n$ be a bounded open set. Then the shape optimization problem
  $$
\mathrm{min}\Bigg\{\sum_{i=1}^m \lambda_i^s(\tilde{\Omega})+\Lambda\mathcal{L}_n(\tilde{\Omega})\colon \tilde{\Omega}\subset D \mbox{ s-quasi-open}\Bigg\}
$$
admits a solution.
\end{prop}
\begin{proof}
  Let $(\Omega_k)_k$ be a minimizing sequence of $s$-quasi-open sets to the problem \eqref{min.shape}.\\Set $V_k=(v_{k}^1,\dots,v_{k}^m)$ as the vector of the first $m$ eigenfunctions on $\Omega_k$, then
  $$
  \left[ V_k \right]^2_{H^s(\R^n)}=
\sum_{i=1}^m [v_{k}^i]^2_{H^s(\R^n)} = \sum_{i=1}^m \lambda_i^s(\Omega_n)
  $$
  which implies that $V_k$ is uniformly bounded in $H^{s}(\R^n;\R^m)$. Thus, up to a subsequence, $(V_k)_k$ weakly converges in $H^{s}_0(D;\R^m)$ and strongly in $L^2(D;\R^m)$ to some $V_\infty \in H^s_0(D;\R^m)$.\\
  Since $V_\infty$ is an orthonormal vector, if we set $\Omega_\infty = \{|V_\infty|>0\}$, by the weak convergence in $H^s$ and the lower semi-continuity of the Lebesgue measure, we get
  \begin{align*}
  \sum_{i=1}^m \lambda_i^s(\Omega_\infty) + \Lambda|\Omega_\infty| &\leq
  \sum_{i=1}^m [v_{\infty}^i]^2_{H^s(\R^n)} + \Lambda|\Omega_\infty|\\
  &\leq \liminf_{k\to \infty} \left( \sum_{i=1}^m [v_{k}^i]^2_{H^s(\R^n)} + \Lambda|\Omega_k|\right)\\
  &\leq \liminf_{k\to \infty} \left( \sum_{i=1}^m \lambda_i^s(\Omega_k) + \Lambda|\Omega_k| \right),
  \end{align*}
  as we claimed.
  \end{proof}
Hence, in view of the variational characterization \eqref{var} of the sum of the first $m$ Dirichlet eigenvalues, given $\Omega$ a shape optimizer of \eqref{min.shape}, the vector of the first $m$ normalized eigenfunctions $V=(v^1,\dots,v^m)$ on $\Omega$ is a solution to the problem
\be\label{con.condiz}
\text{min}\left\{ [V]^2_{H^s(\R^n)} +\Lambda\mathcal{L}_n(\{\abs{V}>0\})\colon V\in H^{s}_0(D;\R^m), \int_D v^i v^j\mathrm{d}x=\delta_{ij}\right\},
\ee
where $\Omega = \{|V|>0\}$. Moreover, by \eqref{Linfinity.u} we get
\be\label{Linfinity.G}
\norm{|V|}{L^\infty(\R^n)} \leq \sqrt{m}\left[\tilde{C}_{n,s} \lambda^s_m(\Omega)\right]^{\frac{n}{4s}}, \quad\mbox{where }\tilde{C}_{n,s} =T_{2,s}\left(\frac{n}{n-2s}\right)^{\frac{n-2s}{2s}}
\ee
with $T_{2,s}$ the sharp Sobolev constant.
In the same spirit of \cite{MTV1,trey1}, by using the Gram-Schmidt
procedure we remove the orthogonality constraint.
\begin{lem}\label{ort}
  Let $\Omega\subset D$ be a $s$-quasi-open set and $V=(v^1,\dots,v^m)$ be the vector of the first $m$ normalized eigenfunctions on $\Omega$. Let $\delta>0$ be fixed and set $\tilde{V}=(\tilde{v}^1,\dots,\tilde{v}^m) \in H^{s}_0(D;\R^m)$ be such that
  $$
  \eps_m := \sum_{i=1}^m \norm{\tilde{v}^i-v^i}{L^1(D)}\leq 1\quad\text{and}\quad\sup_{i=1,\dots,m}\left\{\norm{v^i}{L^\infty(D)}+\norm{\tilde{v}^i}{L^\infty(D)}\right\}\leq \delta.
  $$
  Define $W=(w^1,\dots,w^m)\in H^{s}_0(D;\R^m)$ as the vector obtained orthonormalizing $\tilde{V}$ by the Gram-Schimdt procedure associated to the norm $L^2(\R^n)$, that is
$$
w^i=\frac{\tilde{w}^i}{\norm{\tilde{w}^i}{L^2}}\quad\mbox{with}\quad
 \tilde{w}^i=
  \begin{cases}
    \tilde{v}^1 & \mbox{if } i=1, \\
    \tilde{v}^i - \sum_{j=1}^{i-1}\left(\int_{\R^n} \tilde{v}^i w^j \mathrm{d}x\right) w^j & \mbox{if }j>1
  \end{cases}.
$$
  Then, there exist constants $1\geq \overline{\eps}_m>0$ and $\overline{C}_m>0$, depending only on $n,m,\delta$ and $\Omega$, such that the estimate
  $$
[W]^2_{H^s(\R^n)} \leq \left( 1+ \overline{C}_m \eps_m\right) [\tilde{V}]^2_{H^s(\R^n)}
  $$
  holds for every $\tilde{V}$ as above, with $\eps\leq \overline{\eps}_m$.
\end{lem}
\begin{proof}
In terms of the components, we aim to prove
$$
\sum_{i=1}^m \big[ w^i \big]^2_{H^s(\R^n)} \leq \left( 1+ \overline{C}_m \eps_k\right) \sum_{i=1}^m \big[\tilde{v}^i\big]^2_{H^s(\R^n)}.
$$
Hence, the first step is to prove that there exists $\overline{\eps}_m,\overline{C}_m>0$ such that the following estimates holds whenever $\eps_m\leq \overline{\eps}_m$
  $$
  \sum_{i=1}^m \norm{w^i-v^i}{L^1(\mathcal{B}_r)} \leq C_m\eps_m,\quad
  \sup_{i=1,\dots,m}\norm{w^i}{L^\infty}\leq C_m
$$
 where $\overline{\eps}_m,C_m$ depend only on $d,m,\delta$ and $\Omega$. Since it follows the same lines of the proof of  \cite[Lemma 2.5]{MTV1}, we just sketch the proof. In particular we get
  \begin{align*}
  \norm{\tilde{v}^1-w^1}{L^1} &\leq\frac{|\norm{\tilde{v}^1}{L^1}^2-1|}{\norm{\tilde{v}^1-v^1+v^1}{L^1}^2}\norm{\tilde{v}^1}{L^1}
\\
&\leq\frac{\norm{\tilde{v}^1-v^1}{L^1}^2 +2\norm{v^1}{L^\infty}\norm{\tilde{v}^1-v^1}{L^1}}{1-2\norm{u}{L^\infty}\norm{\tilde{v}^1-v^1}{L^1}}(\norm{v^1}{L^1} + \norm{\tilde{v}^1 -v^1}{L^1})\\
&\leq\eps_1\frac{3\delta}{1-2\delta\eps_1}(\mathcal{L}_n(\Omega)^{1/2} + \eps_1) \leq 12\delta\mathcal{L}_n(\Omega)^{1/2}\eps_1,
  \end{align*}
  by choosing $\eps_1\leq \min\{\mathcal{L}_n(\Omega)^{1/2},(2\delta)^{-1}\}$. Thus, it follows that $\norm{v^1-w^1}{L^1}\leq (1+12\delta\mathcal{L}_n(\Omega)^{1/2})\eps_1$. Similarly, we get
  \be\label{25}
  \norm{w^1}{L^\infty}\leq \frac{\norm{\tilde{v}^1}{L^\infty}}{1-2\norm{v^1}{L^\infty}\norm{\tilde{v}^1-v^1}{L^1}}\leq \frac{\delta}{1-2\delta e_1}\leq 2\delta,
  \ee
which proves our claim for $m=1$. By induction, suppose now that the claim holds for $1,\dots,m-1$ and let us first estimate the distance from $v^m$ to the orthogonalized function
  $$
  \tilde{w}^m:=
  \begin{cases}
    \tilde{v}^1 & \mbox{if } m=1, \\
    \tilde{v}^m - \sum_{i=1}^{m-1}\left(\int_{\R^n} \tilde{v}^m w^i \mathrm{d}x\right) w^i & \mbox{if }k>1
  \end{cases}
  $$
  with respect to the $L^1$-topology. By the inductive step,
\begin{align*}
\sum_{i=1}^{m-1}\abs{\int_{\R^n}\tilde{v}^m w^i\mathrm{d}x} &\leq \sum_{i=1}^{m-1}\int_{\R^n}|\tilde{v}^m-v^m||w^i-v^i| + |v^i||\tilde{v}^m-v^m| + |v^m||w^i-v^i| \mathrm{d}x\\
&\leq C_{m-1}\eps_{m-1}\delta + (m-1)\delta \eps_m + \delta C_{m-1}\eps_{m-1}\\
&\leq (2C_{m-1}\delta + (m-1)\delta)\eps_m,
\end{align*}
which implies
  \begin{align}\label{2.6}
  \begin{aligned}
  \norm{v^m-\tilde{w}^m}{L^1} &\leq \norm{v^m-\tilde{v}^m}{L^1}+\sum_{i=1}^{m-1}\abs{\int_{\R^n}\tilde{v}^m w^1\mathrm{d}x}\left(\norm{v^i}{L^1}+\norm{w^i-v^i}{L^1}\right)\\
  &\leq \left[1+(\mathcal{L}_n(\Omega)^{1/2}+\bar{\eps}_{m-1})(2C_{m-1}\delta + (m-1)\delta)\right] \eps_m
  \end{aligned}
  \end{align}
  and
  \be\label{2.7}
  \norm{\tilde{w}^m}{L^\infty} \leq \delta \left( 1+ C_{m-1}(2C_{m-1}+m-1)\right).
  \ee
  By setting for simplicity $\tilde{C}_m$ as the largest of the constants appearing on the right hand side of \eqref{2.6} and \eqref{2.7}, we get
  $$
  \norm{v^m-\tilde{w}^m}{L^1}\leq \tilde{C}_m \eps_m \quad\text{and}\quad\norm{\tilde{w}^m}{L^\infty}\leq \tilde{C}_m.
  $$
  Recalling that $w^m=\norm{\tilde{w}^m}{L^2}^{-1}\tilde{w}^m$, we have
  \begin{align}\label{28}
  \begin{aligned}
  \abs{\norm{\tilde{w}^m}{L^2}-1}&\leq \abs{2\int_{\R^n}v^m(v^m-\tilde{w}^m)\mathrm{d}x + \int_{\R^n}(v^m-\tilde{w}^m)^2\mathrm{d}x}\\
  &\leq 2\delta \tilde{C}_m\eps_m + (\delta + \tilde{C}_m)\tilde{C}_m \eps_m.
  \end{aligned}
  \end{align}
  Hence, assume that $\eps_m \leq \overline{\eps}_m :=\frac12(2\delta \tilde{C}_m + (\delta+\tilde{C}_m)\tilde{C}_m)^{-1}$. Thus, $1/2 \leq \norm{\tilde{w}^m}{L^2}\leq 3/2$ and we have
  $$
  \norm{w^m}{L^\infty} = \norm{\tilde{w}^m}{L^2}^{-1}\norm{\tilde{w}^m}{L^\infty} \leq 2 \tilde{C}_m.
  $$
  On the other hand, as in \eqref{2.6}, we get
  $$
  \norm{v^m-w^m}{L^1} \leq \norm{v^m -\tilde{w}^m}{L^1}+ \norm{\tilde{w}^m-w^m}{L^1}\leq (1+12\tilde{C}_m\mathcal{L}_n(\Omega)^{1/2})\tilde{C}_m\eps_m,
  $$
  for $\eps_m \leq \overline{\eps}_m$, for some $\overline{\eps}_m>0$ small enough which depends on $\tilde{C}_m,\delta$ and $\mathcal{L}_n(\Omega)$. The inductive claim follows by defining
  $$
  C_m := 2(1+12\tilde{C}_m\mathcal{L}_n(\Omega)^{1/2})\tilde{C}_m.
  $$
  Thus, we can finally prove the main inequality by induction. First, by \eqref{25} for $m=1$
  $$
  [w^1]_{H^s(\R^n)} = \frac{[\tilde{v}^1]_{H^s(\R^n)}}{\norm{\tilde{v}^1}{L^2}}\leq
  \frac{[\tilde{v}^1]_{H^s(\R^n)}}{1-2\norm{v^1}{L^\infty}\norm{v^1-\tilde{v}^1}{L^1}} \leq (1+4\delta \eps_1) [\tilde{v}^1]_{H^s(\R^n)},
  $$
  while for $m>1$ we obtain
  \begin{align*}
  [w^m]_{H^s(\R^n)} &= \frac{[\tilde{w}^m]_{H^s(\R^n)}}{\norm{\tilde{w}^m}{L^2}}\\
  &\leq
  \frac{1}{1-|\norm{\tilde{w}^m}{L^2}-1|}\left[\tilde{v}^m - \sum_{i=1}^{m-1}\left(\int_{\R^n} \tilde{v}^m w^i \mathrm{d}x\right) w^i\right]_{H^s(\R^n)}\\
  &\leq  (1+2(2\delta\tilde{C}_m+(\delta+\tilde{C}_m)\tilde{C}_m)\eps_m)\left([\tilde{v}^m]_{H^s(\R^n)} +\sum_{i=1}^{m-1}\left(\int_{\R^n} \tilde{v}^m w^i \mathrm{d}x\right)[w^i]_{H^s(\R^n)}\right).
  \end{align*}
where we used \eqref{28}. Using again
$$
\sum_{i=1}^{m-1}\left\lvert\int_{\R^n} \tilde{v}^m w^i \mathrm{d}x\right\lvert \leq ((m-1)\delta + 2C_{m-1}\delta) \eps_m,
$$
and the induction hypothesis, we get the claimed result.
\end{proof}
We are now ready to prove the validity of an almost-minimality condition for shape optimizer of \eqref{min.shape}.
\begin{proposition}\label{alm}
  Suppose that $\Omega\subset D$ is a solution to \eqref{min.shape} for some $\Lambda >0$. Then, the vector $V=(v^1,\dots,v^m) \in H^{s}_0(\Omega;\R^m) $ of  normalized eigenfunctions on $\Omega$ satisfies the following almost-minimality condition: for every $\delta\geq \norm{V}{L^\infty(\Omega)}$ there exist $K,\eps>0$ such that
  \be\label{alm.min}
  [V]^2_{H^s(\R^n)} +\Lambda\mathcal{L}_n(\{\abs{V}>0\}) \leq \left( 1+ K \normpic{\tilde{V}-V}{L^1(\R^n)}\right )[\tilde{V}]^2_{H^s(\R^n)} +\Lambda\mathcal{L}_n(\{|\tilde{V}|>0\})
  \ee
  for every $\tilde{V} \in H^{s}_0(D;\R^m)$ such that $\normpic{\tilde{V}}{L^\infty(\R^n)}\leq \delta $ and $\normpic{V-\tilde{V}}{L^1(\R^n)}\leq \eps$.
\end{proposition}
\begin{proof}
Let $\tilde{V}\in H^s_0(D; \R^m)$ be the vector-valued function satisfying the assumptions of Proposition \ref{alm} and let $W \in H^s_0(D;\R^m)$ be the function obtained by the orthonormalizing procedure of Lemma \ref{ort} applied on $\tilde{V}$. Since $V$ is a solution of the minimization problem \eqref{con.condiz} and $W$ satisfies
$$
\int_{\R^n} w^i w^j\mathrm{d}x=\delta_{ij},
$$
we can use it as a test function in \eqref{con.condiz} obtaining
\begin{align*}
[V]^2_{H^s(\R^n)} +\Lambda\mathcal{L}_n(\{\abs{V}>0\}) & \leq [W]^2_{H^s(\R^n)} +\Lambda\mathcal{L}_n(\{\abs{W}>0\})\\
& \leq (1+\overline{C}_k \normpic{V- \tilde{V}}{L^1(\R^n)})[\tilde{V}]^2_{H^s(\R^n)} +\Lambda\mathcal{L}_n(\{|\tilde{V}|>0\}),
\end{align*}
where the last inequality follows by Lemma \ref{ort} and the fact that by construction $\mathcal{L}_n(\{\abs{W}>0\}) \subset \mathcal{L}_n(\{|\tilde{V}|>0\})$.
\end{proof}
\subsection{The extended formulation in $\R^{n+1}$}
The main problem of the previous formulation is the lack of monotonicity result for $s$-harmonic function, which makes unworkable the strategy of the harmonic replacement. Hence, we overcome this difficulty by considering the local realization of the fractional Laplacian in terms of the Caffarelli-Silvestre extension \cite{CS2007}, that gives an equivalent formulation of the previous non-local
variational problem as a local problem defined in one extra dimension. This is nowadays a common strategy and it has been deeply used in the context of fractional problem with free boundary (see \cite{CRS,DR,DS2,DS1,DS3,DSS,EKPSS} among others).\\
Thus, we already know that for $\Omega\subset \R^n$ $s$-quasi-open and $v \in H^s_0(\Omega)$ it holds
\be\label{cs}
[v]^2_{H^s(\R^n)} = d_s\int_{\R^{n+1}_+}|y|^a |\nabla g|^2\mathrm{d}X,\quad\mbox{with }d_s = 2^{2s-1}\frac{\Gamma(s)}{\Gamma(1-s)},
\ee
where $X=(x,y) \in \R^{n+1}, a=1-2s\in (-1,1)$ and $g \in H^{1,a}(\R^{n+1})$ is uniquely defined as the $L_a$-extension of $v$ in $\R^{n+1}_+$ (see \cite{fkj,fks} for more result of weighted degenerate Sobolev spaces $H^{1,a}$ and weighted degenerate Lebesgue spaces $L^{2,a}$).\\
Notice that, in this new setting, we can write the eigenvalue problem \eqref{equation.eigen} as follows
\be\label{equation.eigen.ext}
\begin{cases}
  L_a g=0 & \mbox{in } \R^{n+1}_+ \\
  -\partial^a_y g=\frac{\lambda}{d_s} g & \mbox{on }\Omega \times \{0\}\\
  g=0 & \mbox{on }(\R^n \setminus \Omega )\times \{0\},
\end{cases}
\quad\mbox{with }\partial^a_y g (x,0)=\lim_{y\to 0^+}y^a \partial_y g(x,y)
\ee
and $g(x,0)=v(x)$ in $\R^n$. Similarly, we can reformulate the shape minimization problem \eqref{min.shape} as
\be\label{var.ext}
\sum_{i=1}^{m} \lambda_i^s(\Omega)=d_s\text{min}\left\{ \int_{\R^{n+1}_+}|y|^a |\nabla G|^2\mathrm{d}X \colon G=(g^1,\dots,g^m)\in H^{1,a}_\Omega(\R^{n+1}_+;\R^m), \int_\Omega g^i g^j\mathrm{d}x=\delta_{ij}\right\}
\ee
where $g^i$ is the $L_a$-extension of the normalized eigenfunction $v^i$ (that is, it solves \eqref{equation.eigen.ext} with $\lambda=\lambda_i^s(\Omega), v=v^i$), and
$$
H^{1,a}_\Omega(\R^{n+1}_+;\R^m) = \left\{G \in H^{1,a}(\R^{n+1}_+;\R^m)\colon |G|(x,0)\equiv 0 \mbox{ in }\R^n \setminus \Omega\right\}.
$$
For the sake of simplicity, we will call $G=(g^1,\dots,g^m)$ in \eqref{var.ext} as the vector of normalized eigenfunctions on $\Omega$. Moreover, we can apply an even reflection through the thin-space $\{y=0\}$ and assume that each component of $G \in H^{1,a}_\Omega(\R^{n+1};\R^m)$ is symmetric with respect to the
$y$-variable.\\
Thus, we can rephrase Proposition \ref{alm} as follows:\\

 Let $\Omega\subset D$ be a solution to \eqref{min.shape} for some $\Lambda >0$. Then, given $\tilde{\Lambda}=2\Lambda/d_s$ and $G\in H^{1,a}_{\Omega}(\R^{n+1};\R^m)$

  the vector of normalized eigenfunctions on $\Omega$, for every $\delta\geq\norm{G}{L^\infty(\Omega)}$ there exist $K,\eps>0$ such that
\be\label{almost}
  \begin{aligned}
  \int_{\R^{n+1}}|y|^a |\nabla G|^2\mathrm{d}X +\tilde{\Lambda}\mathcal{L}_n(\{\abs{G}>0\}\cap \R^n) \leq&\, \left( 1+ K \normpic{\tilde{G}-G}{L^1(\R^n)}\right )  \int_{\R^{n+1}}|y|^a |\nabla \tilde{G}|^2\mathrm{d}X+\\
  &\, +\tilde{\Lambda}\mathcal{L}_n(\{|\tilde{G}|>0\}\cap \R^n)
  \end{aligned}
\ee

  for $\tilde{G} \in H^{1,a}_D(\R^{n+1};\R^m)$ such that $\normpic{\tilde{G}}{L^\infty(\R^n)}\leq \delta$ and $\normpic{G-\tilde{G}}{L^1(\R^n)}\leq \eps$.\\\\
 In the following Proposition we stress out how this new formulation allows to localize the problem. More precisely, the normalized eigenvector $G$ satisfies the following localized version of the almost-minimality condition.
  \begin{prop}\label{uti}
   Let $\Omega\subset D$ be a solution to \eqref{min.shape}, for some $\Lambda >0$, and $G=(g^1,\dots,g^m) \in H^{1,a}_{\Omega}(\R^{n+1};\R^m)$ be the vector of normalized eigenfunctions. Then, for every $\delta\geq\norm{G}{L^\infty(\Omega)}$ there exist $\sigma, r_0 >0$ such that, for every $r\in (0,r_0]$ and $X_0\in \R^n\times \{0\}$
    \be\label{almost.local}
  \begin{aligned}
  \int_{B_r(X_0)}|y|^a |\nabla G|^2\mathrm{d}X +\tilde{\Lambda}\mathcal{L}_n(\{\abs{G}>0\}\cap \mathcal{B}_r(X_0)) \leq&\, \left( 1+ \sigma r^n\right )  \int_{B_r(X_0)}|y|^a |\nabla \tilde{G}|^2\mathrm{d}X+\\
  &\, +\tilde{\Lambda}\mathcal{L}_n(\{|\tilde{G}|>0\}\cap \mathcal{B}_r(X_0))\\
  &\,+ \sigma \frac{2}{d_s}\sum_{i=1}^{m} \lambda_i^s(\Omega) r^n,
  \end{aligned}
\ee
for every $\tilde{G} \in H^{1,a}(\R^{n+1};\R^m)$ such that $\normpic{\tilde{G}}{L^\infty(\mathcal{B}_r(X_0))}\leq \delta$ and $G-\tilde{G}\in H^{1,a}_0(B_r(X_0);\R^m)$.
  \end{prop}
  \begin{proof}
  Let $\delta >0$ and $K,\eps>0$ be the constants defined in condition \eqref{almost}. Consider now $r\in (0,r_0]$ with $$
  r_0 = \left( \frac{\eps}{2\omega_n\delta}\right)^{1/n}.
  $$
  Let $\tilde{G} \in H^{1,a}(\R^{n+1};\R^m)$ be such that $\normpic{\tilde{G}}{L^\infty(\mathcal{B}_r(X_0))}\leq \delta$ and  $G-\tilde{G}\in H^{1,a}_0(B_r(X_0);\R^m)$, where $G-\tilde{G}$ is extended by zero in $\overline{\R^{n+1}_+}$. By the definition of $r_0>0$, we have
  \be\label{L1}
\normpic{G-\tilde{G}}{L^1(\R^n)}\leq 2\omega_n \delta r^n \leq 2\omega_n \delta r_0^n \leq \eps.
  \ee
  Hence, by \eqref{almost}, we get
    \begin{align*}
  \int_{B_r(X_0)}|y|^a (|\nabla G|^2-|\nabla \tilde{G}|^2)\mathrm{d}X +\tilde{\Lambda}\mathcal{L}_n(\{\abs{G}>0\}\cap \mathcal{B}_r(X_0)) \leq&\, \sigma r^n \int_{\R^{n+1}}|y|^a |\nabla \tilde{G}|^2\mathrm{d}X+\\
  &\, +\tilde{\Lambda}\mathcal{L}_n(\{\tilde{G}>0\}\cap \mathcal{B}_r(X_0)),
  \end{align*}
  with $\sigma =2K\omega_n \delta$. On the other hand, since by \eqref{var.ext}
  \be \label{boundedinH}
  \int_{\R^{n+1}}|y|^a\abs{\nabla G}^2\mathrm{d}X = \frac{2}{d_s}\sum_{i=1}^{m} \lambda_i^s(\Omega),
  \ee
  it follows
\begin{align*}
\int_{\R^{n+1}}|y|^a |\nabla \tilde{G}|^2\mathrm{d}X  &= \int_{\R^{n+1}}|y|^a |\nabla G|^2\mathrm{d}X + \int_{B_r(X_0)}|y|^a |\nabla \tilde{G}|^2\mathrm{d}X - \int_{B_r(X_0)}|y|^a |\nabla G|^2\mathrm{d}X\\
&= \int_{B_r(X_0)}|y|^a \left(|\nabla \tilde{G}|^2 -|\nabla G|^2\right) \mathrm{d}X+ \frac{2}{d_s}\sum_{i=1}^{m} \lambda_i^s(\Omega)
\end{align*}
from which we deduce that
 \begin{align*}
  \int_{B_r(X_0)}|y|^a |\nabla G|^2\mathrm{d}X +\tilde{\Lambda}\mathcal{L}_n(\{\abs{G}>0\}\cap \mathcal{B}_r(X_0)) \leq&\, \left( 1+ \sigma r^n\right )  \int_{B_r(X_0)}|y|^a |\nabla \tilde{G}|^2\mathrm{d}X+\\
  &\, +\tilde{\Lambda}\mathcal{L}_n(\{|\tilde{G}|>0\}\cap \mathcal{B}_r(X_0))\\
  &\, +\sigma \frac{2}{d_s}\sum_{i=1}^{m} \lambda_i^s(\Omega)r^n.
  \end{align*}
\end{proof}
Hence, given $G\in H^{1,a}(\R^{n+1};\R^m), \tilde{\Lambda}>0$ let us consider the functional
$$
\mathcal{J}(G) = \int_{\R^{n+1}}|y|^a |\nabla G|^2\mathrm{d}X +\tilde{\Lambda}\mathcal{L}_n(\{\abs{G}>0\}\cap \R^n)
$$
and its localized version
$$
\mathcal{J}(G,B_r(X_0)) = \int_{B_r(X_0)}|y|^a |\nabla G|^2\mathrm{d}X +\tilde{\Lambda}\mathcal{L}_n(\{\abs{G}>0\}\cap \mathcal{B}_r(X_0)).
$$
Finally, we end the Section with the following statement, where we collect all the previous formulation. In the rest of the paper, we will always refer to the following two almost-minimality conditions.
\begin{theorem}\label{definition}
  Set $\Lambda>0$ and let $\Omega\subset D$ be a shape optimizer to \eqref{min.shape} and $G=(g^1,\dots,g^m) \in H^{1,a}_\Omega(\R^{n+1};\R^m)$ be the vector of normalized eigenfunctions on $\Omega$.\\ Then, given $\tilde{\Lambda}=2\Lambda/d_s$, it holds the following \emph{almost-minimality condition} in $\R^{n+1}$: for every $\delta\geq\norm{G}{L^\infty(\Omega)}$ there exist $K,\eps>0$ such that
\be\label{almost.fin}
  \mathcal{J}(G) \leq \left( 1+ K \normpic{\tilde{G}-G}{L^1(\R^n)}\right )  \mathcal{J}(\tilde{G})
\ee
  for every $\tilde{G} \in H^{1,a}(\R^{n+1};\R^m)$ such that $\normpic{\tilde{G}}{L^\infty(\R^n)}\leq \delta$ and $\normpic{G-\tilde{G}}{L^1(\R^n)}\leq \eps$.\\
  Similarly, under the same assumptions, the vector $G$ satisfies the following \emph{localized almost-minimality condition}: for $\delta\geq\norm{G}{L^\infty(\Omega)}$ there exist $\sigma, r_0 >0$ such that, for every $r\in (0,r_0]$ and $X_0\in \R^n\times \{0\}$
    \be\label{almost.local.fin}
    \mathcal{J}(G,B_r(X_0)) \leq \left( 1+ \sigma r^n\right )   \mathcal{J}(\tilde{G},B_r(X_0)) + \sigma \frac{2}{d_s}\sum_{i=1}^{m} \lambda_i^s(\Omega)r^n,
\ee
for every $\tilde{G} \in H^{1,a}(\R^{n+1};\R^m)$ such that $\normpic{\tilde{G}}{L^\infty(\mathcal{B}_r(X_0))}\leq \delta$ and $G-\tilde{G}\in H^{1,a}_0(B_r(X_0);\R^m)$.
\end{theorem}
\begin{rem}\label{scaling}
  We remark that the localized almost-minimality condition \eqref{almost.local.fin} scales according to the $C^{0,s}$ rescaling. Indeed, if $G$ satisfies \eqref{almost.local.fin}, then for every $X_0 \in \R^n \times \{0\}, \rho>0$
  $$
  G_{X_0,\rho}(X)=\frac{1}{\rho^{s}}G(X_0+\rho X) \in H^{1,a}_{\Omega_{X_0,\rho}}(\R^{n+1};\R^m), \quad\mbox{with }\,\Omega_{X_0,\rho}=\frac{\Omega-X_0}{\rho},
  $$
  satisfies the same condition with constants $\sigma\rho^n$ and $r_0/\rho$. More precisely, 
  for every $r \in (0,r_0/\rho]$
  $$
\mathcal{J}(G_{X_0,\rho},B_r)\leq \left( 1+ (\sigma\rho^n) r^n\right )  \mathcal{J}(\tilde{G},B_r)+ \sigma \frac{2}{d_s}\sum_{i=1}^{m} \lambda_i^s(\Omega) r^n,
  $$
for every $\tilde{G} \in H^{1,a}(\R^{n+1};\R^m)$ such that $\normpic{\tilde{G}}{L^\infty(\mathcal{B}_{r})}\leq \delta\rho^{-s}$ and $G_{X_0,\rho}-\tilde{G}\in H^{1,a}_0(B_{r};\R^m)$.
  \end{rem}
\section{$C^{0,s}$- regularity and non-degeneracy of eigenfunctions}\label{local}
In this Section we start by focusing on some local properties of shape optimizers. In particular we obtain optimal regularity and  non-degeneracy of the vector of normalized eigenfunctions. As a consequence, we prove that shape optimizers are open sets and they satisfy some density estimates.\\In the end, we discuss the validity of the unique continuation principle for eigenfunctions on disconnected domain.
\subsection{Optimal Regularity}
  The purpose of this subsection is to obtain $C^{0,s}$-regularity
of the eigenfunction associated to a solution $\Omega$ of \eqref{min.shape}.\\
The following is a technical lemma about a decay type estimate for the gradient of $L_a$-harmonic functions.
\begin{lem}\label{grad}
  Let $u \in H^{1,a}(B_1)$ be $L_a$-harmonic in $B_1$ and symmetric with respect to $\{y=0\}$. Then, there exists $C=C(n,a)>0$ such that
$$
\fint_{B_r}|y|^a|\nabla u-\nabla u(0)|^2\mathrm{d}X \leq C r^{2} \fint_{B_1}|y|^a |\nabla u|^2\mathrm{d}X,
$$
for every $r \in (0,1/2]$.
\end{lem}
\begin{proof}
By the Poincaré-type inequality due to Fabes, Kenig and Serapioni
\cite[Theorem 1.5]{fks}, there exists $C_1 = C_1(n,a)>0$ such that
the following inequality holds
\[
  \int_{B_r} |y|^a |u-\bar u_{B_r}|^2 \mathrm{d}X \leq C_1 r^2 \int_{B_r} |y|^a |\nabla u|^2 \mathrm{d}X
\]
where $\bar u_{B_r}$ is the average of $u$ in the ball $B_r$, that is
\[
  \bar u_{B_r} = \frac{1}{r^{n+1+a}|B_1|_a} \int_{B_r}|y|^a u \mathrm{d}X, \quad\mbox{with } |B_1|_a = \int_{B_1}|y|^a\mathrm{d}X.
\]
By the mean value formula for $L_a$-harmonic function, we first get $\bar u_{B_r}=u(0)$.\\ Then, since by \cite[Theorem 2.6]{CRS} we know that
$$
\int_{B_r} |y|^a |\nabla u|^2 \mathrm{d}X \leq \left(\frac{r}{R}\right)^{n+1+a}\int_{B_R} |y|^a |\nabla u|^2 \mathrm{d}X
$$
for $0\leq r\leq R\leq 1$, by combining the previous inequalities with a Caccioppoli type estimate, for $r\leq 1/2$ we get
\begin{align*}
\int_{B_r}|y|^a|u-u(0)|^2\mathrm{d}X &\leq C_1r^2\int_{B_r} |y|^a |\nabla u|^2 \mathrm{d}X\\
&\leq C_1r^2\left(\frac{r}{1/2}\right)^{n+1+a}\int_{B_{1/2}} |y|^a |\nabla u|^2 \mathrm{d}X\\
&\leq C_2 r^{n+1+a+2}\int_{B_1} |y|^a u^2 \mathrm{d}X,
\end{align*}
where $C_2=C_2(n,a)$. Moreover, since for every $i=1,\dots,n$ the derivative $\partial_{x_i}u$ is $L_a$-harmonic, we immediately get
\be\label{1}
\fint_{B_r}|y|^a|\nabla_x u-(\nabla_x u)(0)|^2\mathrm{d}X \leq C_1(n,a) r^{2} \fint_{B_1}|y|^a |\nabla_x u|^2\mathrm{d}X,
\ee
for every $r \in (0,1/2]$. On the other hand, the covariant derivative associated to the $y$-direction
$$
\partial^a_y u = \begin{cases}
                   |y|^a \partial_y u & \mbox{if } y \neq 0\\
                   0 & \mbox{if }y=0
                 \end{cases},
$$
is $L_{-a}$-harmonic, and consequently
\be\label{2}
\fint_{B_r}|y|^{-a}(\partial^a_y u)^2\mathrm{d}X \leq C_1(n,-a) r^2 \fint_{B_1}|y|^{-a}(\partial^a_y u)^2\mathrm{d}X.
\ee
Finally, summing \eqref{1} and \eqref{2} we get
the claimed inequality with $C=\max\{C_1(n,a),C_1(n,-a)\}$.
\end{proof}
The following result is a dichotomy which has been first proposed in the scalar case in \cite[Proposition 2.1.]{DSalmost} for the case of almost minimizer of the thin one-phase problem (see \cite[Section 2]{DSalmostloc} for a similar result in the local case).
\begin{lem}\label{dico}
Let $\sigma'>0,\eps>0$ and $G \in H^{1,a}(B_1;\R^m)$ be symmetric with respect to $\{y=0\}$ and such that
$$
\mathcal{J}(G;B_1) \leq (1+\sigma') \mathcal{J}(\tilde{G};B_1)+\eps,
$$
where $\tilde{G}\in H^{1,a}(B_1;\R^m)$ satisfies $\norm{\tilde{G}}{L^\infty(\mathcal{B}_1)}\leq \norm{G}{L^\infty(\R^n)}$ and $\tilde{G}-G\in H^{1,a}_0(B_1;\R^m)$. Thus, if we set
  $$
  a_i:= \left(\fint_{B_{1}}|y|^a \abs{\nabla g^i}^2\mathrm{d}X \right)^{1/2},
  $$
  there exists universal constants $\eta \in (0,1), M, \sigma_0>0$ such that if $\sigma'\leq \sigma_0$, then for every $i=1,\dots,m$
  $$
  \mbox{either}\quad a_i\leq M\quad\mbox{or}\quad
  \left(\eta \fint_{B_{\eta}}|y|^a \abs{\nabla g^i}^2\mathrm{d}X  \right)^{1/2} \leq \frac{a_i}{2}.
  $$
\end{lem}
\begin{proof}
Let $\tilde{g}^i \in H^{1,a}(\R^{n+1}) \cap L^\infty(\R^n)$ be the $L_a$-harmonic replacement of $g^i$ in $B_1$, that is
$$
\begin{cases}
  L_a \tilde{g}^i=0 & \mbox{in } B_1 \\
  \tilde{g}^i=g^i & \mbox{on }\partial B_1.
\end{cases}
$$
By the symmetry of the datum on $\partial B_1$, we know that $\tilde{g}^i$ is symmetric with respect to $\{y=0\}$. By integration by parts, we easily deduce
\be\label{replac}
\int_{B_1} |y|^a \langle \nabla \tilde{g}^i,\nabla (g^i-\tilde{g}^i)\mathrm{d}X=0.
\ee
Since $\norm{\tilde{g}^i}{L^\infty(\R^n)}\leq \delta$, consider now the admissible competitor $\tilde{G}= (g^1,\dots,\tilde{g}^i,\dots,g^m)$. By the almost minimality condition, we first get
$$
  \int_{B_1}|y|^a (|\nabla g^i|^2-|\nabla \tilde{g}^i|^2)\mathrm{d}X \leq \sigma'  \int_{B_1}|y|^a |\nabla \tilde{G}|^2\mathrm{d}X+ \eps + \tilde{\Lambda}\omega_n.
$$
On the other hand, since
$$
\int_{B_1}|y|^a |\nabla \tilde{G}|^2\mathrm{d}X  \leq + \frac{2}{d_s}\sum_{i=1}^{m} \lambda_i^s(\Omega) + \int_{B_1}|y|^a \left(|\nabla \tilde{g}^i|^2 -|\nabla g^i|^2\right) \mathrm{d}X
$$
we deduce, together with \eqref{replac}, that
$$
\int_{B_1}|y|^a |\nabla (g^i- \tilde{g}^i)|^2\mathrm{d}X  \leq \sigma'  \int_{B_1}|y|^a |\nabla \tilde{g}^i|^2\mathrm{d}X +\tilde{C}
$$
with $\tilde{C}>0$. Since $\tilde{g}^i$ minimizes the Dirichlet's type energy associated to the operator $L_a$, we get
$$
\fint_{B_1}|y|^a |\nabla \tilde{g}^i|^2\mathrm{d}X \leq a_i^2,
$$
and since $\tilde{g}^i$ is symmetric with respect to $\{y=0\}$, by Lemma \ref{grad}
$$
\fint_{B_\eta}|y|^a|\nabla \tilde{g}^i-\nabla \tilde{g}^i(0)|^2\mathrm{d}X \leq C \eta^{2} a^2_i \quad\mbox{for }\eta\leq 1/2.
$$
On the other hand, since $|\nabla_{x}\tilde{g}^i|$ is $L_a$-subharmonic, by \cite[Lemma A.2.]{ww} we get
$$
\norm{\nabla_{x}\tilde{g}^i}{L^\infty(B_{1/2})}\leq C_0 a_i
$$
for some $C_0$ universal. Thus, if we denote $q=\nabla \tilde{g}^i(0) = (\nabla_x \tilde{g}^i(0),0)$ we conclude that $|q|\leq C_0 a_i$. Collecting the previous inequalities, for $r\in (0,1/2]$ we first get
$$
\fint_{B_r}|y|^a |\nabla g^i- q|^2\mathrm{d}X \leq 2\sigma' a_i^2 r^{-n-1} + 2Cr^{2}a_i^2 + \bar{C}r^{-n-1}
$$
for some $\bar{C}>0$ universal. Then, since $|q|\leq C_0 a_i$
\be\label{p}
r\fint_{B_r}|y|^a |\nabla g^i|^2\mathrm{d}X  \leq 4\sigma' a_i^2 r^{-n} + 4Cr^{3} a_i^2 + 2\bar{C}r^{-n}+2C_0^2 r a_i^2.
\ee
Now, if we choose $r$ small enough, such that
$$
2C_0^2 r \leq \frac18, \quad 4Cr^{3} \leq \frac{1}{24}
$$
and $\sigma'$ small so that
$$
4\sigma'  r^{-n}  \leq \frac{1}{24}
$$
we get the following dichotomy:
$$
\mbox{either}\quad2\bar{C}r^{-n}\geq \frac{a_i^2}{24}\quad\mbox{or}\quad
2\bar{C}r^{-n} < \frac{a_i^2}{24}.
$$
Hence, in the first case it follows the first alternative of the claimed dichotomy and in the second one, if we combine \eqref{p} with the choices of $r$ and $\sigma'$, we provide the bounds of the second alternative.
\end{proof}
We can now state our main regularity result for the normalized eigenfunctions associated to the shape minimization problem \eqref{min.shape}.
\begin{proposition}\label{optimalreg}
  Let $D\subset\R^n$ be open and bounded and $\Omega\subset D$ be a shape optimizer to \eqref{min.shape} and $G\in H^{1,a}_\Omega(\R^{n+1};\R^m)$ be the vector of normalized eigenfunctions on $\Omega$. Let $r_0>0$ be the constant associated to $\delta=\Vert G\Vert_{L^\infty(\R^n)}$ in the localized almost-minimality condition \eqref{almost.local.fin}. Then, $G \in C^{0,s}_\loc(\R^{n+1};\R^m)$ and
\be\label{holder.trace}
[g^i]_{C^{0,s}(B_{r_0/2}(X_0))} \leq C\left(1+ \lambda_i^s(\Omega)\right),
\ee
for every $X_0 \in \R^n \times \{0\}, i=1,\dots,m$, with $C>0$ universal constant.
\end{proposition}
\begin{proof}
Let $\delta=\norm{G}{L^\infty(\Omega)}$ and $\eta, M,\sigma_0$ be the constants from Lemma \ref{dico}. Fix $\sigma, r_0 >0$ as the constants of the almost-minimality condition \eqref{almost.local.fin} associated to $\delta$. By the scaling invariance, it is not restrictive to assume that $r_0=1$ and $\sigma\leq \sigma_0$: if not, consider
  $$
  G_{X_0,\rho}(X)=\frac{1}{\rho^{s}}G(X_0+\rho X)
  $$
with $X_0 \in \R^n\times \{0\}$ and $\rho \in (0,r_0]$. By Remark \ref{scaling},
  for every $r \in (0,r_0/\rho]$
  $$
\mathcal{J}(G_{X_0,\rho},B_r)\leq \left( 1+ (\sigma \rho^n) r^n\right )  \mathcal{J}(\tilde{G},B_r)+  \sigma \frac{2}{d_s}\sum_{i=1}^{m} \lambda_i^s(\Omega) r^n,
  $$
for every admissible $\tilde{G} \in H^{1,a}(B_r;\R^m)$. By choosing $\sigma=\sigma_0$ and $\rho = r_0$ we get
$$
\mathcal{J}(G_{X_0,r_0},B_1)\leq \left( 1+ \sigma'\right )  \mathcal{J}(\tilde{G},B_1)+ \eps,
$$
with $\sigma'=\sigma_0 r_0^n\leq \sigma_0$ and $\eps= \sigma 2d_s^{-1}\sum_{i=1}^{m} \lambda_i^s(\Omega)$. Thus, for $i=1,\dots,m$, set
$$
a_i(\eta) = \left(\eta \fint_{B_{\eta}}|y|^a \abs{\nabla g^i}^2\mathrm{d}X  \right)^{1/2}.
$$
We first show by induction that for all $N\geq 0$ the following inequality holds
\be\label{iteration}
a_i(\eta^N) \leq C(\eta)M + 2^{-N}a_i(1),
\ee
with $C(\eta)$ a large constant. By the previous observations, we can apply Lemma \ref{dico} which gives the desired inequality for $N=0$. Thus, let us suppose that it holds for $N$ and let us show that it holds also for the value $N+1$. By the scaling invariance of Remark \ref{scaling}, if we rescaled Lemma \ref{dico} we get
$$
\mbox{either}\quad a_i(\eta^N)\leq M \quad \mbox{or}\quad a_i(\eta^{N+1})\leq \frac{1}{2} a_i(\eta^N).
$$
From the last alternative it follows immediately the validity of \eqref{iteration} for the value $N+1$. Instead, in the first case, we can use that
$$
a_i(\eta^{N+1})\leq \eta^{-\frac{n+a}{2}}a_i(\eta^N) \leq C(\eta)M,
$$
with $C(\eta) = \eta^{-\frac{n+a}{2}}$. Finally $a_i(r)\leq C(1+a_i(1))$ for every $r \in (0,1)$, namely
$$
\int_{B_r}|y|^a |\nabla g^i|^2\mathrm{d}X \leq C(1+\lambda_i^s(\Omega))^2 r^{n}.
$$
Finally, by applying Morrey's type embedding 
we get the claimed estimate.
\end{proof}
Since by \eqref{var.ext} and \eqref{boundedinH}, the vector $G\in H^{1,a}_\Omega(\R^{n+1};\R^m)$ is bounded in $H^{1,a}(\R^{n+1};\R^m)$, we can slightly simplify the almost-minimality condition \eqref{almost.fin} as follows:\\

 for every $\delta\geq\norm{G}{L^\infty(\Omega)}$ there exist $K,\eps>0$ such that
\be\label{almost.fin.new}
  \mathcal{J}(G) \leq \mathcal{J}(\tilde{G}) + K \frac{2}{d_s}\sum_{i=1}^{m} \lambda_i^s(\Omega) \normpic{\tilde{G}-G}{L^1(\R^n)}
\ee

  for every $\tilde{G} \in H^{1,a}(\R^{n+1};\R^m)$ such that $\normpic{\tilde{G}}{L^\infty(\R^n)}\leq \delta$ and $\normpic{G-\tilde{G}}{L^1(\R^n)}\leq \eps$.\\\\
Similarly, we can change condition \eqref{almost.local.fin} as well:\\

for every $\delta\geq\norm{G}{L^\infty(\Omega)}$ there exist $\sigma, r_0 >0$ such that, for every $r\in (0,r_0]$ and $X_0\in \R^n\times \{0\}$
    \be\label{almost.local.new}
    \mathcal{J}(G,B_r(X_0)) \leq \mathcal{J}(\tilde{G},B_r(X_0)) + \sigma \frac{2}{d_s}\sum_{i=1}^{m} \lambda_i^s(\Omega)\normpic{\tilde{G}-G}{L^1(\mathcal{B}_r(X_0))}. 
\ee

for every $\tilde{G} \in H^{1,a}(\R^{n+1};\R^m)$ such that $\normpic{\tilde{G}}{L^\infty(\mathcal{B}_r(X_0))}\leq \delta$ and $G-\tilde{G}\in H^{1,a}_0(B_r(X_0);\R^m)$.\\\\
It is more convenient to work with these alternative formulations since in these conditions the energies simplify in the regions where the competitor and the solution are equal.
\begin{remark}
This new almost-minimality condition \eqref{almost.local.new} holds true also if we remove the assumption on the upper bound in $H^{1,a}(\R^{n+1}_+;\R^m)$, due to \eqref{boundedinH}. Precisely, as in \cite{DSalmost}, by combining the regularity result Proposition \ref{optimalreg} and the validity of a Caccioppoli type inequality for every components of $G$, we can conclude that it is still uniformly bounded in $H^{1,a}(B_{r};\R^m)$ for every $r\in (0,r_0]$.
\end{remark}
\begin{remark}\label{emptyinterior}
  By the previous result, we already know that $\{|G|>0\}\cap \R^n$ and more generally $\{g^i\neq 0\}\cap \R^n$ are open subsets of $\R^n$. Moreover, as we see in Proposition \ref{supp}, by the unique continuation principle we get
  $$
  \mathcal{L}_n\left(\left(\{|G|>0\}\cap \R^n\right)\setminus \left(\{g^i\neq 0\}\cap \R^n\right)\right) = 0,
  $$
  for every $i=1,\dots,m$.
\end{remark}
\subsection{Non-degeneracy}
   The non-degeneracy of solutions near the free boundary points will allow to understand the measure-theoretic structure of the free boundary via a blow-up analysis.\\
   Inspired by \cite{DSalmost,DSalmostloc}, our analysis is mainly based on the following Lemma, in which we compare every components of $G$ with its $L_a$-harmonic replacement.
\begin{lem}\label{harmonic.rep}
Given $\Omega\subset D$ a shape optimizer to \eqref{min.shape} and $G\in H^{1,a}_\Omega(\R^{n+1};\R^m)$ the vector of normalized eigenfunctions on $\Omega$, let us assume that $B_1\subset \{|G|>0\}$.\\ Then, for every $i=1,\dots,m$, given $v^i$ the $L_a$-harmonic replacement of $g^i$ in $B_{7/8}$, that is,
$$
\begin{cases}
  L_a v^i=0 & \mbox{in } B_{7/8}\\
  v^i=g^i & \mbox{on } B_1\setminus B_{7/8}.
\end{cases}
$$
we get
\be\label{compare}
\norm{g^i-v^i}{L^\infty(B_{1/2})}\leq \omega(\sigma),
\ee
where $\omega(\sigma)\to 0$ as $\sigma\to 0$.
\end{lem}
\begin{proof}
 Fixed $i=1,\dots,m$, consider the competitor $\tilde G = (g^1,\dots,v^i,\dots,g^m)$ with $v^i$ is the $L_a$-harmonic replacement of $g^i$ in $B_{7/8}$. Then, by the localized almost-minimality condition \eqref{almost.local.new}
 $$
 \mathcal{J}(G,B_1) \leq \mathcal{J}(\tilde{G},B_1) + \sigma \frac{2}{d_s}\sum_{i=1}^{m} \lambda_i^s(\Omega)\normpic{\tilde{G}-G}{L^1(\mathcal{B}_1)}
 $$
 and by Remark \ref{emptyinterior} we first deduce
$$
\int_{B_{7/8}}|y|^a \abs{\nabla g^i}^2\mathrm{d}X -\int_{B_{7/8}}|y|^a \abs{\nabla v^i}^2\mathrm{d}X \leq \sigma \frac{2}{d_s}\sum_{i=1}^{m} \lambda_i^s(\Omega)\normpic{\tilde{G}-G}{L^1(\mathcal{B}_1)}.
$$
and then, by exploiting the harmonicity of $v^i$, we get
$$
\int_{B_{7/8}}|y|^a \abs{\nabla (g^i - v^i)}^2\mathrm{d}X \leq  \sigma \frac{2}{d_s}\sum_{i=1}^{m} \lambda_i^s(\Omega)\normpic{\tilde{G}-G}{L^1(\mathcal{B}_1)}.
$$
By Poincaré inequality, there exists $C>0$ depending on $n$ and $s \in (0,1)$ such that
$$
\int_{B_{3/4}}|y|^a (g^i-v^i)^2\mathrm{d}X \leq C\sigma,
$$
with $g^i-v^i$ uniformly $C^{0,s}$ in $B_{3/4}$. Finally, if at some point $X \in B_{1/2}$ we have $(g^i-v^i)(X)\geq M$, then
$$
M^{n+2}\leq C\sigma,
$$
for some $C>0$. Thus \eqref{compare} holds with $\omega(\sigma)= C \sigma^{1/(n+2)}$.
\end{proof}
\begin{remark}\label{Gsub}
The previous result allows to overcome the absence of a minimality condition by comparing the almost-minimizer with its harmonic replacement. Nevertheless, since each eigenfunction $g^i$ solves \eqref{equation.eigen.ext} for $\lambda = \lambda_i^s(\Omega)$, we can easily compute that $|G|=|(g^1,\dots,g^m)|\in H^{1,a}_\Omega(\R^{n+1}_+)$ is a weak solution to
$$
\begin{cases}
  -L_a |G|\leq 0 & \mbox{in } \R^{n+1}_+ \\
  -\partial^a_y |G| \leq \lambda_m^s(\Omega)|G| & \mbox{on }\Omega.
\end{cases}
$$
Indeed, it holds
\begin{align*}
L_a |G| &= \sum_{i=1}^{m} \left(\frac{1}{|G|}g^i L_a g^i + |y|^a\frac{\abs{\nabla g^i}^2}{|G|} - |y|^a \frac{g^i \langle\nabla g^i,\nabla |G|\rangle}{|G|^2}\right)\\
&= \frac{|y|^a}{|G|^3}\sum_{i,j} \left((g^j)^2\abs{\nabla g^i}^2- g^i g^j\langle\nabla g^i,\nabla g^j\rangle\right)\geq 0,
\end{align*}
weakly in $H^{1,a}(\R^{n+1}_+)$. Moreover, by \eqref{Linfinity.G}
we get $\norm{|G|}{L^\infty(\R^n)} \leq \sqrt{m}\left[\tilde{C}_{n,s} \lambda^s_m(\Omega)\right]^{\frac{n}{4s}}$ and consequently, by taking the symmetric extension through $\{y=0\}$, we get
$$
-L_a \left( |G| + \frac{\sqrt{m}\tilde{C}_{n,s}^{n/4s}}{1-a} \lambda^s_m(\Omega)^{\frac{n+4s}{4s}} |y|^{1-a}\right) \leq 0 \quad\mbox{in $\R^{n+1}_+$}.
$$
Similarly, we have
$$
-L_a \left( |g^i| + \frac{\tilde{C}_{n,s}^{n/4s}}{1-a} \lambda^s_i(\Omega)^{\frac{n+4s}{4s}} |y|^{1-a}\right) \leq 0 \quad\mbox{in $\R^{n+1}_+$}
$$
for every $i=1,\dots,m$.
\end{remark}
We start by proving the following weak non-degeneracy condition.
\begin{prop}\label{non-deg1}
Let $\Omega\subset D$ be a shape optimizer to \eqref{min.shape} and $G\in H^{1,a}_\Omega(\R^{n+1};\R^m)$
be the vector of normalized eigenfunctions on $\Omega$. Set $r_0>0$ to be the constant associated to $\delta=\Vert G\Vert_{L^\infty(\R^n)}$ in the localized almost-minimality condition \eqref{almost.local.new}.\\ Then, there exists a universal constant $c_2>0$ such that
  \be
  \abs{G}(X)\geq c_2 \mathrm{dist}(X, \partial \{\abs{G}>0\})^{s}\quad
  \text{in $\mathcal{B}_{r_0/2}^+(G)$}.
  \ee
\end{prop}
\begin{proof}
By Remark \ref{scaling}, up to translation and rescaling it is enough to show that if $G$ satisfies \eqref{almost.local.new} for $\sigma>0$ sufficiently small in a large ball and
  \be\label{assumption}
  \mathrm{dist}(0, \partial \{\abs{G}>0\})=1,
  \ee
 then $|G|(0) \geq c_2 >0$ for some $c_2$ small to be made precise later.\\
 Indeed, assume $\mathcal B_1 \subset \{|G|>0\}$ and 
 for every $i=1,\dots,m$ set $v^i$ as the $L_a$-harmonic replacement of $g^i$ in $B_{7/8}$: according to Lemma \ref{harmonic.rep} the result follows once we prove the claimed statement for the vector $V=(v^1,\dots,v^m)$. Now, the $v^i$'s are uniformly bounded say in $B_{3/4}$ and hence, since they are $L_a$-harmonic and symmetric with respect to $\{y=0\}$, we get
$$|v^i(X) - v^i(0)| \leq K |X|, \quad \text{in $B_{1/2}$},$$ for $K>0$ universal. Thus,
$$v^i(X) \leq v^i(0) + K|X|, \quad \text{in $B_{1/2}$.}$$
Let
$$
V_{\delta}(X) = \frac{1}{\delta^{s}}V(\delta X),\quad G_{\delta}(X) = \frac{1}{\delta^{s}}G(\delta X), \quad\mbox{for }X \in B_1
$$
with $\delta>0$ universal to be chosen universal later. Then, we get
$$v^i_\delta \leq v^i(0)\delta^{-s}+ K \delta^{1-s} \leq C v_\delta^i(0) \quad\text{in $B_1$},$$
for every $i=1,\dots,m$ and $\delta$ small enough. Moreover, since the $v^i_\delta$'s  are $L_a$-harmonic in $B_1$, the bound above and the validity of mean value formulas for the $L_a$-operator 
imply
$$\|v^i_\delta\|_{L^\infty(B_{1})}, \|\nabla_x v^i_\delta\|_{L^\infty(B_{1/2})}, \|\partial^a_y v^i_\delta\|_{L^\infty(B_{1/2})}\leq C v^i_\delta(0),$$
with $C>0$ universal (possibly changing from line to line). Let $\varphi \in C_0^{\infty}(B_{1/2}), 0 \leq \varphi \leq 1$ such that $\varphi \equiv 1$ in $B_{1/4}$, then
$$\int_{B_1}|y|^a|\nabla v^i_\delta|^2 \mathrm{d}X \geq \int_{B_1}|y|^a |\nabla (v^i_\delta(1-\varphi))|^2 \mathrm{d}X - C (v_\delta^i(0))^2$$
and on the other hand
$$\mathcal{L}_n(\{|V_\delta|>0\}\cap B_1) \geq \mathcal{L}_n(\{|V_\delta|(1-\varphi)>0\}\cap B_1) + C_0.$$ In conclusion, by the minimality of $V_\delta$ with respect to its Dirichlet-type energy, we first deduce
$$
\mathcal{J}(V_\delta,B_1)\leq \mathcal{J}(G_\delta,B_1) \leq
\mathcal{J}(V_\delta(1-\varphi),B_1) +\sigma C_n
$$
and then
$$
C|V_\delta|^2(0) \geq \tilde{\Lambda}C_0 - \sigma C_n
$$
Finally, we reach the claim for $\sigma>0$ sufficiently small.
\end{proof}

Finally, we are able to prove the strong non-degeneracy near the free boundary.
\begin{prop}\label{non-deg2}
Consider $\Omega\subset D$ a shape optimizer to \eqref{min.shape} and $G\in H^{1,a}_\Omega(\R^{n+1};\R^m)$
the vector of normalized eigenfunctions on $\Omega$.\\
Let $r_0>0$ as in Proposition \ref{optimalreg}, then if $X_0 \in F(G)$, for every $r \in (0,r_0/2)$ we have
  \be\label{non-deg}
  \sup_{\mathcal B_r(X_0)} \abs{G} \geq c r^{s},
  \ee
  for some universal constant $c >0, r_0>0$.
\end{prop}
In view of Proposition \ref{optimalreg} and Proposition \ref{non-deg1}, Proposition \ref{non-deg2} follows by applying the next lemma to the function
$$
f(X)=|G|(X) + \frac{\sqrt{m}\tilde{C}_{n,s}^{n/4s}}{1-a} \lambda^s_m(\Omega)^{\frac{n+4s}{4s}} |y|^{1-a},
$$
introduced in Remark \ref{Gsub}.
\begin{lem}
  Let $f\geq 0$ be defined in $B_1$ and subharmonic
  in $B^+_1(f).$ Assume that there is a small constant $\eta>0$ such that \be\label{cs*}\|f\|_{C^{0,s}(B_1)}\leq \eta^{-1},\ee and $f$ satisfies the non-degeneracy condition on $\mathcal B_1$,
\be\label{c3change}
f(X) \geq \eta \; \mathrm{dist}(X,\{f=0\})^{s} \quad\text{for every } X \in \mathcal{B}_1.
\ee
Then if $0 \in F(f)$, we get  $$\sup_{\mathcal B_r} f \geq c(\eta)\; r^{s}, \quad \text{for }r \leq 1.$$
\end{lem}
\begin{proof} The proof follows the lines of \cite[Proposition 3.3]{CRS}
(see also \cite[Lemma 2.9]{DT2}).\\
Given a point $X_0 \in \mathcal B_1^+(f)$, to be chosen close enough to $0 \in F(f)$, we construct a sequence of points $(X_k)_k \subset  \mathcal B_1$ such that $$f(X_{k+1})\geq(1+\delta)f(X_k), \quad |X_{k+1}- X_k| \leq C(\eta) \mathrm{dist}(X_{k},\{f=0\}),$$ with $\delta$ small depending on $\eta$.

Then, using \eqref{c3change} 
and that $(f(X_k))_k$ grows geometrically, we find \begin{align*}|X_{k+1} - X_0| &\leq \sum_{i=0}^{k} |X_{i+1}-X_i|  \leq C(\eta) \sum_{i=0}^k \mathrm{dist}(X_{i},\{f=0\}) \\ & \leq \frac{C(\eta)}{\eta^2} \sum_{i=0}^k f(X_{i})^{1/s} \leq c(\eta) f(X_{k+1})^{1/s} 
.\end{align*}
Hence for a sequence of radii $r_k = \mathrm{dist}(X_k,\{f=0\}),$
we have that $$\sup_{\mathcal B_{r_k}(X_0)} f \geq c r_k^{s}$$ from which we obtain that
 $$\sup_{\mathcal B_{r}(X_0)} f \geq c r^{s}, \quad \text{for all $r \geq |X_0|.$}$$
The conclusion follows by letting $X_0$ go to $0\in F(f)$.

We now show that the sequence of $X_k$'s  exists. After scaling, assume we constructed $X_k$ such that $$f(X_k) = 1.$$ Let us call with $Y_k \in F(f)$ the point where the distance from $X_k$ to $\{f=0\}$ is achieved. By \eqref{cs*} and \eqref{c3change}, we get $$c(\eta) \leq r_k=|X_k-Y_k| \leq C(\eta).$$
Assume by contradiction that we cannot find $X_{k+1}$ in $\mathcal B_M(X_k)$ with $M$ large to be specified later, such that $$f(X_{k+1}) \geq 1+\delta.$$ Then $f \leq 1+\delta +w$ with $w$ a $L_a$-harmonic function in $B_M^+(X_k)$ such that $$w=0 \quad \text{on $\{y =0\}$}, \quad w=f \quad \text{on $\partial B_M(X_k) \cap \{y>0\}$}.$$
Thus, we have
$$w \leq C(n) \frac{y^{2s}}{M} \sup_{B_M^+(X_k)} f \leq C \eta^{-1} y^{2s} M^{s-1} \leq \delta \quad \text{in $B_{r_k}(X_k),$}$$ for $M$ sufficiently large depending on $\delta$. Thus, \be\label{bound1}f \leq 1+2\delta \quad \text{in $B_{r_k}(X_k).$}\ee
On the other hand, $f(Y_k)=0, Y_k \in \partial B_{r_k}(X_k)$. Thus from the H\"older continuity of $f$ we find \be\label{bound2}f \leq \frac 1 2, \quad \text{in $B_{c(\eta)}(Y_k)$}.\ee
If $\delta$ is sufficiently small \eqref{bound1}-\eqref{bound2} contradict that $$1=f(X_k) \leq \fint_{B_{r_k}(X_k)}|y|^a f\mathrm{d}X.$$ \end{proof}
Before stating the main result on density estimates, we need the following lemma in which we prove a compactness result for sequence of almost-minimizers uniformly bounded in $H^{1,a}(\R^{n+1};\R^m)$.
\begin{lem}\label{compact}
Let $(G_k)_k$ be a sequence of almost-minimizer in $B_1$ in the sense of condition \eqref{almost.local.new}, uniformly bounded in $H^{1,a}(B_1;\R^m)$. Then, up to a subsequence, there exists a limit function $G_\infty$ such that
\begin{itemize}
\item $G_\infty \in H^{1,a}_{\loc}(B_1;\R^m)\cap C^{0,s}_\loc(\overline{B_1};\R^m)$;
  \item $G_k \to G_\infty$ in $C^{0,\alpha}_\loc(\overline{B_1};\R^m)$, for every $\alpha \in (0,s)$;
  \item $G_k \rightharpoonup G_\infty$ weakly in $H^{1,a}_\loc(B_1;\R^m)$;
  \item $G_\infty$ is an almost-minimizer in $B_1$  in the sense of condition \eqref{almost.local.new}.
\end{itemize}
\end{lem}
\begin{proof}
By Proposition \ref{optimalreg} we already know that $G_k \to G_\infty$ uniformly on every compact set of $B_1$ and so in $C^{0,\alpha}_\loc(\overline{B})$, for every $\alpha \in (0,s)$. Moreover, by Ascoli-Arzel\'a theorem it follows that $G_\infty \in C^{0,s}(\overline{B})$.\\
By assumptions, the sequence is uniformly bounded in $H^{1,a}(B_1;\R^m)$ and sot it weakly converges to some $G_\infty \in H^{1,a}(B_1;\R^m)$.

In conclusion, let us show that for every $\delta\geq\norm{G_\infty}{L^\infty(\Omega)}$ there exist $\sigma, r_0 >0$ such that, for $r \in (0,r_0]$ we have
$$
\mathcal{J}(G_\infty,B_r) \leq \mathcal{J}(G_\infty+ \Psi,B_r) + \sigma \frac{2}{d_s}\sum_{i=1}^{m} \lambda_i^s(\Omega)\norm{\Psi}{L^1(\mathcal{B}_r)},
$$
for every $\Psi  \in H^{1,a}_0(B_r;\R^m)$ such that $\normpic{\Psi}{L^\infty(\mathcal{B}_r)}\leq \delta$. For the sake of simplicity, we denote
$$
\tilde{C}=\frac{2}{d_s}\sum_{i=1}^{m} \lambda_i^s(\Omega).
$$
Since we already know by Proposition \ref{optimalreg} that there exists a local minimizer H\"{o}lder continuous of class $C^{0,s}$, we can assume that $\Psi$ is continuous. Therefore, for every $k>0$ let us consider the competitor
$$
G_{k,\eps} = \sum_{i=1}^m (g^i_k+\psi^i-\eps \eta)_+ e^i - (g^i_k+\psi^i+\eps \eta)_- e^i,
$$
with $\eta \in C^\infty_c(B_{(1+r)/2})$ such that $0\leq\eta\leq1$ and  $\eta \equiv 1 $ on a neighborhood of $\overline{B_r}$. Hence, by the almost minimality of $G_k$ in $B_{(r_0+r)/2}$, namely we get for $r \in (0,r_0]$
$$
\mathcal{J}(G_k,B_{(r_0+r)/2}) \leq \mathcal{J}(G_{k,\eps},B_{(r_0+r)/2}) + \sigma\tilde{C} \norm{G_{k,\eps}-G_k}{L^1(\mathcal{B}_{(r_0+r)/2})},
$$
we have
\begin{align*}
\mathcal{L}_n(\mathcal{B}_{(r_0+r)/2}\cap \{|G_{k}|>0\} ) \leq &\, \sum_{i=1}^m\int_{B_{(r_0+r)/2}}{|y|^a\abs{\nabla \psi^i}^2 + 2|y|^a \langle \nabla \psi^i,\nabla g^i_k\rangle\mathrm{d}X}+\\
&\,+\eps \sum_{i=1}^m\int_{\mathrm{supp}\eta \setminus B_r}{\eps|y|^a\abs{\nabla \eta}^2 + 2|y|^a\langle \nabla \eta,\nabla (g^i_k +\psi^i) \rangle\mathrm{d}X}+\\
&\, +\mathcal{L}_n(\mathcal{B}_{(r_0+r)/2}\cap \{|G_{k,\eps}|>0\} ) + \sigma\tilde{C} \norm{G_{k,\eps}-G_k}{L^1(\mathcal{B}_{(r_0+r)/2})}.
\end{align*}
In particular, localizing the measure of the positivity set in $\mathcal{B}_r$, we get
\begin{align*}
\mathcal{L}_n(\mathcal{B}_r\cap \{|G_{k}|>0\} ) \leq &\, \sum_{i=1}^m\int_{B_{(r_0+r)/2}}{|y|^a\abs{\nabla \psi^i}^2 + 2|y|^a \langle \nabla \psi^i,\nabla g^i_k\rangle\mathrm{d}X}+ C\eps + \sigma\tilde{C} \norm{\Psi}{L^1(\mathcal{B}_r)}\\
&\,+\int_{\mathcal{B}_{(r_0+r)/2}}{\chi_{\{|G_{k,\eps}|>0\}}\mathrm{d}x} - \int_{\mathcal{B}_{(r_0+r)/2}\setminus \overline{\mathcal{B}_r}}{\chi_{\{|G_{k}|>0\}}\mathrm{d}x},
\end{align*}
where we used that $(G_k)_k$ is uniformly bounded in $H^{1,a}(B_{(r_0+r)/2})$. Since
\begin{align*}
\{g^i_k - \eps \eta>0 \}\setminus \overline{B_r} &\subseteq \{g^i_k >0\}\setminus \overline{B_r}\\
\{g^i_k + \eps \eta<0 \}\setminus \overline{B_r} &\subseteq \{g^i_k <0\}\setminus \overline{B_r}
\end{align*}
and by the uniform convergence
\begin{align*}
\{g^i_k+\psi^i -\eps >0 \}\cap\overline{B_r} &\subseteq \{g^i_{\infty} +\psi^i >0\}\cap\overline{B_r}\\
\{g^i_k+\psi^i + \eps <0 \}\cap \overline{B_r} &\subseteq \{g^i_{\infty}+\psi^i <0\}\cap\overline{B_r},
\end{align*}
we deduce\begin{align*}
\mathcal{L}_n(\mathcal{B}_r \cap \{|G_{k}|>0\}) \leq &\,\int_{B_{(r_0+r)/2}}{|y|^a \left(\abs{\nabla \Psi}^2 + 2 \langle \nabla \Psi,\nabla G_k\rangle\right)\mathrm{d}X}\\
&\, +\mathcal{L}_n(\mathcal{B}_r \cap \{|G_{{\infty}}+\Psi|>0\})+ \sigma\tilde{C} \norm{\Psi}{L^1(\mathcal{B}_r)}+C\eps.
\end{align*}

Now, using that $G_k \rightharpoonup G_\infty$ weakly in $H^{1}_\loc(B_1)$ and uniformly on $\overline{B_r}$, we obtain
$$
\mathcal{J}(G_\infty,B_r) \leq \int_{B_r}{|y|^a\abs{\nabla (G_\infty + \Psi)}^2\mathrm{d}X} + \mathcal{L}_n(\mathcal{B}_r \cap \{|G_{\infty}+\Psi|>0\} )+\sigma\tilde{C} \norm{\Psi}{L^1(\mathcal{B}_r)}+C\eps
$$
for every $\eps>0$, which implies the desired inequality.
\end{proof}

The following density estimates for the positivity set of $\abs{G}$ are obtained by a straightforward combination of the non-degeneracy condition \eqref{non-deg} and the optimal regularity of local minimizer.
\begin{cor}\label{density}
Let $\Omega\subset D$ be a shape optimizer to \eqref{min.shape} and $G\in H^{1,a}_\Omega(\R^{n+1};\R^m)$
be the vector of normalized eigenfunctions on $\Omega$.
Let $r_0>0$ be as in Proposition \ref{optimalreg} and $X_0\in F(G)$.\\ Then, for every $r \in (0,r_0/2)$
\be\label{density.eq}
\eps_0 \omega_n r^n\leq \mathcal{L}_n(\mathcal{B}_r(X_0)\cap \{|G|>0\}) \leq (1-\eps_0)\omega_n r^n,
\ee
for some $\eps_0>0$.
\end{cor}
\begin{proof}
For the sake of simplicity, assume that $X_0=0\in F(G)$. The proof of the left hand side is a combination of Proposition \ref{optimalreg} and Proposition \ref{non-deg2}. More precisely, on one hand for $r$ small enough there exists $X_r\in \mathcal{B}^+_r(G)$ such that $\abs{G}(X_r)\geq C r^{s}$. On the other one, since $\abs{G}$ is of class $C^{0,s}$, by setting
  $$
  C_0 = \min\left\{1,\frac{C}{[|G|]_{C^{0,s}}}\right\},
  $$
  we have that $\abs{G}>0$ in $\mathcal{B}_{C_0 r}^+(X_r)$, which proves the claimed lower bound.\\
On the other hand, since $\abs{G}$ is non-negative, up to rescaling, the upper bound in \eqref{density.eq}  is equivalent to
$$\mathcal{L}_n(\mathcal{B}_1\cap \{\abs{G}=0\})\geq \eps_0.$$ Thus, suppose by contradiction that there exists a sequence $(G_k)_k$ of eigenfunctions associated to the problem \eqref{min.shape} such that $0 \in F(G_k)$ and
$$
\lim_{k\to \infty} \mathcal{L}_n(\mathcal{B}_1\cap\{\abs{G_k}=0\}) =0.
$$
By Proposition \ref{optimalreg} and Lemma \ref{compact}, we already know that $G_k \to G_\infty$ weakly in $H^{1,a}(B_{1/2})$ and uniformly on every compact set of $B_{1/2}$. Moreover, $G_\infty\in H^{1,a}_{\loc}(B_{1/2}^+)\cap C^{0,s}_\loc(\overline{B_{1/2}})$ is a local minimizer in $B_{1/2}$. Now, fixed $r \in (0,1/2)$ and $X_0 \in B_{1/2}$, consider $\tilde{g}^i_k \colon B_r(X_0) \to \R$ be the  $L_a$-harmonic replacement of $g^i_k$ in $B_r(X_0)$, that is be such that
$$
\begin{cases}
L_a \tilde{g}^i_k =0 & \mbox{in } B_r(X_0) \\
  \tilde{g}^i_k=g^i_k & \mbox{on }\partial B_r(X_0).
\end{cases}
$$
By the almost-minimality condition \eqref{almost.local.new} for $G_k$, given the competitor $\widetilde{G}_k = (g^1_k, \dots, \tilde{g}^i_k, \dots, g^m_k)$, we deduce
\be\label{convergence}
\int_{B_r(X_0)}{|y|^a|\nabla (g^i_k-\tilde{g}^i_k)|^2\mathrm{d}X} \leq
\mathcal{L}_n(\mathcal{B}_r(X_0)\cap\{\abs{G_k} =0\}) +  \sigma C[|G_k|]_{C^{0,s}(B_{r_0/2})} r^{n+s}.
\ee
Since the H\"{o}lder seminorm is uniformly bounded, up to a subsequence, the sequence $(\widetilde{G}_k)_k$ do converge uniformly on every compact set of $B_{r}(X_0)$ to some function $\widetilde{G}_\infty\in H^{1}_{\loc}(B_{r}^+(X_0))$. Thus, by applying Fatou's Lemma to \eqref{convergence}, we get
$$
\int_{B_r(X_0)}{|y|^a|\nabla (g^i_\infty-\tilde{g}^i_\infty)|^2\mathrm{d}X} \leq \sigma Cr^{n+s}, \quad \mbox{for }r \in (0,1/2).
$$
for some $C>0$. Finally, by \cite[Theorem 2.6]{CRS}, for $r<\rho<1/2$ we get
$$
\int_{B_{r}(X_0)}{|y|^a|\nabla g^i_\infty|^2\mathrm{d}X}\leq
C \rho^{n+s} + C \left(\frac{r}{\rho}\right)^{n+2-2s}\int_{B_{\rho}(X_0)}{|y|^a|\nabla g^i_\infty|^2\mathrm{d}X}.
$$
Hence, fixed $\delta <1/2$ such that $q=C \delta<1$, if
$\rho=\delta^{k-1}, r=\delta^{k}$ and $\mu=\delta^{n+s}$ we get
$$
\int_{B_{\delta^{k}}(X_0)} |y|^a|\nabla g^i_\infty|^2 \mathrm{d}X \leq C \mu^{k-1} + C \mu\delta^{2-s} \int_{B_{\delta^{k-1}}(X_0)} |y|^a|\nabla g^i_\infty|^2 \mathrm{d}X
$$
and iterating the previous estimate, with $\delta>0$ such that $q=C\delta^{2-s}<1$, we get
$$
\int_{B_{\delta^{k}}(X_0)}|y|^a |\nabla g^i_\infty|^2 \mathrm{d}X \leq C \mu^{k-1}\sum_{i=0}^{k-1} q^i 
\leq C \mu^{k-1}\frac{1}{1-q}.
$$
Hence, there exists a universal constant $\tilde{C}>0$ such that
$$
\int_{B_{r}(X_0)}{|y|^a |\nabla g^i_\infty|^2\mathrm{d}X} \leq \tilde{C} r^{n+s},
$$
for every $r \in (0,1/2)$. By a covering argument and Morrey's embeddings, the function $g^i_\infty$ is H\"{o}lder continuous of order $C^{0,s+\eps}$ with $\eps = \min(s/2,(1-s)/2)$, in contradiction with Proposition \ref{non-deg2}.
\end{proof}
\begin{rem}\label{upper}
In the local case \cite{MTV1}, the authors highlight how the sign of the first eigenfunction plays a major role in the proof of the upper bound on the density. Indeed, in the local setting is fundamental to know that the first eigenfunction of the shape optimizer $\Omega$ is non-degenerate near $\partial \Omega$ (see \cite[Lemma 2.10]{MTV1}) and that $\Omega$ is connected (see \cite[Corollary 4.3]{MTV1}).\\
Instead, in our case the validity of the upper bound only relies on the different local regularity of fractional eigenfunctions near their zero set depending on whether or not they change sign.
\end{rem}
\subsection{Unique continuation of eigenfunctions}\label{sub.unique}
In this last part of the Section we collect few observation related to some specific feature due to the nonlocal attitude of the problem \eqref{equation.eigen}. More precisely, since we are not able to prove that shape optimizer of \eqref{min.shape} are connected set, we show some nonlocal effect for eigenfunctions in disconnected sets.
\begin{remark}\label{disconnected}
  In the local case \cite[Corollary 4.3]{MTV1}, the authors proved that every shape optimizer is a connected set. Their proof strictly relies on the invariance of the spectrum of the union of disjoint sets under the translations of each component and on the decomposition of the whole spectrum as the union of the spectrum of each connected components.\\In the nonlocal case, we can not use a similar strategy since the mutual position of the connected components impacts the value of each eigenvalues (see \cite{brascoparini} for some consequences of this feature in the case of the second eigenvalue). Obviously, the decomposition of the spectrum in each component is false in the nonlocal case.
\end{remark}
It is known by \cite[Theorem 2.8.]{brascoparini} that the first normalized eigenfunction on an open bounded set $\Omega$, even disconnected, is strictly positive (or negative) in the whole domain.\\On the other hand, if we restrict our attention to shape optimizers of \eqref{min.shape} we can prove that the support of the first $m$ normalized eigenfunction coincides with $\Omega$, up to a $(n-1)$-dimensional set.
\begin{prop}\label{supp}
  Let $\Omega\subset D$ be a shape optimizer to \eqref{min.shape}, even disconnected, and $v^i\in H^s_0(\Omega)$ be the $i$-th normalized eigenfunction on $\Omega$. Then
  $$
  \text{if $v^i\equiv 0$ on some ball $\mathcal{B}\subset \Omega$ then $v^i\equiv 0$ on $\Omega$}.
  $$
  Consequently, $\mathcal{L}_n(\Omega)=\mathcal{L}_n(\{v^i\neq 0\})$.
\end{prop}
\begin{proof}
  By the Caffarelli-Silvestre extension, the previous statement is equivalent to
   $$
  \text{if $g^i\equiv 0$ on some ball $\mathcal{B}\subset \Omega$ then $g^i\equiv 0$ on $\Omega$},
  $$
  where $g^i$ solves \eqref{equation.eigen.ext} with $\lambda = \lambda_i^s(\Omega)$. Thus, suppose that $g^i\equiv 0$ on  $\mathcal{B}_r(X_0)\subset \Omega$ for some $X_0 \in \Omega, r>0$, then by localizing the system \eqref{equation.eigen.ext} in $B_r(X_0)$ we get
$$
\begin{cases}
  L_a g^i=0 & \mbox{in } B_r(X_0) \cap \{y>0\} \\
  -\partial^a_y g=g=0 & \mbox{on }\mathcal{B}_r(X_0).
\end{cases}
$$
Thus, by the unique continuation principle for $L_a$-harmonic functions (see \cite{STT2020} and reference therein) it follows that $g^i\equiv 0$ on $\overline{B_r(X_0)}\cap \{y\geq 0\}$.\\ Finally, by exploiting the validity of the unique continuation principle for elliptic divergence form operator, we get that $g^i\equiv 0$ in $\R^{n+1}\cap \{y>\eps\}$, for every $\eps>0$. Since every eigenfunctions is H\"{o}lder continuous, by letting $\eps$ goes to zero we get that the trace $g(\cdot,0)$ is identically zero, in contradiction with the normalization of the eigenfunction.\\
Finally, since $\{g^i=0\}\cap \Omega$ has empty interior, we get the second part of the statement.
\end{proof}
Lastly, the proof of Theorem \ref{mmm0} follows by combining Proposition \ref{existence}, Proposition \ref{optimalreg}, Remark \ref{emptyinterior} and Proposition \ref{supp}.
\section{Weiss monotonicity formula}\label{weiss.sect}
In this Section we establish a Weiss type monotonicity formula in the spirit of \cite{DT2,MTV1}. As it is well known in the literature, this result implies the characterization of the possible blow-up limits at free boundary points.\\

For a vector-valued function $G \in H^{1,a}(B_1; \R^m)$, let us consider
\be
W(X_0,G,r) = \frac{1}{r^n}\mathcal{J}(G,B_r(X_0))- \frac{s}{r^{n+1}}\int_{\partial B_r(X_0)}{|y|^a\abs{G}^2\mathrm{d}\sigma}
\ee
The monotonicity result is an essential
tool for proving the regularity of the free boundary and to estimate the dimension of its singular part.
\begin{thm}\label{weiss}
Let $\sigma, r_0>0$ and suppose that $G\in H^{1,a}(B_{r_0}(X_0);\R^m)\cap C^{0,s}(B_{r_0}(X_0);\R^m)$ satisfies the localized almost-minimality condition \eqref{almost.local.new} for every $r\in (0,r_0]$. Then, if $X_0 \in F(G)$,  we get
 \be\label{weiss.form}
  \frac{d}{dr} W(X_0,G,r) \geq \frac{1}{r^{n+2}}\sum_{i=1}^m \int_{\partial^+ B^+_r(X_0)}{|y|^a\left( \langle \nabla g^i, X-X_0 \rangle - s g^i \right)^2\mathrm{d}\sigma} -2\sigma \tilde{C}[|G|]_{C^{0,s}(B_{r_0})} r^{s-1},
  \ee
  with $\tilde{C}=\frac{2\omega_n}{d_s}\sum_{i=1}^{m} \lambda_i^s(\Omega) $. Thus, there exists finite the limit $W(X_0,G,0^+)=\lim_{r\to 0^+}W(X_0,G,r)$ and for $\sigma =0$  the function $W(X_0,G,\cdot)$ is constant in $(0,+\infty)$ if and only if $G$ is $s$-homogeneous with respect to $X_0$.
\end{thm}
In order to simplify the notations, since the problem is invariant under translation on $\{y=0\}$, in the following computations we will assume $X_0=0$ and denote $W(r)=W(0,G,r)$.
\begin{lem}\label{hom.ext}
Let $\sigma, r_0>0$ and suppose that $G\in H^{1,a}(B_{r_0};\R^m)\cap C^{0,s}(B_{r_0};\R^m)$ satisfies the localized almost-minimality condition \eqref{almost.local.new} for every $r\in (0,r_0]$.\\Then, if $0 \in F(G)$, we get
\begin{align*}
\int_{B^+_r}{|y|^a \abs{\nabla G}^2\mathrm{d}X}+ \mathcal{L}_n(B^+_r(G) \cap \R^n ) \leq &\,\, \frac{r^{-a}}{n}\int_{\partial^+B^+_r}{|y|^a\left(r^{2(1-s)}\abs{\nabla_{S^{n}} G}^2 + s^2\frac{\abs{G}^2}{r^{2s}} \right)\mathrm{d}\sigma}+\\
&\,+ \frac{r}{n} \mathcal{H}^{n-1}(\partial B^+_r(G) \cap \R^n )  + 2\sigma \tilde{C}[|G|]_{C^{0,s}(B_{r_0})} r^{n+s},
\end{align*}
where $\partial B^+_r(G)= \partial B^+_r \cap \{ \abs{G}>0\}$ and $\tilde{C}=\frac{2\omega_n}{d_s}\sum_{i=1}^{m} \lambda_i^s(\Omega) $.
\end{lem}
\begin{proof}
  Let us consider now the $s$-homogeneous extension $\widetilde{G}=(\tilde{g}^1,\dots,\tilde{g}^m)$ of the trace of $G$ on $\partial B_r$, defined by
$$
\widetilde{G}(X)=\frac{\abs{X}^s}{r^s}G\left(X \frac{r}{\abs{X}}\right).
$$
Then, for every $i=1,\dots,m$ we get
$$
\abs{\nabla \tilde{g}^i}^2(X) = s^2\frac{1}{r^{2s}\abs{X}^{2(1-s)}}g^i\left(X \frac{r}{\abs{X}}\right)^2 + \frac{r^{2(1-s)}}{\abs{X}^{2(1-s)}}\abs{\nabla_{S^n} g^i}^2\left(X\frac{r}{\abs{X}} \right).
$$
Integrating over $B_r^+$ and summing for $i=1,\dots,m$, we obtain
\begin{align*}
\int_{B^+_r}{|y|^a |\nabla \widetilde{G}|^2\mathrm{d}X}
&= \int_0^r \frac{1}{\rho^{2(1-s)}}\int_{\partial^+ B^+_\rho}{|y|^a \left(s^2\frac{1}{r^{2s}}\abs{G}^2\left(X \frac{r}{\rho}\right) + r^{2(1-s)}\abs{\nabla_{S^n}G}^2\left(X\frac{r}{\rho} \right)\right)\mathrm{d}\sigma}\mathrm{d}\rho\\
&= \frac{r^{-a}}{n} \int_{\partial^+ B^+_r}{|y|^a\left( s^2\frac{\abs{G}^2}{r^{2s}}+r^{2(1-s)}\abs{\nabla_{S^{n}} G}^2 \right)\mathrm{d}\sigma},
\end{align*}
while for the measure term we have that
$$
\mathcal{L}_n(B^+_r(G) \cap \R^n ) =
\frac{r}{n} \mathcal{H}^{n-1}(\partial B^+_r(G) \cap \R^n )
$$
Finally, since $\widetilde{G}= G$ on $\partial B_r$, $\normpic{\tilde{G}}{L^\infty(\mathcal{B}_r)}=\normpic{G}{L^\infty(\mathcal{B}_r)}$ and
 $$
   \normpic{\tilde{G}-G}{L^1(\mathcal{B}_r)}\leq 2\omega_n [|G|]_{C^{0,s}(B_{r_0})} r^{n+s}.
    $$
the almost-minimality condition \eqref{almost.local.new} gives the claimed inequality.
\end{proof}
\begin{proof}[Proof of Theorem \ref{weiss}]
By the estimate of Lemma \ref{hom.ext}, we immediately get
\begin{align*}
W'(r) = &\,\frac{1}{r^n}\left( \int_{\partial^+B^+_r}{|y|^a \abs{\nabla G}^2\mathrm{d}X}+ \mathcal{H}^{n-1}(\partial B^+_r(G) \cap \R^n ) \right)+\\
&\, -\frac{n}{r^{n+1}}\left(\int_{B^+_r}{|y|^a \abs{\nabla G}^2\mathrm{d}X}+ \mathcal{L}_n(B^+_r(G) \cap \R^n )\right) +\\
&\, - \frac{2s}{r^{n+1}}\sum_{i=1}^m\int_{\partial^+ B^+_r}{|y|^a g^i\partial_r g^i\mathrm{d}\sigma} +2\frac{s^2}{r^{n+2}}\int_{\partial^+ B^+_r}{|y|^a\abs{G}^2\mathrm{d}\sigma}\\
\geq &\, \frac{1}{r^n}\sum_{i=1}^m \int_{\partial^+ B^+_r}{|y|^a \left(\abs{\partial_r g^i}^2 -2s g^i\partial_r u^i + \frac{s^2}{r^2}\abs{g^i}^2\right)\mathrm{d}\sigma } -2\sigma\tilde{C}[|G|]_{C^{0,s}(B_{r_0})}   r^{s-1}\\
= &\, \frac{1}{r^{n}}\sum_{i=1}^m \int_{\partial^+ B^+_r}{|y|^a\left( \partial_r g^i- \frac{s}{r} g^i \right)^2\mathrm{d}\sigma} -2\sigma\tilde{C}[|G|]_{C^{0,s}(B_{r_0})}  r^{s-1}.
\end{align*}
with $\tilde{C}=\frac{2\omega_n}{d_s}\sum_{i=1}^{m} \lambda_i^s(\Omega)$. Hence, we deduce that
\be\label{of}
r \mapsto W(r)+2\sigma\frac{\tilde{C}}{s}[|G|]_{C^{0,s}(B_{r_0})} r^{s}
\ee
is monotone non-decreasing for $r \in (0,1)$ and consequently there exists the limit $W(0^+)=\lim_{r\to 0^+}W(0,G,r)$. Finally, by exploiting the $C^{0,s}$ regularity of $G$ we can easily exclude the possibility that $W(0^+)=-\infty$: indeed, since $0\in F(G)$ we get
$$
W(r) \geq -\frac{s}{r^{n+1}}\int_{\partial B_r}{|y|^a\abs{G}^2\mathrm{d}\sigma} \geq - s [|G|]^2_{C^{0,s}} \int_{\partial B_1}|y|^a \mathrm{d}\sigma,
$$
for every $r>0$. Finally, if $\sigma=0$ the right hand side of \eqref{weiss.form} is non-negative, and so we deduce that $W'(r)\equiv 0$ for $r \in (0,+\infty)$ if and only if
$$
\left\langle \nabla g^i(X), \frac{X}{\abs{X}}\right\rangle =\frac{s}{\abs{X}} g^i(X)\quad \text{ in $\R^{n+1}$,}
$$
namely the components $g^i$ are $s$-homogeneous in $\R^{n+1}$.
\end{proof}
\section{Compactness and convergence of blow-up sequences}\label{blow}
This Section is dedicated to the convergence of the blow-up sequences and the analysis of the
blow-up limits, both being essential for determining the local behavior of the free boundary and for the characterization of the Regular and Singular strata.\\

Let us recall the notion of blow-up sequence associated to a local minimizer $G$ in $B_1$. Given $(X_k)_k \subset F(G)$ and $r_k \searrow 0^+$ such that $B_{r_k}(X_k)\subset B_1$, we define a blow-up sequence by
\be\label{blow.upseq}
G_{X_k,r_k}(X) = \frac{1}{r^{s}_k}G(X_k+r_k X).
\ee
Thus, the sequence $(G_{X_k,r_k})_k$ is uniformly H\"{o}lder continuous in the class $C^{0,s}$ and locally uniformly bounded in $\R^{n+1}$. Thus, by Lemma \ref{compact} we already know that, up to a subsequence, $(G_{X_k,r_k})_k$ converges locally uniformly on every compact set to a function $G_0 \in H^{1,a}_{\loc}(\R^{n+1};\R^m)\cap C^{0,s}_\loc(\R^{n+1};\R^m)$ such that, for every $R > 0$ the following properties hold
\begin{itemize}
\item $G_{X_k,r_k} \to G_0$ in $C^{0,\alpha}_\loc(\overline{B_R};\R^m)$, for every $\alpha \in (0,s)$;
  \item $G_{X_k,r_k} \rightharpoonup G_0$ weakly in $H^{1,a}(B_R;\R^m)$;
  \item $G_0$ is an almost-minimizer in $B_R$ in the sense of definition \eqref{almost.fin.new} (and also \eqref{almost.local.new}).
\end{itemize}
In order to improve the previous compactness result, we need to show that the blow-up sequence satisfies a scaled almost-minimality condition. As we show in the following Lemma, this is a direct consequence from the scaling properties of the functional.
\begin{lem}\label{scal}
Suppose that $G\in H^{1,a}(\R^{n+1};\R^m)\cap L^\infty(\R^{n+1};\R^m)$ satisfies the almost-minimality condition \eqref{almost.fin.new} for some constants $K,\eps >0$. Then, for every $X_0 \in \{y=0\}, r>0$ the function $G_{X_0,r}$ defined in \eqref{blow.upseq} satisfies:
\be\label{almost.rescaled}
  \mathcal{J}(G_{X_0,r}) \leq \mathcal{J}(\tilde{G}) + K \frac{2}{d_s}\sum_{i=1}^{m} \lambda_i^s(\Omega) \normpic{\tilde{G}-G_{X_0,r}}{L^1(\R^n)}r^s
\ee
  for every $\tilde{G} \in H^{1,a}(\R^{n+1};\R^m)$ such that $\normpic{\tilde{G}}{L^\infty(\R^n)}\leq \delta r^{-s}$ and $\normpic{G-\tilde{G}}{L^1(\R^n)}\leq \eps r^{-n-s}$.
\end{lem}
\begin{proof}
  Since the problem is translation invariance in $\{y=0\}$, suppose $X_0=0$ and denote $G_r=G_{0,r}$. Let $\tilde{G} \in H^{1,a}(\R^{n+1};\R^m)$ be such that
  $$
  \normpic{\tilde{G}}{L^\infty(\R^n)}\leq \frac{\delta}{r^s}
  \quad\mbox{and} \quad    \normpic{G-\tilde{G}}{L^1(\R^n)}\leq \frac{\eps}{r^{n+s}}.
  $$
  Consider now the function $\Phi= G_r - \tilde{G}$ and $\Phi_{1/r}(X)=r^s\Phi(X/r), \tilde{G}_{1/r}(X)= r^s \tilde{G}(X/r)$. We notice that
  $$
  \norm{\Phi_{1/r}}{L^1(\R^n)}= r^{n+s}\norm{\Phi}{L^1(\R^n)}\leq \eps \quad\mbox{and}\quad \normpic{\tilde{G}_{1/r}}{L^\infty(\R^n)}=r^s\normpic{\tilde{G}}{L^\infty(\R^n)}\leq \delta,
  $$
  which allows to consider exploit the almost-minimality of $G$ with respect to the competitor $\tilde{G}_{1/r} = G + \Phi_{1/r}$. Thus, we get
\begin{align*}
  \mathcal{J}(G_{r}) &= \frac{1}{r^n}\mathcal{J}(G)\\
   & \leq \frac{1}{r^n}\mathcal{J}(\tilde{G}_{1/r}) +  K \frac{2}{d_s}\sum_{i=1}^{m} \lambda_i^s(\Omega) \frac{\normpic{\tilde{G}_{1/r}-G}{L^1(\R^n)}}{r^n}\\
   &\leq  \mathcal{J}(\tilde{G})
    + K \frac{2}{d_s}\sum_{i=1}^{m} \lambda_i^s(\Omega)
    \normpic{\tilde{G}-G_{r}}{L^1(\R^n)}r^s,
  \end{align*}
  as we claimed.
\end{proof}
\begin{prop}\label{compact2}
Let $\Omega\subset D$ be a shape optimizer to \eqref{min.shape} and $G\in H^{1,a}_\Omega(\R^{n+1};\R^m)$
be the vector of normalized eigenfunctions on $\Omega$.
Given $(X_k)_k \subset F(G)$ and $r_k \searrow 0^+$, for every $R > 0$ the following properties hold (up to extracting a subsequence):
\begin{itemize}
  \item $G_{X_k,r_k} \to G_0$ strongly in $H^{1,a}(B_R;\R^m)$;
  \item the sequence of the characteristic functions $$\chi(\{|G_{X_k,r_k}|>0\})\to \chi(\{|G_{0}|>0\})$$
      strongly in
$L^1(\mathcal B_R)$;
  \item the sequence of the closed sets $\overline{\mathcal{B}^+_R(G_{X_k,r_k})}$ and its complement in $\R^n$, converge in the Hausdorff sense respectively to $\overline{\mathcal{B}^+_R(G_{0})}$ and $\R^n \setminus \overline{\mathcal{B}^+_R(G_{0})}$
  \item the blow-up limit $G_0$ is non-degenerate at zero, namely there exists a dimensional constant $c_0 > 0$ such that
      $$
      \sup_{\mathcal{B}_r } \abs{G_0} \geq c_0 r^{s}\quad\text{for every $r>0$}.
      $$
\end{itemize}
\end{prop}
\begin{proof}
For notational simplicity, we set  $G_k = G_{X_k,r_k}$.
  Since $\abs{G_{k}}$ converges locally uniformly to $\abs{G_0}$, we get
  $$
\chi(\{|G_{0}|>0\}) \leq \liminf_{k\to\infty}  \chi(\{|G_k|>0\})
  $$
 Now, let us prove that $G_k$ converges strongly in $H^{1,a}_\loc(\R^{n+1};\R^m)$ to $G_0$ and that the characteristic functions $\chi(\{|G_k|>0\})$ converge to $\chi(\{|G_k|>0\})$ in $L^1$. Namely, fixed a radius $R>0$, it is sufficient to prove that
  $$
  \lim_{k\to \infty}\int_{B_R}{|y|^a\abs{\nabla G_k}^2\mathrm{d}X} + \mathcal{L}_n(\mathcal{B}_r^+(G_k) )=\int_{B_R}{|y|^a \abs{\nabla G_0}^2\mathrm{d}X}+\mathcal{L}_n(\mathcal{B}_r^+(G_0)).
  $$
  Consider now $\eta \in C^\infty_c(\R^{n+1}), 0\leq \eta \leq 1$ such that $\eta \equiv 1$ on $B_R$, and the competitor $\widetilde{G}_k \in H^{1,a}(\R^{n+1};\R^m)$ defined by
  $$
  \widetilde{G}_k = \eta G_0 + (1-\eta) G_k.
  $$
  For the sake of notational simplicity, let us set:
$$
\Omega_k = \{\abs{G_k}>0 \}\cap \R^n, \quad \widetilde{\Omega}_k=\{|\widetilde{G}_k|>0 \}\cap \R^n\quad\mbox{and}\quad \Omega_0=\{\abs{G_0}>0 \}\cap \R^n.
$$
Since $\widetilde{G}_k=G_k$ on $\{\eta =0\}$, by the almost-minimality condition \eqref{almost.rescaled} applied on $G_k$, given $\delta\geq\norm{G}{L^\infty(\Omega)}$ and $K,\eps>0$ the associated constants, we get
$$
\mathcal{J}(G_{k}) \leq \mathcal{J}(\tilde{G}_k) + K \frac{2}{d_s}\sum_{i=1}^{m} \lambda_i^s(\Omega) \normpic{\tilde{G}_k-G_{k}}{L^1(\R^n)}r^s_k,
$$
and consequently
  \begin{align}\label{minim}
  \begin{aligned}
  \int_{\{\eta >0\}}{|y|^a\abs{\nabla G_k}^2\mathrm{d}X} &+ \tilde{\Lambda}\mathcal{L}_n(\Omega_k\cap \{\eta >0\})\\
  \leq & \int_{\{\eta >0\}}{|y|^a|\nabla \widetilde{G}_k|^2\mathrm{d}X} + \tilde{\Lambda}\mathcal{L}_n(\widetilde{\Omega}_k\cap \{\eta >0\}) +  K \frac{2}{d_s}\sum_{i=1}^{m} \lambda_i^s(\Omega) \norm{\tilde{G}_k-G_k}{L^1(\R^n)}r^s_k\\
  \leq &   \int_{\{\eta >0\}}{ |\nabla \widetilde{G}_k|^2\mathrm{d}X} + \tilde{\Lambda}\mathcal{L}_n(\Omega_0\cap \{\eta =1\})+\tilde{\Lambda}\mathcal{L}_n(\{0<\eta<1\})+\\
  &\,+ K\frac{2}{d_s}\sum_{i=1}^{m} \lambda_i^s(\Omega) \norm{\eta(G_0-G_k)}{L^1(\R^n)}r^s_k
  \end{aligned}
  \end{align}
 On $\{\eta >0\}$ we calculate
 \begin{align*}
     \abs{\nabla G_k}^2 -  |\nabla \widetilde{G}_k|^2= &\,\,  \abs{\nabla G_k}^2 -  |\eta \nabla G_0 + (1-\eta)\nabla G_k +  (G_0-G_k)\nabla \eta|^2\\
=   &  \,\, (1-(1-\eta)^2)\abs{\nabla G_k}^2 - \eta^2|\nabla G_0|^2 - |G_0-G_k|^2|\nabla \eta|^2+\\
&- 2 (G_0-G_k)\langle \nabla \eta, \eta \nabla G_0 + (1-\eta)\nabla G_k \rangle - 2\eta(1-\eta)\langle \nabla G_0,\nabla G_k\rangle.
 \end{align*}
Since $G_k$ converges strongly in $L^{2,a}(B_R;\R^m)$ and weakly $H^{1,a}_\loc(\R^{n+1};\R^m)$ to $G_0$, we can estimate
\begin{align*}
\limsup_{k\to \infty}&\int_{\{\eta>0\}}{|y|^a\left(  \abs{\nabla G_k}^2 -  |\nabla \widetilde{G}_k|^2\right)\mathrm{d}X}=\\
&= \limsup_{k\to \infty}\int_{ \{\eta>0\}}{ |y|^a\left( (1-(1-\eta)^2) \abs{\nabla G_k}^2 -  \eta^2|\nabla G_0|^2-2\eta(1-\eta) \langle\nabla G_0,\nabla G_k\rangle\right)\mathrm{d}X}\\
&= \limsup_{k\to \infty}\int_{ \{\eta>0\}}{ |y|^a(1-(1-\eta)^2) \left(\abs{\nabla G_k}^2 - |\nabla G_0|^2\right)\mathrm{d}X}\\
&\geq  \limsup_{k\to \infty}\int_{\{\eta=1\}}{ |y|^a\left(\abs{\nabla G_k}^2 - |\nabla G_0|^2\right)\mathrm{d}X},
\end{align*}
where in the last inequality we used that $\abs{\nabla G_k}$ weakly converges in $L^{2,a}(\{0<\eta<1\})$ to $\abs{\nabla G_0}$.

Combining this fact with inequality \eqref{minim}, we obtain
\begin{align*}
\limsup_{k\to \infty}&\left(\int_{\{\eta=1\}}{|y|^a \left(\abs{\nabla G_k}^2 - |\nabla G_0|^2\right)\mathrm{d}X} + \mathcal{L}_n(\Omega_k \cap \{ \eta=1\}) -\mathcal{L}_n(\Omega_0 \cap \{ \eta=1\})\right)\\
&\leq
\limsup_{k\to \infty}\left(\int_{\{\eta>0\}}{|y|^a \left(\abs{\nabla G_k}^2 - |\nabla \widetilde{G}_k|^2\right)\mathrm{d}X}+ \mathcal{L}_n(\Omega_k \cap \{ \eta=1\}) -\mathcal{L}_n(\Omega_0 \cap \{ \eta=1\})\right)\\
&\leq
\limsup_{k\to \infty}\mathcal{L}_n(\Omega_k \cap \{ \eta=1\}) -\mathcal{L}_n(\Omega_k \cap \{ \eta>0\}) + \mathcal{L}_n(\{0<\eta<1\})\\
&\leq \mathcal{L}_n(\{0<\eta<1\}).
\end{align*}
Finally, since $\eta$ is arbitrary outside $B_R$, the right hand side can be made arbitrarily small, and this implies the desired equality.\\

By Corollary \ref{density}, we already know that
\be\label{togh}
\eps_0 \omega_n r^n \leq \mathcal{L}_n(\mathcal{B}_r\cap \{|G_k|>0\})\leq (1-\eps_0)\omega_n r^n, \quad \text{for $r<r_0/r_k$},
\ee
and for every $k>0$. Now, it is well-known that the convergence of the sequence of characteristic functions in the strong topology of $L^1$, together with \eqref{togh}, implies the Hausdorff convergence of $\overline{\Omega_k \cap B_R}$ to $\overline{\Omega_0 \cap B_R}$ locally in $\R^n$. Obviously, the same result holds for the complements  $\Omega_k^c$ in $\R^n$.\\

Finally, the non-degeneracy of the blow-up limit is a straightforward combination of the uniform convergence and the non-degeneracy condition \eqref{non-deg}. Namely, by Proposition \ref{non-deg2}, for every $k>0$ the rescaled function $G_k$ is non-degenerate in the sense
$$
\text{for every }y \in \overline{\Omega_k}, r\leq \frac{1}{2r_k} \quad \sup_{\mathcal{B}_r(y)}\abs{G_k}\geq c_0 r^{s}.
$$
The previous inequality is obtained by applying \eqref{non-deg} in $B_{r_kr}(y)$ for the local minimizer $G$. Finally, by the uniform convergence of $G_k$ and the Hausdorff convergence of $\Omega_k \cap B_r $ in $\R^n$, for every $y\in \overline{\Omega_0}$ we get
$$
\sup_{\mathcal{B}_r(y)}\abs{G_k}\geq c_0 r^{s},\,\,\text{for every }r>0.
$$
\end{proof}
\begin{prop}\label{limit}
Let $\Omega\subset D$ be a shape optimizer to \eqref{min.shape} and $G\in H^{1,a}_\Omega(\R^{n+1};\R^m)$
be the vector of normalized eigenfunctions on $\Omega$.
Assume that $X_0 \in F(G)$, then every blow-up limit $G_0$ is a global minimizer of the functional $\mathcal{J}$. More precisely, for every $R>0$ and for every $\tilde{G}_0\in H^{1,a}_\loc(\R^{n+1};\R^m)$ such that $\tilde{G}_0-G_0 \in H^{1,a}_0(B_R;\R^m)$ we have
$$
\int_{B_R}{|y|^a\abs{\nabla G_0}^2\mathrm{d}X} + \tilde{\Lambda}\mathcal{L}_n(\mathcal{B}_R\cap \{|G_0|>0\})\leq
\int_{B_R}{|y|^a|\nabla \tilde{G}_0|^2\mathrm{d}X} + \tilde{\Lambda}\mathcal{L}_n(\mathcal{B}_R\cap \{|\tilde{G}_0|>0\}).
$$
\end{prop}
\begin{proof}
By Remark \ref{scaling}, it is not restrictive to assume $X_0=0$ and to consider $G_r=G_{0,r}$ a blow-up sequence centered at $0\in F(G)$. By Lemma \ref{scal}, for every $R>0$ it holds
\be\label{portaqua}
 \mathcal{J}(G_{r}) \leq \mathcal{J}(\tilde{G}) + K \frac{2}{d_s}\sum_{i=1}^{m} \lambda_i^s(\Omega) \normpic{\tilde{G}-G_{r}}{L^1(\R^n)}r^s
\ee
for every $\tilde{G} \in H^{1,a}(\R^{n+1};\R^m)$ such that $\tilde{G} - G_{r}\in H^{1,a}(B_R;\R^m)$ and
$$
\normpic{\tilde{G}}{L^\infty(\R^n)}\leq \delta r^{-s}\quad\mbox{and}
\quad \normpic{G-\tilde{G}}{L^1(\R^n)}\leq \eps r^{-n-s}.
$$
Let now $\tilde{G}_0\in H^{1,a}_\loc(\R^{n+1};\R^m)\cap L^\infty_\loc(\R^{n+1};\R^m)$ be such that $\tilde{G}_0-G_0 \in H^{1,a}_0(B_R;\R^m)$ and let $\eta \in C^\infty_c(B_R)$ be such that $0\leq \eta \leq 1$. Set $G_{r_k}$ the blow-up sequence associated to $r_k \searrow 0^+$ converging to $G_0$ in the sense of Lemma \ref{compact} and Proposition \ref{compact2}, and consider the new test function
$$
\tilde{G}_{r_k} = \tilde{G}_0 + (1-\eta)(G_{r_k}-G_0).
$$
Since $\tilde{G}_0 = G_0$ outside of $B_R$, we get $\tilde{G}_{r_k}=G_{r_k}$ outside of $B_R$. Moreover, since
$$
\tilde{G}_{r_k} - G_{r_k}= \tilde{G}_0  - G_0  -\eta(G_{r_k}-G_0).
$$
and $G_{r_k}\to G_0$ in $L^1(\mathcal{B}_R)$, there exists $k_0=k_0(n,s)>0$ such that
$$
\normpic{\tilde{G}_{r_k} - G_{r_k}}{L^1(\R^n)}\leq 2\normpic{\tilde{G}_0  - G_0}{L^1(\R^n)}\quad\mbox{and}\quad
\normpic{\tilde{G}_{r_k} - G_{r_k}}{L^\infty(\R^n)}\leq
2\normpic{\tilde{G}_0  - G_0}{L^\infty(\R^n)}
$$
for $k\geq k_0$. Hence, by testing \eqref{portaqua} with $\tilde{G}_{r_k}$ we deduce
\begin{align*}
\int_{B_R}|y|^a |\nabla G_{r_k}|^2\mathrm{d}X +\tilde{\Lambda}\mathcal{L}_n(\mathcal{B}^+_R(G_{r_k}))
\leq &\, \int_{B_R}|y|^a |\nabla \tilde{G}_{r_k}|^2\mathrm{d}X +\tilde{\Lambda}\mathcal{L}_n(\mathcal{B}^+_R(\tilde{G}_{0})\cap \{\eta =1\}) +\\
&\, +\tilde{\Lambda}\mathcal{L}_n(\{0<\eta<1\}\cap \R^n)+\\
&\, + K\frac{2}{d_s}\sum_{i=1}^{m} \lambda_i^s(\Omega) \normpic{\tilde{G}_{r_k}-G_{r_k}}{L^1(\R^n)}r^s_k.
\end{align*}
Since $\tilde{G}_{r_k}\to \tilde{G}_0, G_{r_k}\to G_0$ strongly in $H^{1,a}(B_R;\R^m)$, we get
\begin{align*}
\int_{B_R}|y|^a |\nabla G_{0}|^2\mathrm{d}X +\tilde{\Lambda}\mathcal{L}_n(\{\abs{G_{0}}>0\}\cap \mathcal{B}_R)
\leq &\, \int_{B_R}|y|^a |\nabla \tilde{G}_{0}|^2\mathrm{d}X +\tilde{\Lambda}\mathcal{L}_n(\{|\tilde{G}_{0}|>0\}\cap \mathcal{B}_R) +\\
&\,+ \tilde{\Lambda}\mathcal{L}_n(\{0<\eta<1\}\cap \mathcal{B}_R)
\end{align*}
Finally, by choosing $\eta$ such that $\mathcal{L}_n(\{\eta =1\}\cap \mathcal{B}_R)$ arbitrarily close to $\mathcal{L}_n(\mathcal{B}_R)$, we get the claimed inequality.
\end{proof}
The following is a straightforward application of the Weiss monotonicity formula to the previous characterization of the blow-up limits.
\begin{cor}\label{weiss.cor}
Let $\Omega\subset D$ be a shape optimizer to \eqref{min.shape} and $G\in H^{1,a}_\Omega(\R^{n+1};\R^m)$
be the vector of normalized eigenfunctions on $\Omega$. Assume that $X_0 \in F(G)$, then every blow-up limit $G_0=(g^i_0,\dots,g^m_0)$ of $G$ at $X_0$ is $s$-homogeneous in $\R^{n+1}$, that is,
  $$
  \left\langle \nabla g^i_0(X), \frac{X}{\abs{X}}\right\rangle =\frac{s}{\abs{X}} g^i_0(X) \,\text{ in $\R^{n+1}$},
$$
for every $i=1,\dots,m$. Moreover, the Lebesgue density of $F(G)$ exists finite at every $X_0 \in F(G)$ and it satisfies
\begin{align}\label{density.weiss}
\begin{aligned}
\left\{\abs{G}>0\right\}^{(\gamma)} &= \left\{X_0 \in F(G) \colon \lim_{r \to 0^+} \frac{\mathcal{L}_n(\mathcal{B}_r(X_0)\cap \{|G|>0\})}{\mathcal{L}_n(\mathcal{B}_r)}=\gamma\right\}\\ &= \left\{X_0 \in F(G) \colon W(X_0,G,0^+)=\omega_n \gamma\right\}.
\end{aligned}
\end{align}
\end{cor}
\begin{proof}
Let $X_0\in F(G)$ and $G_0$ a blow-up limit of $G$ at $X_0$ associated to a sequence $r_k \searrow 0^+$.
  By Lemma \ref{compact}, Proposition \ref{compact2} and Proposition \ref{limit}, we already know that $G_0$ is a global minimizer of $\mathcal{J}$. On the other hand, by the definition of the Weiss formula, for every $\rho,r>0$ we get
  $$
  W(X_0,G,r\rho) = W(0,G_{X_0,r},\rho).
  $$
  Fixed $R>0$, since up to a subsequence $G_{X_0,r_k}\to G_0$ uniformly and strongly in $H^{1,a}(B_R;\R^m)$, we deduce
  $$
  W(0,G_0,R) = \lim_{k\to \infty}W(0,G_{X_0,r_k},R) =
  \lim_{k\to \infty}W(X_0,G,R r_k) =
  \lim_{r \to 0^+}W(X_0,G,r),
  $$
  where the last limit is unique and it does not depend on the sequence $(r_k)_k$, thanks to the monotonicity result Theorem \ref{weiss}. Finally, since $G_0$ is a global minimizer of $\mathcal{J}$ and $R\mapsto W(0,G_0,R)$ is constant we get that the blow-up limit is $s$-homogeneous (see case $\sigma=0$ of Theorem \ref{weiss}).\\
  Moreover, the homogeneity of the blow-up limits and the strong convergence of the blow-up sequences imply
  \be\label{equiva}
  \mathcal{L}_n(\mathcal{B}_1 \cap \{|G_0|>0\}) = W(0,G_0,1) = \lim_{r\to 0^+}W(X_0,G,r)=\lim_{r\to 0^+}\frac{  \mathcal{L}_n(\mathcal{B}_r(X_0) \cap \{|G|>0\})}{r^n}.
  \ee
  Hence, the density $W(X_0,G,0^+)$ coincides, up to a multiplicative constant, with the Lebesgue density of the free boundary.
\end{proof}
\begin{rem}
Notice that, by \eqref{equiva}, the measure of the positivity set of the blow-up limit at some $X_0 \in F(G)$ in $\mathcal B_1$ does not depend on the blow-up limit itself.
\end{rem}

\begin{rem}\label{eigen.remark}
  In the classification of the blow-up limits, we will use some results related to
eigenvalues of the weighted Laplace-Beltrami operator related to the operator $L_a$. By direct computation, the operator $L_a$ can be decomposed as
$$
L_a u=\sin^{a}(\theta_n) \frac{1}{r^n}\partial_r \left(r^{n+2-a}\partial_r u\right) + \frac{1}{r^{2-a}} L_a^{S^n} u
$$
where $y=r \sin(\theta_n)$ and the Laplace-Beltrami type operator is defined as
$$
L_a^{S^n}  = \mbox{div}_{S^n}(\sin^a (\theta_n)\nabla_{S^n} ),
$$
with $\mbox{div}_{S^n}$ and $\nabla_{S^n}$ respectively the tangential divergence and gradient on $S^n$. In particular, the following results hold true.

Let $\omega \subset S^{n-1}\times\{0\}$ be an open subset of the $(n-1)$-sphere and let $\Sigma_\omega = \{ r\theta \colon \theta \in \omega, r>0\}\times \{0\}$ be the cone generated by $\omega$ in $\{y=0\}$. Then, $g$ is a $\alpha$-homogeneous solution of
    $$
\begin{cases}
-L_a g =0 &\emph{in } \R^{n+1}_+\\
\partial^a_{y}  g = 0& \emph{on } \Sigma_\omega \\
g = 0& \emph{on } \{y=0\}\setminus \Sigma_\omega,
\end{cases}
$$
    if and only if its trace $\varphi = g\lvert_{S^{n}}$ on the sphere satisfies
\be\label{lap-bel}\begin{cases}
-L_a^{S^n} \varphi = \lambda(\alpha) \varphi  &\emph{in } S^n_+\\
\partial^a_{\theta_n} \varphi = 0& \emph{on } \omega \\
\varphi = 0& \emph{on } (S^{n-1}\times\{0\})\setminus \omega,
\end{cases}
\ee
with $\lambda(\alpha)=\alpha(\alpha+n+a-1)$ the characteristic eigenvalue associated to the Section $\omega$. Moreover, both the map $\omega\mapsto \alpha(\omega)$ and  $\omega\mapsto \lambda(\alpha(\omega))$ are monotone with respect to the inclusion of spherical sets.

In particular, if  $\alpha<\min(1,2s)$, then $\varphi$ cannot change sign and it is indeed a multiple of the principle eigenvalue. Finally, for every spherical set $\omega\subset S^{n-1}$ such that $\mathcal{H}^{n-1}(S) \leq n\omega_n/2$ we have the inequality
    $$
    \lambda_1 \geq s \left(n-s \right),
    $$
and the equality is achieved if and only if, up to a rotation, $\omega = S^{n-1} \cap \{x_n>0\}$.\\

The proof of these claims uses the monotonicity of the eigenvalue with respect to the inclusion of spherical set and the P\'{o}lya-Szeg\"{o} inequality for the Schwarz symmetrization applied to the eigenvalue problem \eqref{lap-bel} (see \cite{cones, susste} for further details).

\end{rem}

The following Lemma characterizes the structure of the blow-up limits. In particular, we can prove that the norm of every blow-up limit is a global minimizer of a scalar thin one-phase problem.
\begin{prop}\label{caract}
Under the same hypotheses of Corollary \ref{weiss.cor}, every blow-up limit $G_0$ is of the form $$G_0(X)=\xi \abs{G_0}(X)\quad\mbox{where}\quad \xi\in\R^m, \abs{\xi}=1$$ and $\abs{G_0}$ is a global minimizer of the scalar thin functional
\be\label{J}
\mathcal{J}(g,B_R) = \int_{B_R}{|y|^a\abs{\nabla g}^2\mathrm{d}X} + \tilde{\Lambda}\mathcal{L}_n(\mathcal{B}_R\cap \{g>0\} ), \quad\mbox{for }R>0.
\ee
Moreover, there exists a dimensional constant $\delta \in (0, 1/2)$ such that one of the following possibilities holds:
\begin{enumerate}
  \item[1.] The Lebesgue density of $\{|G|>0\}$ at $X_0$ is $1/2$ and every blow-up limit $G_0$ is of the form
      \be\label{blow.reg}
      G_0(X)= \frac{\sqrt{\Lambda}}{\Gamma(1+s)}U(\langle x,\nu\rangle, y)\xi\quad\mbox{where}\quad\xi \in \R^m, \abs{\xi}=1, \nu\in S^{n-1}\times\{0\}.
      \ee
  \item[2.] The Lebesgue density of $\{|G|>0\}$ at $X_0$ satisfies
      \be\label{dens}
      \frac12 +\delta \leq \lim_{r \to 0} \frac{|\mathcal{B}_r(X_0)\cap \{|G|>0\})|}{|\mathcal{B}_r|} \leq 1-\delta,
      \ee
      and $|G_0|$ is a nonnegative global minimizer of \eqref{J} with singularity in zero.
\end{enumerate}
\end{prop}
\begin{proof}
 Let $X_0 \in F(G)$ and $G_0$ a blow-up limit of $G$ at $X_0$. By Corollary \ref{weiss.cor}, we already know that $G_0$ is an $s$-homogeneous global minimizer such that $$|\mathcal{B}^+_1(G_0)| = \gamma|\mathcal{B}_1|,$$
for some $\gamma\in (0,1)$ (because of the density estimates).
As in \cite[Lemma 2.4]{DT2}, by computing the first variation of the functional $\mathcal{J}(\cdot,B_R)$ with respect to a direction $\xi e^i$, with $\xi \in C^{\infty}(\{\abs{G}>0\})$, we deduce that the blow-up limit satisfies
$$
\begin{cases}
  L_a g^i_0=0 & \mbox{in } \R^{n+1}_+ \\
  -\partial^a_{y} g^i_0 =0 & \mbox{on } \{|G_0|>0\}\cap \R^n \\
  g^i_0=0 & \mbox{on }\R^n\setminus \{|G_0|>0\}.
\end{cases}
$$
Hence, in view of Remark \ref{eigen.remark} all the components are equal up to a multiplicative constant. Moreover by nondegeneracy, $G_0$ cannot be identically zero.

Thus, there exists $\xi \in \R^m$ such that $\abs{\xi}=1$ and $G_0= \xi g$, where $|G_0|=g$ and $g$ is a global minimizer of \eqref{J}. Indeed, for every $R>0$ let $\tilde{g}\in H^{1,a}_\loc(\R^n)$ be such that $\mbox{supp}(g-\tilde{g})\subseteq B_R$. Then, given the competitor $\widetilde{G}=\xi \tilde{g}$, we easily get that $\mathcal{J}(G_0,B_R)\leq \mathcal{J}(\widetilde{G},B_R)$ is equivalent to
$$
\int_{B_R}{|y|^a\abs{\nabla g}^2\mathrm{d}X} + \tilde{\Lambda} \mathcal{L}_n(\mathcal{B}_R\cap \{g>0\} ) \leq \int_{B_R}{|y|^a\abs{\nabla \tilde{g}}^2\mathrm{d}X} + \tilde{\Lambda}\mathcal{L}_n(\mathcal{B}_R\cap \{\tilde{g}>0\} ).
$$
The desired claims now follow by the known results for the scalar case (for $s=1/2$ see \cite[Proposition 5.3]{DS1}, for $s\in (0,1)$ see \cite[Section 5.2]{EKPSS}).
\end{proof}
\begin{rem}
The constant in the expression \eqref{blow.reg} is different from the one appearing in \cite{CRS}. Unfortunately, this is due to a computational mistake already noticed by \cite{DS1, XX}: we refer to \cite[Proposition 2.1.]{XX} for a detailed computation of the explicit constant in front of the one-plane solution $U$.
\end{rem}
In the same fashion of \cite{DT2}, the problem of the existence of singular points coincide with the existence of singular global minimizer of the scalar problem \eqref{J}. Indeed, by \cite{DS1, EKPSS}  we have
\be\label{n*}
n^* = \inf\{k \in \N\colon \text{there exists an $s$-homogeneous global minimizer of \eqref{J} with singularity in zero}\}\geq 3.
\ee
\begin{cor}\label{stable}
  Let $n<n^*$, then every blow-up limit $G_0$ is of the form \eqref{blow.reg}.
\end{cor}
Finally, we introduce the notion of regular and singular part of $F(G)$.
\begin{defn}\label{000}
  Let $\Omega\subset D$ be a shape optimizer to \eqref{min.shape} and $G$ be the vector of normalized eigenfunctions. Set $X_0 \in F(G)$, we say that
  \begin{itemize}
    \item $X_0$ is a regular point in $\mathrm{Reg}(F(G))$, if the Lebesgue density of $\{|G|>0\}$ at $X_0$ is $1/2$;
    \item $X_0$ is a singular point in $\mathrm{Sing}(F(G))$, if $X_0 \not\in \mathrm{Reg}(F(G))$.
  \end{itemize}
\end{defn}
We conclude the Section with two results on the strata of $F(G)$: first we initiate the analysis of the regularity of regular part of $F(G)$ by proving that $\mathrm{Reg}(F(G))$ is relatively open in $F(G)$. Then, by adapting a well-known approach in the context of Weiss-type monotonicity formulas, we estimate the Hausdorff dimension of the Singular stratum with respect to the threshold $n^*$.
\begin{cor}
  The regular part $\mathrm{Reg}(F(G))$ is an open subset of $F(G)$.
\end{cor}
\begin{proof}
  This result is deeply based on the upper semi-continuity of the function $X_0 \mapsto W(X_0,G,0^+)$. Thus, let $X_0 \in \mathrm{Reg}(F(G))$ and suppose by contradiction that there exists a sequence $(X_k)_k \in \mathrm{Sing}(F(G))$ converging to $X_0$. By Corollary \ref{weiss.cor} and Proposition \ref{caract}, we know that
  $$
  W(X_0,G,0^+)= \frac{\omega_n}{2}\quad\mbox{and}\quad W(X_k,G,0^+)\geq \omega_n\left(\frac{1}{2}+\delta\right),
  $$
    for some universal $\delta \in (0,1/2)$. On the other hand, fixed $r > 0$, the function $X \mapsto W(X,G,r)$ is continuous and so
  $$
  W(X_0,G,r)=\lim_{k \to \infty} W(X_k,G,r),
  $$
  Moreover, exploiting Theorem \ref{weiss} and the monotonicity of \eqref{of}
  $$
  W(X_0,G,r)+\sigma\overline{C}r^{s} =
  \lim_{k \to \infty} W(X_k,G,r) + \sigma\overline{C}r^{s} \geq \omega_n\left(\frac12 +\delta\right),
  $$
  for some $\overline{C}>0$ defined in Theorem \ref{weiss}. Finally, passing to the limit as $r \to 0$ we obtain $  W(X_0,G,0^+) \geq \omega_n(1/2+\delta)$, in contradiction with the assumption $X_0 \in \{|G|>0\}^{(1/2)}$.
\end{proof}
The following result on the smallness of the singular set
is nowadays standard and our approach follows essentially the strategy developed in \cite{weiss} and then generalized in \cite{MTV1} for the local counterpart of \eqref{min.shape}.
\begin{prop}\label{singular}
  The singular part $\mathrm{Sing}(F(G))$ satisfies the following estimate:
 \begin{itemize}
      \item if $n < n^*$, $\mathrm{Sing}(F(G))$ is empty;
\item if $n = n^*$, $\mathrm{Sing}(F(G))$ contains at most a finite number of isolated points;
\item if $n > n^*$, the $(n - n^*)$-dimensional Hausdorff measure of $\mathrm{Sing}(F(G))$ is locally finite in $\mathcal{B}_1$;
    \end{itemize}
    where $n^*\geq 3$ is defined in \eqref{n*}.
\end{prop}
\begin{proof}
  We split the proof in three cases:
  \begin{enumerate}
    \item if $n < n^*$, then by Corollary \ref{stable}, for every point $X_0 \in F(G)$ the blow-up limits are of the form \eqref{blow.reg} and so the Lebesgue density of $\{|G|>0\}$ at $X_0$ is $1/2$ and $F(G)=\mathrm{Reg}(F(G))$;
    \item if $n = n^*$, suppose there exists a sequence $(X_k)_k\subset \mathrm{Sing}(F(G))$ such that $X_k \to X_0\in \mathrm{Sing}(F(G))$. Set $r_k = |X_k-X_0|$ and consider the blow-up sequence
        $$
        G_k(X)=G_{X_k,r_k}(X)=\frac{1}{r_k^s}G(X_k+r_k X),
        $$
        which converges to some blow-up limit $G_0$.\\

        Case 1: if $\mathrm{Sing}(F(G_0))\setminus\{0\}\neq \emptyset$, then there is a direction $\xi \in S^{n-1}\cap F(G_0)$ such that $t\xi$ is a singular point of $F(G_0)$, for every $t>0$. Thus, by performing a blow-up analysis of $G_0$ at $\xi$ we will get a new blow-up limit $G_{00}$ such that $$
        \{t\xi\colon t\in \R\}\subset \mathrm{Sing}(F(G_{00}))
        $$
        and the whole function is invariant along the direction $\xi\in S^{n-1}$. Finally, by a standard dimension reduction argument, the restriction $|G_{00}|\lvert_{\R^{n}}$ is a global minimizer of \eqref{J} with a non-trivial singular set, in contradiction with the definition of $n^*$.\\

        Case 2: if $\mathrm{Sing}(F(G_0))\setminus\{0\}=\emptyset$, then $\xi_k=(X_0-X_k)r_k^{-1} \to \xi \in S^{n-1}\cap \mathrm{Reg}(F(G_0))$. By Corollary \ref{weiss.cor}, Proposition \ref{caract}, there exists $R>0$ such that
  $$
  W(\xi,G_{0},R)\leq  \omega_n\left(\frac{1}{2}+\frac{\delta}{4}\right) ,
  $$
    for $\delta \in (0,1/2)$ as in \eqref{dens}. By the convergence of $G_k$, we first get
     $$
  W(\xi,G_{k},R)\leq  \omega_n\left(\frac{1}{2}+\frac{\delta}{3}\right) ,
  $$
  which implies that
 $$
  W(\xi_k,G_{k},R) \leq W(\xi_0,G_{k},R)+ C(n,s)\left([|G|]^2_{C^{0,s}} +\tilde{\Lambda}n\omega_n\right)\left(\frac{|\xi_k-\xi_0|}{R}\right)^{\min\{2s,1\}},
  $$
  for $|\xi_0-\xi_k|<R$ and small enough. Thus, by choosing $k$ large enough, we get that
  $$
W(X_0,G, r_k r_0) = W(\xi_k,G_{k},r_0)\leq  \omega_n\left(\frac{1}{2}+\frac{\delta}{2}\right),
  $$
  in contradiction with the assumption $X_0 \in \mathrm{Sing}(F(G))$;
    \item if $n>n^*$, assume that for some $t>0$ we have $\mathcal{H}^{n-n^*+t}(\mathrm{Sing}(F(G)))>0$. Then, it can be proved that there exists some point $X_0 \in \mathrm{Sing}(F(G))$ and a
blow-up limit $G_0$ at $X_0$ such that $\mathcal{H}^{n-n^*+t}(\mathrm{Sing}(F(G_0)))>0$. Since $|G_0|$ is a global minimizer of \eqref{J}, this is in contradiction with the estimates of the dimension of the singular set in the scalar case (see \cite{EKPSS}).
  \end{enumerate}
\end{proof}
\section{Viscosity formulation on $\mathrm{Reg}(F(G))$}\label{visco}
In this short Section we recall some basic facts about the scalar thin one-phase free boundary problem, and we state the viscosity formulation satisfied by the normalized eigenfunctions on optimal domains near the free boundary.\\We show that the eigenfunctions are indeed viscosity solutions of thin-problem. Thus, the study of the regular part of the free boundary can be performed with the viscosity approach of \cite{DR, DSalmost, DSS, DT2} trying to reduce to the method carried out in the scalar case \cite{DSalmost} with the methodology introduced in \cite{DT2}, both for $s=1/2$.
\subsection{The viscosity problem associated to \eqref{J}} In this subsection we collect basic definitions and results for the scalar thin one-phase problem arising from the critical condition of the functional \eqref{J}. Consider
\begin{equation}\label{FB}\begin{cases}
L_a g = 0, \quad \textrm{in $B_1^+(g):= B_1 \setminus \{(x,0) : g(x,0)=0\} ,$}\\
\frac{\partial g}{\partial t^{s}}= \alpha, \quad \textrm{on $F(g):= \mathcal B_1 \cap \partial_{\R^n}\{(x,0) : g(x,0)>0\}$},
\end{cases}\end{equation}
where $\alpha>0$ and
\begin{equation}\label{limit}\dfrac{\p g}{\p t^{s}}(x_0):=\di\lim_{t \rightarrow 0^+} \frac{g(x_0+t\nu(x_0),0)} {t^{s}} , \quad \textrm{$X_0=(x_0,0) \in F(g)$},\end{equation}
with $\nu(x_0)$ the unit normal to the free boundary $F(g)$ at $x_0$ pointing toward $\mathcal{B}_1^+(g)$. For further details and proofs, we refer the reader to \cite{CRS,DR, DS1,DS2,DS3}.

First, we state the notion of viscosity solutions to \eqref{FB}, as introduced in \cite{DR}.

\begin{defn}Given $g, \varphi$ continuous, we say that $\varphi$
touches $g$ by below (resp. above) at $X_0 \in B_1$ if $g(X_0)=
\varphi(X_0),$ and
$$g(X) \geq \varphi(X) \quad (\text{resp. $g(X) \leq
\varphi(X)$}) \quad \text{in a neighborhood $O$ of $X_0$.}$$ If
this inequality is strict in $O \setminus \{X_0\}$, we say that
$\varphi$ touches $g$ strictly by below (resp. above).
\end{defn}

\begin{defn}\label{defsub} We say that $\varphi \in C(B_1)$ is a strict comparison subsolution (resp. supersolution) to \eqref{FB} if $\varphi$ is a  non-negative function in $B_1$ which is even with respect to $\{y=0\}$ and it satisfies
\begin{enumerate} \item $\varphi$ is $C^2$ and $L_a \varphi \geq 0$ (resp. $L_a \varphi \leq 0$) in $B_1^+(\varphi)$;
\item $F(\varphi)$ is $C^2$ and if $x_0 \in F(\varphi)$ we have
\be\label{expa}
\varphi(x_0+t\nu(x_0),z) = \alpha(x_0) U(t,z) + o(|(t,z)|^{s}), \quad \textrm{as $|(t,z)| \rightarrow 0^+,$}
\ee
 with $$\alpha(x_0) \geq \alpha\quad \text{(resp. $\alpha(x_0)\leq \alpha$)},$$ where $\nu(x_0)$ denotes the unit normal at $x_0$ to $F(\varphi)$ pointing toward $\mathcal{B}_1^+(\varphi);$
\item Either $\varphi$ is not harmonic in $B_1^+(\varphi)$ or $\alpha(x_0) >\alpha$ (resp. $\alpha(x_0)<\alpha$) at all $x_0 \in F(\varphi).$
\end{enumerate}
\end{defn}
Notice that if  $F(\varphi)$ is $C^2$ then any function $\varphi$ which is harmonic in $B^+_1(\varphi)$ has an asymptotic expansion at a point $X_0\in F(\varphi),$
\be\label{exp}
\varphi(x,y) = \alpha(x_0) U((x-x_0) \cdot \nu(x_0), y)  + o(|x-x_0|^{s}+y^{s}).
\ee
\begin{defn}\label{scalar}We say that $g$ is a viscosity solution to \eqref{FB} if $g$ is a  continuous non-negative function in $B_1$ which is even with respect to $\{y=0\}$ and it satisfies
\begin{enumerate} \item $L_a g = 0$ \quad in $B_1^+(g)$; \item Any (strict) comparison subsolution (resp. supersolution) cannot touch $g$ by below (resp. by above) at a point $X_0 = (x_0,0)\in F(g). $\end{enumerate}\end{defn}
Let us introduce the notion of $\eps$-domain variation first introduced in \cite{DR} for the case $s=1/2$ and then generalized in \cite{DSS}. This methodology allows to ``linearize" the problem \eqref{FB}, as long as an appropriate Harnack type inequality is established. Recently, this strategy has been adapted to the scalar thin almost-minimizers in \cite{DSalmost} and to  vectorial thin one-phase problem in \cite{DT2}.\\

We denote by $P$ the half-hyperplane $$P:= \{X \in \R^{n+1} : x_n \leq 0, y=0\}$$ and by $$L:= \{X \in \R^{n+1}: x_n=0, y=0\}.$$

Let $g$ be a  continuous non-negative function in $\overline{B}_\rho$. We define the multivalued map $\tilde g$ which associate to each $X \in \R^{n+1} \setminus P$ the set $\tilde g(X) \subset \R$ via the formula
\begin{equation}\label{deftilde} U(X) = g(X - w e_n), \quad \forall w \in \tilde g(X).\end{equation}
We write $ \tilde g(X)$ to denote any of the values in this set.

Notice that if g satisfies  \begin{equation}\label{flattilde}U(X - \eps e_n) \leq g(X) \leq U(X+\eps e_n) \quad \textrm{in $B_\rho,$ for $\eps>0$}\end{equation} then $\tilde g(X) \neq \emptyset$ for $X \in B_{\rho-\eps} \setminus P$  and $|\tilde g(X)| \leq \eps,$ thus  we can associate to $g$ a possibly multi-valued map $\tilde{g}$ defined at least on $B_{\rho-\eps} \setminus P$ and taking values in $[-\eps,\eps]$ which satisfies\begin{equation}\label{til} U(X) = g(X - \tilde g(X) e_n).\end{equation}
 Moreover if $g$ is strictly monotone along the $e_n$-direction in $B^+_\rho(g)$, then $\tilde{g}$ is also single-valued.\\
 We refer to \cite[Section 2]{DSS} and \cite[Section 3]{DR} for other basic properties of the $\eps$-domain variations.

\subsection{The viscosity formulation of the eigenvalue problem.}\label{general.linear} Consider now the vector valued thin type problem
\be \begin{cases} \label{VOPnew}
-L_a G =0 & \text{in $B_r\setminus \{y=0\}$;}\\
-\partial^a_y G = \sigma G & \text{in $\mathcal{B}_r^+(G)$;}\\
 \frac{\partial}{\partial t^{s}} |G|=\alpha & \text{on $F(G)\cap \mathcal{B}_r,$}
\end{cases}\ee
where $\alpha>0$ and $\sigma=(\sigma^1,\dots,\sigma^m)$, with $\sigma^i>0$ for $i=1,\dots,m$.

\begin{defn}\label{solution} We say that $G=(g^1,\ldots,g^m) \in C(B_r, \R^m)$ is a viscosity solution to \eqref{VOPnew} in $B_r$ if each $g^i$ is even with respect to $\{y=0\}$,
\be \label{visc1}
\begin{cases}
  -L_a g^i=0 & \mbox{in } B_r\setminus \{y=0\} \\
  -\partial^a_y g^i = \sigma^i g^i & \mbox{in }\mathcal{B}_r^+(G)
\end{cases}, \quad \forall i=1,\ldots,m,
\ee
and the free boundary condition is satisfied in the following sense. Given $X_0 \in F(G)$, and a continuous
function $\varphi$ in a neighborhood of $X_0$, then
\begin{enumerate}
\item If $\varphi$ is a strict comparison subsolution to \eqref{FB}, then for all unit directions $f$, $\langle G, f\rangle $ cannot be touched by below by $\varphi$ at $X_0.$
\item If $\varphi$ is a strict comparison supersolution to \eqref{FB}, then $\abs{G}$ cannot be touched by above by $\varphi$ at $X_0.$
\end{enumerate}
\end{defn}

In the next proposition we prove that the vector of normalized eigenfunctions associated to shape optimizers of \eqref{min.shape} are indeed viscosity solutions to \eqref{VOPnew} in $B_1$ for some specific $\sigma \in \R^m$ and $\alpha>0$.

\begin{prop}\label{regolar.visc}
Let $G\in H^{1,a}_\Omega(\R^{n+1};\R^m)$
be the vector of normalized eigenfunctions associated to a shape optimizer $\Omega$ of \eqref{min.shape}. Then, $G$ is a viscosity solution of
\be \begin{cases} \label{VOPnew.eigen}
-L_a G =0 & \text{in $\R^{n+1}\setminus \{y=0\}$;}\\
-\partial^a_y G = \lambda G & \text{in $\{|G|>0\}\cap \R^n$;}\\
 \frac{\partial}{\partial t^{s}} |G|= \frac{\sqrt{\Lambda}}{\Gamma(1+s)} & \text{on $F(G)$}
\end{cases}
\quad\text{with $\lambda=(\lambda_1^s(\Omega),\dots,\lambda_m^s(\Omega))$,}
\ee
 in the sense of Definition \ref{solution}.
\end{prop}
\begin{proof}
By \eqref{equation.eigen.ext} we already know that the first two conditions of \eqref{visc1} are satisfied. Hence, let $\varphi$ be a strict comparison subsolution to \eqref{FB}, and suppose by contradiction that there exists a unit direction $f$ in $\R^m$ such that $\langle G,f\rangle$ is touched by below by $\varphi$ at $Y_0 \in F(G)$.\\
Consider now the blow-up sequences centered in the touching point
$$
G_{k}(X) = \frac{1}{r_k^{s}}G(Y_0 + r_k X) \quad\mbox{and}\quad \varphi_{k}(X) = \frac{1}{r_k^{s}}\varphi(Y_0 + r_k X),
$$
for some sequence of radii $r_k \to 0^+$. Up to a subsequence, they converge respectively to some $G_{0}$ and $\varphi_{0}$ uniformly on every
compact set of $\R^{n+1}$. By Definition \ref{defsub}, we get, up to rotation, that
\be \label{absurd}
\varphi_{0}(X)=\alpha U\left(x_n,y\right)\quad\text{with $\alpha >\frac{\sqrt{\Lambda}}{\Gamma(1+s)}$},
\ee
On the other side, by Proposition \ref{caract} the norm $|G_{0}|$ is a $s$-homogeneous
global minimizer of the scalar thin one-phase functional \eqref{J} such that $\{|G_0|>0\}\cap \{y=0\}\supset \{y=0,x_{n}> 0\}$. By Remark \ref{eigen.remark}, we deduce that $\{|G_0|=0\}\cap \{y=0\} = P$ and consequently
\be\label{asbefore}
      G_{0}(X)=\frac{\sqrt{\Lambda}}{\Gamma(1+s)}U(x_n, y)\xi\quad\mbox{where}\quad\xi \in \R^m, \abs{\xi}=1.
      \ee
      Hence, we immediately deduce that
$$
\varphi_0(X)= \alpha U(x_n,y) \leq \langle G_0, f \rangle = \frac{\sqrt{\Lambda}}{\Gamma(1+s)}\langle \xi, f\rangle U(x_n,y),
$$
in contradiction with the hypothesis \eqref{absurd}.\\
On the other hand, let $\varphi$ be a comparison strict supersolution and let us assume that $\abs{G}$ is touched by above by $\varphi$ at some $Y_0 \in F(G)$. By the same blow-up procedure we get, up to rotation, that
$$
\varphi_{0}(X)=\alpha U\left(x_n,y\right)\quad\text{with $\alpha <\frac{\sqrt{\Lambda}}{\Gamma(1+s)}$},
$$
and that $|G_{0}|$ is a $s$-homogeneous
global minimizer of the scalar thin one-phase functional \eqref{J} such that $\{|G_0|=0\}\cap \{y=0\}\subset P$. As before, we get \eqref{asbefore} and since $\abs{G_0}\leq \varphi_0$, the absurd follows from the fact that $\alpha<\frac{\sqrt{\Lambda}}{\Gamma(1+s)}$.
\end{proof}
\begin{rem}\label{rescale.visc} We remark that if $G$ is a viscosity solution to \eqref{VOPnew} in $B_r(X_0)$ for some $\alpha>0, \lambda \in \R^m, X_0 \in F(G)$, then $$G_{X_0,r}(X) = \frac{1}{r^s}G(X_0+rX), \quad X \in B_1$$ is a viscosity solution to \eqref{VOPnew} in $B_1$ with $\alpha>0$ and $\tilde{\sigma}=r^s \sigma \in \R^m$.\\
Similarly, by Proposition \ref{regolar.visc}, given a vector $G$ of normalized eigenfunctions associated to a shape optimizer $\Omega$ and $X_0 \in F(G), r>0$, we get that $G_{X_0,r}$ is a viscosity solution to
\be\label{VOPnew.eigen.resc}
\begin{cases}
-L_a G_{X_0,r} =0 & \text{in $\R^{n+1}\setminus \{y=0\}$;}\\
-\partial^a_y G_{X_0,r} = \left(\lambda r^s \right)G_{X_0,r} & \text{in $\{|G_{X_0,r}|>0\}\cap \R^n$;}\\
 \frac{\partial}{\partial t^{s}} |G_{X_0,r}|= \frac{\sqrt{\Lambda}}{\Gamma(1+s)} & \text{on $F(G_{X_0,r})$},
\end{cases}
\ee
with $\lambda=(\lambda_1^s(\Omega),\dots,\lambda_m^s(\Omega))$. The behaviour of \eqref{VOPnew}, under translations and rescaling, reflects that for $r>0$ sufficiently small, the solutions near $F(G)$ resemble the ones with $\sigma=0$.
\end{rem}
\begin{rem}\label{sub} Notice that, if $G$ is a viscosity solution to \eqref{VOPnew} for some $\alpha>0,\sigma\in\R^m$, as in Remark \ref{Gsub}, the function
$$
|G| + \frac{|\sigma|}{1-a}|y|^{1-a}
$$
is a viscosity subsolution to the scalar thin one-phase problem \eqref{FB} in the sense of Definition \ref{scalar}. Indeed,
by the free boundary condition in Definition \ref{solution}, we easily deduce the validity of its scalar counterpart in Definition \ref{scalar}.
\end{rem}
\section{Regularity of $\mathrm{Reg}(F(G))$}\label{reg.sect}
In this Section we start by introducing the basic tools for our study of the regular part $\mathrm{Reg}(F(G))$ and then we state and prove an Harnack type inequality which is crucial for the linearization. Lastly, we conclude the Section with an improvement of flatness type lemma, from which the main result of $C^{1,\alpha}$ regularity for flat free boundary follows by standard arguments (see for example \cite{DT, DT2}).
\subsection{Flat solutions}
 In view of Definition \ref{000}, Proposition \ref{regolar.visc}, Remark \ref{rescale.visc} and the non-degeneracy property, we can reduce our analysis to understanding flat viscosity solutions defined below.
\begin{defn}
Let $\alpha=1$ and $G$ be a viscosity solution to \eqref{VOPnew} in $B_1$ for some $\sigma>0$. We say that $G$ is $\eps$-flat in the $(f,\nu)$-directions in $B_1$, if there exist some unit directions $f\in \R^m, \nu \in \R^n$,
\be \label{flat} |G(X) - U(\langle x,\nu\rangle,y) f| \leq  \eps \quad \text{in $B_1$,}
\ee
and \be\label{nondegenra}
 |G| \equiv 0 \quad \text{in $\mathcal{B}_1 \cap \{\langle x, \nu \rangle  < - \eps\}$}.\ee
\end{defn}
For the sake of simplicity we consider the notion of $\eps$-flat solution for $\alpha=1$, but, up to a multiplicative constant, this is not a restrictive assumption.

In \cite{DT2}, in collaboration with D. De Silva we introduce a two different approach in order to prove the validity of an Harnack type inequality near the free boundary in the case of vectorial problems. More precisely, in the thin case the reduction to a scalar problem just requires the construction of two appropriate barriers. In our case, we will adapt the same strategy to the case of almost-minimizer.

\begin{remark}\label{positive1}
  Let $G\in H^{1,a}_\Omega(\R^{n+1};\R^m)$
be the vector of normalized eigenfunctions on a shape optimizer $\Omega$. Then, by the minimum principle \cite[Theorem 2.8.]{brascoparini}, we already know that
  $$
  g^1 >0 \quad\text{in $\{|G|>0\}$.}
  $$
This observations is fundamental and allows to treat $g^1$ as a supersolution to the  scalar thin one-phase problem \eqref{FB} in the sense of Definition \ref{defsub}. Notice that, as in \cite[Proposition 6.2.]{DT2}, the positivity of $g^1$ can actually be proved just by assuming the flatness condition \eqref{flat}.
\end{remark}

The following result translates the flatness hypothesis on the vector-valued function $G$ into the property that its first component is trapped between nearby translation of a one-plane solution of \eqref{FB}, while the remaining ones are sufficiently small.
\begin{lem}\label{translate}
Let $G$ be a viscosity solution to \eqref{VOPnew} in $B_1$ for $\alpha=1, \sigma >0$. There exists $\eps_0>0$ universal such that, if G is $\eps_0$-flat in the $(f^1, e_n)$-directions in $B_1$, then\begin{enumerate}
\item for $i=2,\ldots,m, $ \be |g^i| \leq C \eps_0 U(X+\eps_0 e_n) \quad \text{in $B_{1/2};$}\ee
\item \be U(X-C\eps_0 e_n) \leq g^1 \leq |G| \leq U(X+C\eps_0 e_n) \quad \text{in $B_{1/2}$},\ee
\end{enumerate}
with $C>0$ universal.
\end{lem}
\begin{proof}
For the bound $(i),$ notice first that $-\partial^a_y |g^i|\leq \lambda^i |g^i|\leq \lambda^i \eps_0$ in $B_1$ and so the function
$$
h^i=|g^i| + \frac{\lambda^i \eps_0}{1-a} |y|^{1-a}
$$
is $L_a$-subharmonic in $B_1$  and it satisfies
$$
h^i \leq C\eps_0, \quad h^i \equiv 0 \quad\text{on $\{X \in \mathcal{B}_1 \colon x_n \leq -\eps_0\}$},
$$
for some dimensional constant $C>0$. Then, let $v$ be the $L_a$-harmonic function in $B_1 \setminus \{X \in \mathcal{B}_1\colon x_n <-\eps_0\}$ such that
$$
v=C\eps_0 \quad\text{on $\partial B_1$}, \quad v=0 \quad\text{on $\{X \in \mathcal{B}_1\colon x_n \leq -\eps_0\}$}.
$$
Hence, by comparison principle $h^i\leq v$ in $B_1$ and so $|g^i|\leq v$ in $B_1$. Then by the boundary Harnack inequality, say for $\bar{X} = \frac12 e_n$, we deduce
$$
v(X)\leq \bar{C}\frac{v(\bar{X})}{U(\bar{X}+\eps_0 e_n)}U(X+\eps_0 e_n) \leq C \eps_0 U(X+\eps_0 e_n) \quad\text{in $B_{1/2}$},
$$
with $C>0$ universal.\\ For the bounds in $(ii),$ we already know by Remark \ref{positive1} that $g^1$ is strictly positive and it satisfies
$$
U(X)-\eps_0 \leq g^1(X) \leq U(X)+\eps_0 \quad\text{in $B_1$},
$$
Moreover, by taking $\eps_0$ possibly smaller, we have
$$
\{X \in \mathcal{B}_1 \colon x_n\leq -\eps_0  \}\subset
\{X \in \mathcal{B}_1 \colon g^1=0 \}
\subset
\{X \in \mathcal{B}_1 \colon x_n\leq \eps_0  \}.
$$
Now, since $-\partial^a_y g^1 \geq 0$ in $\mathcal{B}_1^+(g^1)$ we get that $g^i$ is $L_a$-superharmonic in $B_1^+(g^1)$.\\ Thus, according to the proof of \cite[Lemma 5.3.]{DR} for the case $s=1/2$, we get the lower bound in $(ii)$: here details are omitted as they apply verbatim by considering harmonic function with respect to the $L_a$-operator.\\
Now, since $g^1\leq |G|$, the upper bound follows by exploiting the observations of Remark \ref{Gsub}. Thus, consider
$$
v(x,y)=|G|(x,y) + \overline{C} |y|^{1-a}, \quad\mbox{with }\overline{C}=\frac{\sqrt{k}\tilde{C}_{n,s}^{n/4s}}{1-a} \lambda^s_k(\Omega)^{\frac{n+4s}{4s}}.
$$
Let $f_1$ be the $L_a$-harmonic function in the set
$$
B_1^{\eps_0} = B_1\setminus \{X \in \mathcal{B}_1\colon x_n \leq -\eps_0\},
$$
such that $f_1=v$ on $\partial B_1$ and $f_1=0$ on $\{X\in\mathcal{B}_1\colon x_n\leq -\eps_0\}$. Since $v$ is $L_a$-subharmonic, by the maximum principle it follows that
\be\label{due}
v\leq f_1\quad\text{in $\overline{B_1}$} \longmapsto |G|\leq f_1\quad\text{in $\overline{B_1}$}.
\ee
Finally, we can prove that there exits $C>0$ universal such that
$$
f_1(X) \leq U(X+C\eps_0 e_n)\quad\text{in $B_{1/2}$},
$$
which implies the claimed result. However, since the proof of this inequality follows essentially the ideas of the one of \cite[Lemma 5.3.]{DR}, by replacing the Laplacian with the $L_a$-operator, we omit the details.
\end{proof}
\subsection{Harnack type inequality}
In the case of local minimizers of vectorial problems with thin free boundary \cite[Lemma 6.6.]{DT2}, with D. De Silva we prove an Harnack inequality by using the observation that $|G|$ and $g^1$ are respectively a subsolution and a supersolution for the scalar one phase problem in $B_1$, which means that the strategy of the scalar case applies straightforwardly also in that context (see \cite[Lemma 2.4]{DT} for the similar result for the local case).\\
In the case of normalized eigenfunctions for the fractional Laplacian, we adapt the same strategy by using that locally the eigenvalue problem \eqref{VOPnew} resembles the vectorial one-phase problem in \cite{DT2} (see Remark \ref{rescale.visc}).\\
Therefore, most details are omitted as the results of \cite{DSS} (or \cite{DR} for $s=1/2$) can be applied directly, after observing that in their proofs it is enough for the function to be either a subsolution or a supersolution of \eqref{FB} (depending on the desired bound).\\ As pointed out in \cite{DSalmost}, alternatively one could prove the validity of an Harnack inequality by specifying the size of the neighborhood around the contact point between the
solution and some explicit barriers.
\begin{thm}\label{harnack}
Let $\alpha=1$ and $G$ be a solution of \eqref{VOPnew} in $B_1$ satisfying $$
\mathcal{J}(G,B_1)\leq  \mathcal{J}(\tilde{G},B_1)+\sigma,
$$
for every $\tilde{G}-G \in H^{1,a}_0(B_1;\R^m)$, with $\sigma \leq \eps^{2(n+2)}$. Then, there exists a universal constant $\overline{\eps}>0$ such that if
\be\label{flat_2} U(X+\eps a_0 e_n) \leq g^1 \leq |G| \leq U(X+\eps b_0 e_n) \quad \text{in $B_r(X_0) \subset B_1$,}
\ee  with $$\eps(b_0 -a_0) \leq  \bar\eps r,$$
and
\be \label{smalli}|g^i|\leq r^{s}\left( \frac{b_0-a_0}{r}\eps \right)^{5/8}\quad \text{in $B_{1/2}(X_0)$, \quad i=2,\ldots, m,}\ee then
\be\label{flat_2harnack} U(X+\eps a_1 e_n)\leq g^1 \leq |G| \leq U(X+\eps b_1 e_n) \quad \text{ in $B_{\eta r}(X_0)$,}
\ee
with $$a_0 \leq a_1 \leq b_1 \leq b_0, \quad  b_1 - a_1 = (1-\eta)(b_0-a_0) ,$$ for a small universal constant $\eta>0$.
 \end{thm}

Here $\tilde{g}^1_\eps$ and $\widetilde{|G_\eps|}$ are the $\eps$-domain variations associated to $g^1$ and $|G|$ respectively and
$$
a_\eps := \left\{(X,\tilde{g}^1_\eps(X)) \colon X \in B_{1-\eps}\setminus P\right\}\quad\text{and}\quad A_\eps := \left\{(X,\widetilde{|G_\eps|}(X)) \colon X \in B_{1-\eps}\setminus P\right\}.
$$
Since domain variations may be multivalued, we mean that given $X$ all pairs $(X, \tilde g^1_\eps(X))$ and $(X,\widetilde{|G_\eps|}(X))$ belong respectively to
$a_\eps$ and  $A_\eps$, for all possible values of the $\eps$-domain variations. The following key corollary follows directly from the previous result.

\begin{cor}\label{AA}
Let $\alpha=1$ and $G$ be a solution of \eqref{VOPnew} in $B_1$ such that $$
\mathcal{J}(G,B_1)\leq  \mathcal{J}(\tilde{G},B_1)+\sigma,
$$
for every $\tilde{G}-G \in H^{1,a}_0(B_1;\R^m)$, with $\sigma \leq \eps^{2(n+2)}$. Then, there exists a universal constant $\overline{\eps}>0$ such that if
$$
U(X-\eps e_n) \leq g^1 \leq |G| \leq U(X+\eps e_n) \quad \text{in $B_1$,}
$$
and
$$|g^i|\leq \eps^{3/4}\quad \text{in $B_{1/2}$, \quad i=2,\ldots, m,}
$$
with $\eps \leq \bar{\eps}/2$ and $k_0>0$ such that
\be\label{m0}
4\eps (1-\eta)^{k_0}\eta^{-k_0} \leq \overline{\eps}, \quad \eps \leq C (1-\eta)^{5 k_0},
\ee
for some $C>0$ universal, then the sets $a_\eps \cap (B_{1/2}\times [-1,1])$ and $A_\eps \cap (B_{1/2}\times [-1,1])$ are trapped above the graph of a function $y=a_\eps(X)$ and below the graph of a function $y=b_\eps(X)$ with
$$
b_\eps - a_\eps \leq 2(1-\eta)^{k_0-1},
$$
where $a_\eps,b_\eps$ have modulus of continuity bounded by the H\"{o}lder function $\gamma t^\beta$, with $\gamma, \beta$ depending only on $\eta$.
\end{cor}
 Indeed, by iterating the Harnack inequality for $k=0, \ldots, k_0$ (the second inequality in \eqref{m0} guarantees that \eqref{smalli} is preserved), we obtain
 \be\label{serve}
 U(X+\eps a_k e_n) \leq g^1 \leq |G| \leq U(X+\eps b_k e_n)\quad \text{in $B_{\eta^k}$}
\ee
with $b_k-a_k = 2(1-\eta)^k$. Thus, by the properties of the $\eps$-domain variations (see \cite[Lemma 3.1]{DSS} and \cite[Lemma 3.1]{DR} for $s=1/2$) we get
$$
a_k \leq \tilde{g}^1_\eps \leq \widetilde{|G_\eps|} \leq b_k\quad \text{in $B_{\eta^k - \eps},$}
$$
and
\begin{align*}
a_\eps \cap (B_{\eta^k - \eps} \times [-1,1]) &\subset B_{\eta^k - \eps} \times [a_k,b_k],\\
A_\eps \cap (B_{\eta^k - \eps} \times [-1,1]) &\subset B_{ \eta^k - \eps} \times [a_k,b_k],
\end{align*}
for $k=0, \ldots, k_0.$
\begin{lem}\label{fbharnack}
Let $\alpha=1$ and $G$ be a solution to \eqref{VOPnew} in $B_1$ for some $\sigma \in \R^m$. There exists $\eps_0>0$ universal such that, if for $0<\eps\leq \eps_0$
$$
U(X)\leq g^1(X)  \quad\text{in $B_{1/2},$}
$$
and at $\overline{X}=(\bar{x},\bar{y}) \in B_{1/8}(\frac14 e_n)$ we have $
U(\overline{X} + \eps e_n)\leq g^1(\overline{X}),
$
then $$
U(X+\tau \eps e_n)\leq g^1(X) \quad \text{in $B_\delta$},
$$
for universal constants $\tau,\delta >0$. On the other hand, by taking $\sigma\leq \eps^{2(n+2)}$, if
$$
|G|(X) \leq U(X)\quad\text{in $B_{1/2}$}
$$
and
\be\label{vismall}|g^i|\leq \eps^{5/8} \quad \text{in $B_{1/2}$}, \quad i=2,\ldots, m,\ee
then if $g^1(\overline{X})\leq U(\overline{X} - \eps e_n)$
we get $$|G|(X) \leq U(X-\tau \eps e_n)\quad\text{ in $B_\delta,$}$$
for some universal $\delta>0$.
\end{lem}

\begin{proof}[Proof of Lemma \ref{fbharnack}] The first statement follows immediately from the fact that $g^1$ is a supersolution to \eqref{FB} hence we can apply \cite[Lemma 4.3]{DSS} (see \cite[Lemma 6.3]{DR} for $s=1/2$).

  Now, let us consider the case
  $$
|G|(X) \leq U(X),\quad|g^i|\leq \eps^{5/8}\,\,\mbox{ for } i=2,\ldots, m,
$$
in $B_{1/2}$. Since $|\sigma|\leq \eps^{2(n+2)}$, by Remark \ref{sub} the function $|G| + C \eps^{2(n+2)} |y|^{1-a}$ is a subsolution and in order to apply again \cite[Lemma 4.3]{DSS}, we need to check that $$|G|(\bar X) + C \eps^{2(n+2)} |\bar{y}|^{1-a}\leq U(\bar X - c\eps e_n)$$ for some $c>0$ universal.
Since $g^1(\overline{X})\leq U(\overline{X} - \eps e_n)$, we get \be\label{dr}
    g^1(\overline{X})-U(\overline{X})\leq   U(\overline{X}-\eps e_n)-U(\overline{X}) = -\partial_t U(\overline{X}-\lambda e_n)\eps \leq -c\eps,\, \lambda\in (0,\eps)
  \ee
  and
  $$
    |G|(\overline{X})+ C \eps^{2(n-2)} |\bar{y}|^{1-a}-U(\overline{X})\leq g^1(\overline{X}) + 2C\eps^{5/4} -U(\overline{X})\leq -\frac{c}{2}\eps.
  $$ The desired bound follows arguing as in \eqref{dr}.
      \end{proof}
  We are now ready to sketch the proof of the Harnack inequality.
  \begin{proof}[Proof of Theorem \ref{harnack}]
  Without loss of generality, let us assume $a_0=-1$ and $b_0 = 1$. Also, up to rescaling, we can take $r=1$ and so $2\eps \leq \overline{\eps}$. Moreover, we denote with $\eps_0$ and $\delta$ the universal constants in Lemma \ref{fbharnack}, and choose $\bar \eps=\eps_0$, with $\eps\leq \bar \eps, \sigma \leq \eps^{2(n+2)}$.\\
    Now, we distinguish two cases depending on the position of $B_r(X_0)$.\\

    {\rm Case 1: }If $\mathrm{dist}(X_0,\{x_n=-\eps, y=0\})\leq \delta/2$ we aim to apply Lemma \ref{fbharnack}.
    Assume that for $\overline{X}=1/4 e_n$
    $$
    g^1(\overline{X})\leq U(\overline{X}).
    $$
Since,
     $$
    g^1 \leq |G| \leq U(X+\eps e_n) \quad \text{in $B_{1/2}(-\eps e_n) \subset B_1(X_0)$,}
    $$
 and  for $\eps$ small enough, it holds $\overline{X} \in B_{1/8}((-\eps + 1/4)e_n)$, by \eqref{smalli} we can apply Lemma \ref{fbharnack} in $B_{1/2}(-\eps e_n)$ and deduce that
    $$
    g^1 \leq |G| \leq U(X+(1-\eta)\eps e_n) \quad \text{in $B_{\delta}(-\eps e_n)$.}
    $$
    Finally, the improvement follows by choosing $\eta < \delta/2$, which implies that $B_\eta(X_0) \subset B_\delta(-\eps e_n)$. If instead $g^1(\overline{X})\geq U(\overline{X})$, we can proceed analogously as in the scalar result \cite[Theorem 6.1]{DR}.\\

    {\rm Case 2.}   If $\mathrm{dist}(X_0,\{x_n=-\eps, y=0\}) > \delta/2$, then $g^1>0$ and $L_a$-superharmonic in $B_1(X_0)\cap \{g^1>0\}$. Let $v^1$ be the $L_a$-harmonic replacement of $g^1$ in $B_{7/8}(X_0)$. Exploiting Lemma \ref{harmonic.rep} and the choice $\sigma \leq \eps^{2(n+2)}$ we get
    $$
    \normpic{g^1-v^1}{L^\infty(B_{1/2}(X_0))}\leq C \eps^2.
    $$
    Thus,
    $$
    U(X-\eps e_n) -C \eps^2\leq v^1 \leq U(X+\eps e_n) +C \eps^2 \quad\mbox{in }B_{1/2}(X_0),
    $$
    which implies that  $U(X+\eps a_0 e_n) \leq v^1 \leq U(X+\eps b_0 e_n)$ in $B_{1/2}(X_0)$ for some $a_0<b_0$.\\ Then, by applying directly \cite[Lemma 4.3]{DSS} (or \cite[Theorem 6.1]{DR} for $s=1/2$), as in this case we only need that $v^1$ is a positive $L_a$-harmonic function in $B_1^+(v^1)$. Thus, going back to $g^1$, the conclusion
    \be\label{flat_2harnack2} U(X+\eps a_1 e_n)\leq g^1 \leq U(X+\eps b_1 e_n) \quad \text{ in $B_{\eta}(X_0)$,}\ee does hold for $\eta$ small.
    On the other hand, reasoning as in Lemma \ref{translate}-(i) we have in the same ball,
    $$|G| \leq U(X+\eps b_1 e_n) + C \eps^{5/8}U(X+\eps e_n) \leq U(X+\bar b_1\eps e_n),$$
    and our claim is proved.
      \end{proof}
\subsection{The improvement of flatness lemma}\label{final}
We can finally conclude the Section with an improvement of flatness lemma, from which the main result of $C^{1,\alpha}$ regularity of a flat free boundary follows by standard arguments (see for example \cite{DT}). Lastly, we give a straightforward proof of Theorem \ref{mmm} by applying the results for flat-solutions of \eqref{VOPnew}.\\
In view of Lemma \ref{translate}, the flatness can be expressed as in \eqref{flat1}-\eqref{non_d1}.
\begin{lem}\label{IMPF}
Let $\alpha=1$ and $G$ be a viscosity solution of \eqref{VOPnew} in $B_1$ such that $$
\mathcal{J}(G,B_1)\leq  \mathcal{J}(\tilde{G},B_1)+\sigma,
$$
for every $\tilde{G}-G \in H^{1,a}_0(B_1;\R^m)$, with $\sigma \leq \eps^{2(n+2)}$. Suppose that $0 \in F(G)$ and 
\be\label{flat1}  U(X-\eps e_n) \leq g^1 \leq |G| \leq U(X+\eps e_n)  \quad\text{in $B_1$},
\ee
with
\be\label{non_d1} |G-g^1f^1| \leq \eps^{3/4} \quad \text{in $B_1$}.\ee
Then, there exists a universal $\rho_0>0$ such that if  $\rho \in (0,\rho_0]$ and $\eps \in (0,\eps_0]$ for some $\eps_0=\eps_0(\rho)$, then for unit vectors $\nu \in \R^n$ and $f \in\R^m$,
\be\label{flat_imp}
 U\left(\langle x, \nu\rangle -\frac{\eps}{2}\rho, y\right) \leq \langle G ,f \rangle  \leq |G| \leq U\left(\langle x, \nu\rangle +\frac{\eps}{2}\rho, y\right)\quad\text{in $B_\rho$},
\ee
and
\be\label{non_dimpr} |G - \langle G ,f \rangle f| \leq \left(\frac \eps 2\right)^{3/4} \rho^{s} \quad \text{in $B_\rho$},\ee
with $|\nu -e_n|, |f-f^1 |\leq C\eps$, for some universal constant $C>0.$ 
\end{lem}
\begin{proof}
Following the strategies of \cite{DT, DT2, DR} 
, we proceed with a contradiction argument.\\

{\rm Step 1 - Compactness and linearization.} Set $\rho \leq \rho_0$ to be made precise later. Let us suppose there exist $\eps_k \to 0, \sigma_k \leq \eps_k^{2(n+2)}$ and a sequence of solutions $(G_k)_k$ of \eqref{VOPnew} such that $0 \in F(G_k)$ and
\be\label{contrad1}
U(X-\eps_k e_n) \leq g^1_k \leq |G_k| \leq U(X+\eps_k e_n) \quad \text{in $B_{1}$,}
\ee
and
\be\label{contrad2}
|G_k - g^1_kf^1| \leq  \eps_k^{3/4} \quad \text{in $B_{1},$}\ee but  either of the conclusions \eqref{flat_imp} or \eqref{non_dimpr} does not hold.
  Denote with $\tilde{g}^1_k$ and $\widetilde{|G|_k}$ the $\eps_k$-domain variations associated to $g^1_k$ and $|G_k|$ respectively.
In view of \eqref{contrad1}-\eqref{contrad2}, we can apply Corollary \ref{AA} and Ascoli-Arzel\`{a} theorem to conclude that, up to a subsequence, the sets
$$
a_k := \left\{(X,\tilde{g}^1_k(X)) \colon X \in B_{1-\eps_k}\setminus P\right\}\quad\text{and}\quad A_k := \left\{(X,\widetilde{|G_k|}(X)) \colon X \in B_{1-\eps_k}\setminus P\right\},
$$
converge uniformly (in Hausdorff distance) in $B_{1/2}\setminus P$ to the graphs
$$
a_{\infty} := \left\{(X,\tilde{g}^1_{\infty}(X)) \colon X \in B_{1/2}\setminus P\right\}\quad\text{and}\quad A_{\infty} := \left\{(X,\widetilde{|G_{\infty}|}(X)) \colon X \in B_{1/2}\setminus P\right\},
$$
with $\tilde{g}^1_{\infty}$ and $\widetilde{|G_{\infty}|}$ H\"{o}lder continuous functions in $B_{1/2}$. Moreover,
\be\label{equal}\widetilde{|G_\infty|} \equiv \tilde{g}_\infty^1 \quad \text{in $B_{1/2}.$}\ee

Since $g^1_k$ is a sequence of supersolutions to the scalar one-phase problem \eqref{FB}, while $|G_k|+C\sigma_k |y|^{1-a}$ is a sequence of subsolutions to the same problem, we conclude by the arguments in \cite[Lemma 5.5.]{DSS} (see also the proof of Step 2 of \cite[Lemma 7.4.]{DR} for the case $s=1/2$) that $\widetilde{|G_\infty|} \equiv \tilde{g}_\infty^1$ satisfies (in the viscosity sense) the linearized problem
\be\begin{cases}\label{linearized}
 L_a(U_n w)=0 & \text{in $B_1\setminus P,$}\\
|\nabla_r w|=0 &\text{on $B_1\cap L$},
\end{cases}\ee
where $$
|\nabla_r w|(X_0)= \lim_{(x_n,y)\to (0,0)} \frac{w(x_0',x_n,y)-w(x_0',0,0)}{r},\quad r=|(x_n,y)|,\,X_0=(x_0',0,0).
$$  In particular, since $(\tilde{g}^1_k)_k$ and $(\widetilde{|{G}_k|})_k$ are uniformly bounded in $B_1$, we get a uniform bound on $\tilde g^1_\infty \equiv \widetilde{|{G}_\infty|}$,
hence by \cite[Theorem 6.2.]{DSS} (see \cite[Lemma 4.2]{DR} for $s=1/2$), since $\tilde{g}^1_\infty(0)=0,$ we deduce that for $C_0$ universal,
\be\label{old}
     \abs{\tilde{g}^1_\infty(X) -\langle \nu', x'\rangle}\leq C_0\rho^{1+\gamma} \quad\text{in $B_{2\rho}$},
\ee
  for some vector $\nu' \in \R^{n-1}$ and $\gamma>0$ universal constant. Details are omitted as we reduced to the scalar case, hence the arguments of \cite[Theorem 5.1.]{DSS} (see \cite[Theorem 7.1.]{DR} for $s=1/2$) apply directly.\\

\smallskip

{\rm Step 2 - Improvement of flatness.} In view of \eqref{old} (see also \cite[Corollary 3.3.]{DSS} and \cite[Theorem 8.2.]{DR}),
for $\rho< 1/(8C_0)$ small enough, we get
  $$
   \langle \nu', x'\rangle - \frac18\rho\leq \tilde{g}^1_\infty(X) \leq \langle \nu', x'\rangle +\frac18\rho \quad\text{in $B_{2\rho}$}
  $$
  and, for $k$ sufficiently large, we deduce from the uniform convergence of $a_k$ to $a_\infty$ and of $A_k$ to $A_\infty$ that
\be \label{k}
   \langle \nu', x'\rangle - \frac14\rho\leq \tilde{g}^1_k(X) \leq \widetilde{|G_k|}(X)\leq \langle \nu', x'\rangle +\frac14\rho \quad\text{in $B_{2\rho} \setminus P$}.
   \ee
The argument of the proof of \cite[Lemma 7.2.]{DR} then gives 
\be\label{flat_imp2}
 U\left(\langle x, \nu\rangle -\frac{\eps_k}{4}\rho, y\right) \leq g^1_k \leq |G_k| \leq U\left(\langle x, \nu\rangle +\frac{\eps_k}{4}\rho, y\right)\quad\text{in $B_{\frac 3 2 \rho}$},
\ee
for a unit vector $\nu$ with $|\nu-e_n| \leq C\eps_k.$

On the other hand, by \eqref{contrad1}-\eqref{contrad2} we conclude that, up to a subsequence, $g^i_k/\eps_k^{3/4} \to g^i_*$ uniformly to some $g^i_*$ such that
$$
\begin{cases}
  L_a g^i_*=0 & \mbox{in } B_{1/2}\setminus P\\
  g^i_* = 0 & \mbox{on } P\cap B_{1/2}.
\end{cases}
$$
for every $i=2, \ldots, m$. Thus, for $k>0$ sufficiently large 
$$|g^i_k - M_i U \eps_k^{3/4}| \leq C \eps_k^{3/4} \rho^\gamma U \quad \text{in $B_{\frac 3 2 \rho}$},$$
with $\gamma, C>0$ universal and $|M_i| \leq M$ universal.
From the properties of the function $U$ and \eqref{contrad1}, we deduce that
\be\label{small}|g^i_k - M_i g^1_k \eps_k^{3/4}| \leq C\eps_k^{3/4}(\rho^\gamma U+ \eps_k^{s})\leq \left(\frac{\eps_k}{8}\right)^{3/4} \rho^{s} \quad \text{in $B_{\frac 3 2\rho}$},\ee by choosing $\rho\leq \rho_0$ small enough and then $k$ sufficiently large.

Now, inspired by the construction of Step 2 of \cite[Lemma 7.1.]{DT2}, set
$$\xi_k^1:= f^1 + \eps_k^{3/4} \sum_{i \neq 1} M_i f^i, \quad \bar f^1_k: = \frac{\xi^1_k}{|\xi_k^1|}.$$ Notice that,
\be\label{den}\bar f^1_k = \xi_k^1 + O(\eps_k^{3/2}).\ee
We claim that
\be\label{xi1} U\left(\langle x, \nu\rangle -\frac{\eps_k}{2}\rho, y\right) \leq \langle G_k , \bar f^1_k\rangle  \leq |G_k| \leq U\left(\langle x, \nu\rangle +\frac{\eps_k}{2}\rho, y\right)\quad\text{in $B_\rho$},\ee
and
\be\label{xi2} |G_k - \langle G_k , \bar f^1_k\rangle \bar f^1_k| \leq  \left(\frac{\eps_k}{2}\right)^{3/4} \rho^{s} \quad \text{in $B_\rho$}, \ee
thus reaching a contradiction.
First, the upper bound in \eqref{xi1} is a direct consequence of \eqref{flat_imp2}. For the lower bound, we observe that by \eqref{contrad1}, \eqref{contrad2} and \eqref{den},
$$|G_k -  U \bar f^1_k| \to 0, \quad \text{as $k\to \infty$},$$
while
\be\label{fbi}|G_k| \equiv 0 \quad \text{in $\mathcal{B}_{1/2}\cap\{x_n \leq -\eps_k\}$}.\ee
As pointed out in Remark \ref{positive1}, we know that
\be\label{pos}\langle G_k ,\bar f^1_k\rangle >0 \quad \text{in $B^+_{\frac 1 2}(G_k)$},\ee
for $k$ sufficiently large. Moreover, by the definition of $\bar f^1_k$, \eqref{den} and \eqref{pos},
\be\label{bou}\langle G_k ,\bar f^1_k \rangle \geq  \left(U\left(\langle x, \nu\rangle -\frac{\eps_k}{4}\rho, y\right) -C \eps_k^{3/2}\right)^+ \quad\text{in $B_{\frac32 \rho}$}.\ee
Set
$$h(X):=\left(U\left(\langle x, \nu\rangle -\frac{\eps_k}{4}\rho, y\right) -C \eps_k^{3/2}\right)^+$$
and call with $H$ the $L_a$-harmonic function in $B_{\frac 3 2 \rho} \setminus \{\langle (x,0), \nu \rangle \leq \frac{\eps_k}{4}\rho\}$ satisfying
$$H=h \quad \text{on $\p B_{\frac 3 2 \rho}$}, \quad H=0 \quad \text{on $\mathcal{B}_{\frac32 \rho}\cap\{\langle x, \nu \rangle = \frac{\eps_k}{4}\rho\}$.}$$ Then, by \eqref{flat_imp2}-\eqref{pos}-\eqref{bou} and the comparison principle, we conclude that $$\langle G_k , \bar f^1_k \rangle \geq H \quad \text{in $B_{\frac 3 2 \rho}$}.$$
On the other hand, by applying the Boundary Harnack
$$H \geq (1-C\eps_k^{3/2})U\left(\langle x, \nu\rangle -\frac{\eps_k}{4}\rho, y\right) \quad \text{on $B_{\rho},$}$$
for $C>0$ universal, from which the required lower bound follows for $k$ sufficiently large.

We are left with the proof of \eqref{xi2}. By the definition of $\bar f^1_k$ and \eqref{den}, it is sufficient to show that
$$ |G_k - \langle G_k , \xi^1_k\rangle \xi^1_k| \leq  \left(\frac{\eps_k}{4}\right)^{3/4} \rho^{s} \quad \text{in $B_\rho$}.$$ Set $\bar G_k:=G_k - \langle G_k , \xi^1_k\rangle \xi^1_k$, then by \eqref{contrad2} we immediately get
$$|\bar g_k^1|= \eps^{3/4} \bigg|\sum_{i \neq 1} M_i g^i_k\bigg| \leq C \eps_k^{3/2}.$$
Instead, for the remaining components, by using \eqref{small} we have
$$|\bar g^i_k| = \bigg|g_k^i - \eps_k^{3/4}M_i g^1_k- \eps_k^{3/2} M_i\sum_{j\neq 1}M_j g^j_k\bigg| \leq \left(\frac{\eps_k}{8}\right)^{3/4} \rho^{s} + C \eps_k^{9/4},$$
and the desired bound follows for $k$ large enough.
\end{proof}
Finally, we are able to conclude the Section with the proof of the Theorem \ref{mmm}.
\begin{proof}[Proof of Theorem \ref{mmm}]
The statements about $n^*$ and the fact that  $\{|G|>0\} \cap \{y=0\}$ has locally finite perimeter follow exactly as in the scalar case (see \cite[Section 5]{DS1} and \cite[Theorem 1.2.]{EKPSS}). Moreover, the estimate on the Hausdorff dimension of $\mathrm{Sing}(F(G))$ follows combining Theorem \ref{mmm0}, Proposition \ref{caract}, Definition \ref{000}, Proposition \ref{singular}.\\Lastly, let us deal with the regularity result for  $\mathrm{Reg}(F(G))$.
First, by Proposition \ref{regolar.visc}, the vector of normalized eigenfunctions associated to a shape optimizer of \eqref{min.shape} is a viscosity solution of \eqref{VOPnew.eigen}.\\
Fix $\lambda=(\lambda_1^s(\Omega),\dots,\lambda_m^s(\Omega)$, then by Remark \ref{rescale.visc}, for every $X_0 \in \mathrm{Reg}(F(G)), r>0$ the function
$$
G_2(X)=\frac{\Gamma(1+s)}{\sqrt{\Lambda}r^s}G(X_0 +r X)
$$
is a viscosity solution to \eqref{VOPnew} with $\alpha=1, \sigma= r^s \lambda, 0\in F(G_2)$. Moreover, by Proposition \ref{caract} and Definition \ref{000}, given $\eps>0$ there exists $r>0$ such that $G_2$ is $\eps$-flat solution of \eqref{VOPnew} in $B_1$, in some directions $(f,\nu)$ with $\alpha=1$ and $\sigma \leq \eps^{2(n+2)}$.\\
Then, by a standard iteration of Lemma 3.3., we get that $\mathrm{Reg}(F(G_2)) \in C^{1,a}$ in $B_{1/2}$.
\end{proof}
\bibliography{bibliography}
\bibliographystyle{abbrv}
\end{document}